\documentclass[11pt, oneside]{amsart}
\usepackage{geometry}                		
\usepackage[justification=centering,font=small,labelfont=bf]{caption}
\usepackage{amssymb}

\usepackage{graphicx}
\usepackage{bm}
\usepackage[colorlinks,linkcolor=blue,citecolor=blue,urlcolor=black]{hyperref}
\usepackage{xcolor}

\theoremstyle{plain}
\newtheorem{thm}{Theorem}[section]
\newtheorem{lem}[thm]{Lemma}
\newtheorem{cor}[thm]{Corollary}
\newtheorem{prop}[thm]{Proposition}

\theoremstyle{remark}
\newtheorem{remark}{Remark}

\theoremstyle{definition}
\newtheorem{defn}[thm]{Definition}

\numberwithin{equation}{section}

\newcommand{\E}{\mathbb{E}}
\newcommand{\C}{\mathbb{C}}

\newcommand{\Z}{\mathbb{Z}}
\newcommand\ttp{\mathtt{p}}

\newcommand{\nn}{\ensuremath{\mathbb{N}}}
\newcommand{\cc}{\ensuremath{\mathbb{C}}}

\newcommand{\Epsilon}{\ensuremath{\mathcal{E}}}

\newcommand{\cd}{\mathcal{D}}

\newcommand{\Leg}[2]{\left( \frac{#1}{#2} \right)}

\newcommand{\bigO}{\textnormal{O}}
\newcommand\OO{\textnormal{O}}

\newcommand{\abs}[1]{{\left| #1\right|}}

\newcommand\mRe{\mathop{\mathrm{Re}}}
\newcommand\mIm{\mathop{\mathrm{Im}}}
\newcommand\ord{\mathop{\mathrm{ord}}}
\newcommand\lcm{\mathop{\mathrm{lcm}}}
\newcommand\sgn{\mathop{\mathrm{sgn}}}
\newcommand\Res{\mathop{\mathrm{Res}}}

\renewcommand\geq{\geqslant}

\renewcommand\leq{\leqslant}

\newcommand\sumast{\mathop{\sum\nolimits^{\ast}}}
\newcommand\prodast{\mathop{\prod\nolimits^{\ast}}}

\newcommand\dd{\mathrm{d}}

\newcommand\exmid{\mathrel{\|}}
\newcommand\sumdy{\mathop{\sum\nolimits^{\mathrm{dy}}}}

\newcommand\wdt{0.23\textwidth}

\title{The shape of quadratic Gauss paths}

\author{Justine Dell}
\address{University of California San Diego (UCSD), Department of Mathematics, 9500 Gilman Drive \#0112, La Jolla, CA 92093, USA}
\email{jsdell@ucsd.edu}

\author{Djordje Mili\'cevi\'c}
\address{Bryn Mawr College, Department of Mathematics, 101 North Merion Avenue, Bryn Mawr, PA 19010, USA}
\curraddr{Institute for Advanced Study, 1 Einstein Drive, Princeton, NJ 08540, USA}
\email{dmilicevic@brynmawr.edu}

\thanks{Research supported in part by the National Science Foundation Grant DMS-1903301, the Simons Foundation Award MPS-TSM-00008085, and by the Charles Simonyi Endowment (D.M.).}

\begin{document}

\begin{abstract}
We consider the distribution of quadratic Gauss paths, polygonal paths joining partial sums of quadratic Gauss sums to square-free fundamental discriminant moduli in a dyadic range $[Q,2Q]$. We prove that this striking ensemble converges in law, as $Q\to\infty$, to a random Fourier series we explicitly describe, and we prove a convergence in probability result and a classification result for the limiting shapes that explain the visually remarkable properties of these Gauss paths.
\end{abstract}

\maketitle
\setcounter{tocdepth}{1}
\begin{center}
\begin{minipage}{0.75\textwidth}
\tableofcontents
\end{minipage}
\end{center}

\section{Introduction}
\label{intro}

\subsection{Gauss paths}
Cancellation in exponential sums is a major player across analytic number theory. A fascinating insight into its chaotic formation is provided by polygonal paths joining the consecutive partial sums, which have been studied at least since the work of Lehmer~\cite{Lehmer1976} and Loxton~\cite{Loxton1983,Loxton1985} on exponential sums with quadratic and other analytically defined smoothly varying phases. The pioneering work of Kowalski--Sawin~\cite{KowalskiSawin2016} and subsequent papers including \cite{RicottaRoyer2018,RicottaRoyerShparlinski2020,MilicevicZhang2023,Hussain2022,HussainLamzouri2024} extended this paradigm and studied the distribution of paths arising from oscillatory sums of arithmetic origin, such as Kloosterman and character sums.

\begin{figure}[ht]
    \centering
    \includegraphics[height=0.3\textwidth]{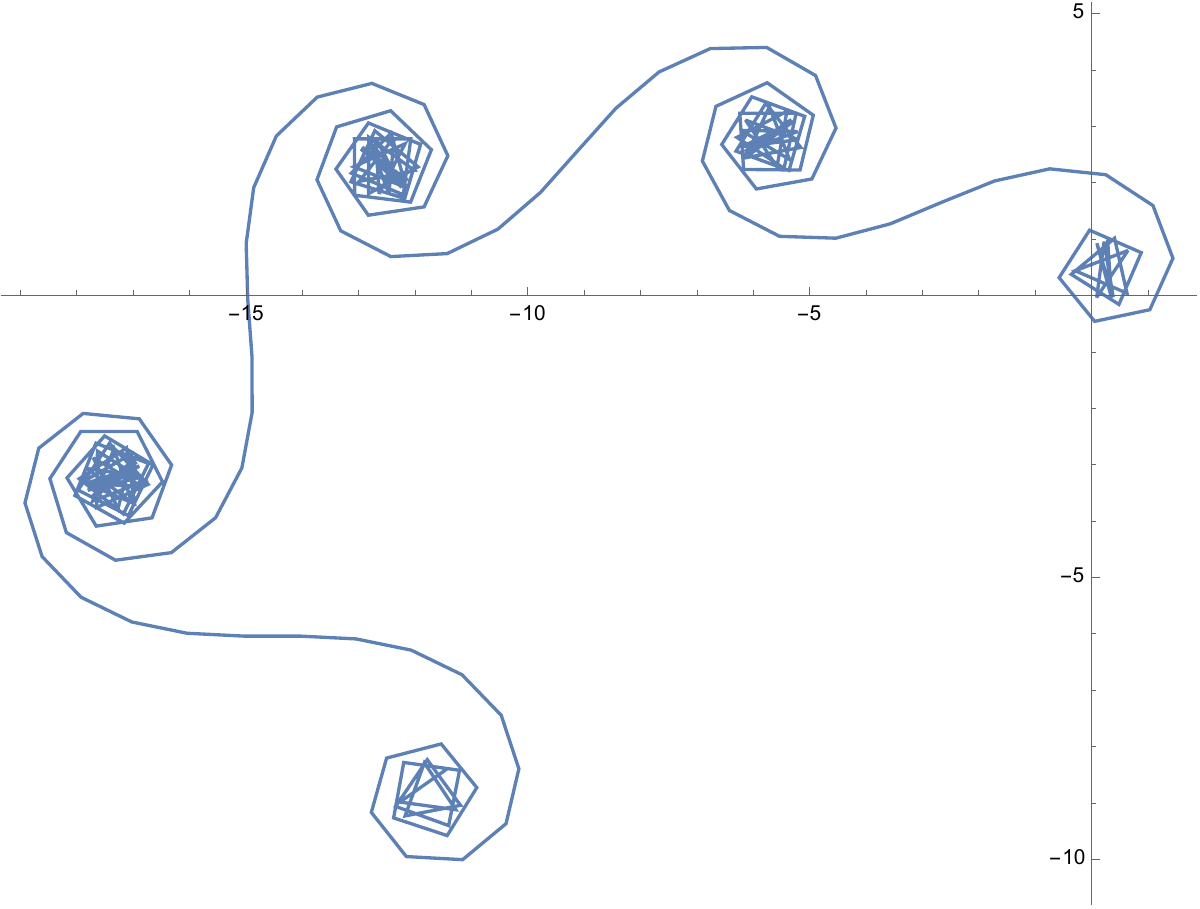}
\qquad
    \includegraphics[height=0.3\textwidth]{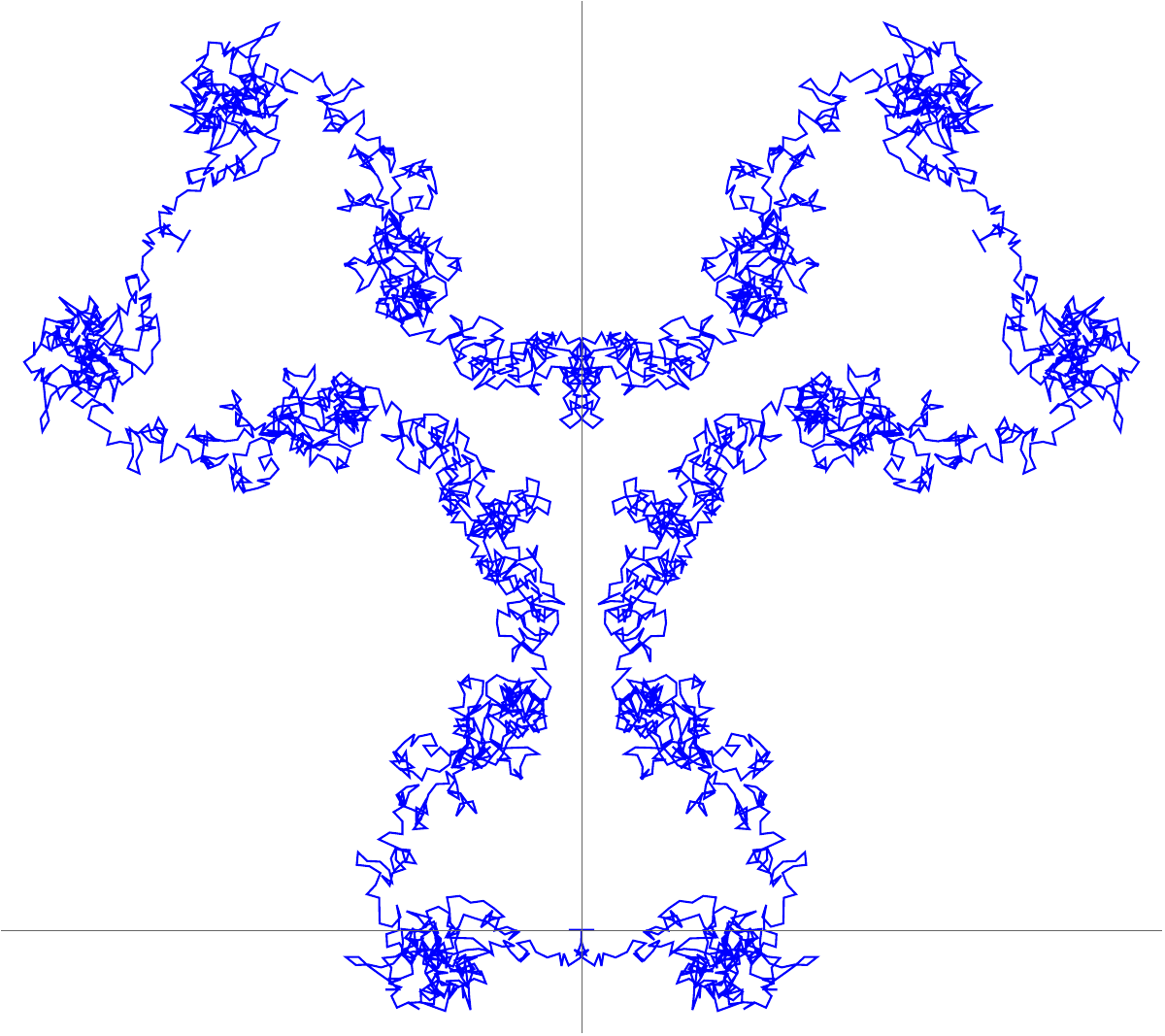}
    \caption{On the left, the path of partial sums of $\sum_{n=1}^Ne(tn^{1/2})$, as considered by Loxton~\protect{\cite{Loxton1983}}, here for $t=1{,}000$, $450\leqslant N\leqslant 650$. On the right, the Kloosterman path corresponding to the partial sums of $\mathrm{Kl}(5,1;3^8)$~\protect{\cite{MilicevicZhang2023}}.}
    \label{fig: earlypaths}
\end{figure}

A natural family of character paths arises from the normalized quadratic Gauss sums to square-free fundamental discriminant moduli:
\begin{equation}
\label{gausssum-def}
G(c)=\frac1{\sqrt{c}}\sum_{m=1}^c\Big(\frac mc\Big)e^{2\pi im/c}.
\end{equation}
We term the polygonal paths joining the partial sums of $G(c)$ Gauss paths. Having perhaps been conditioned to expect arithmetically defined sums such as Kloosterman sums to exhibit fractal-like behavior seen in Figure~\ref{fig: earlypaths}, the pictures of several (mostly randomly chosen) Gauss paths to large moduli shown in Figure~\ref{fig: gausspaths} might take one for a surprise.

\begin{figure}[ht]
    \centering
    \includegraphics[height=0.28\textwidth]{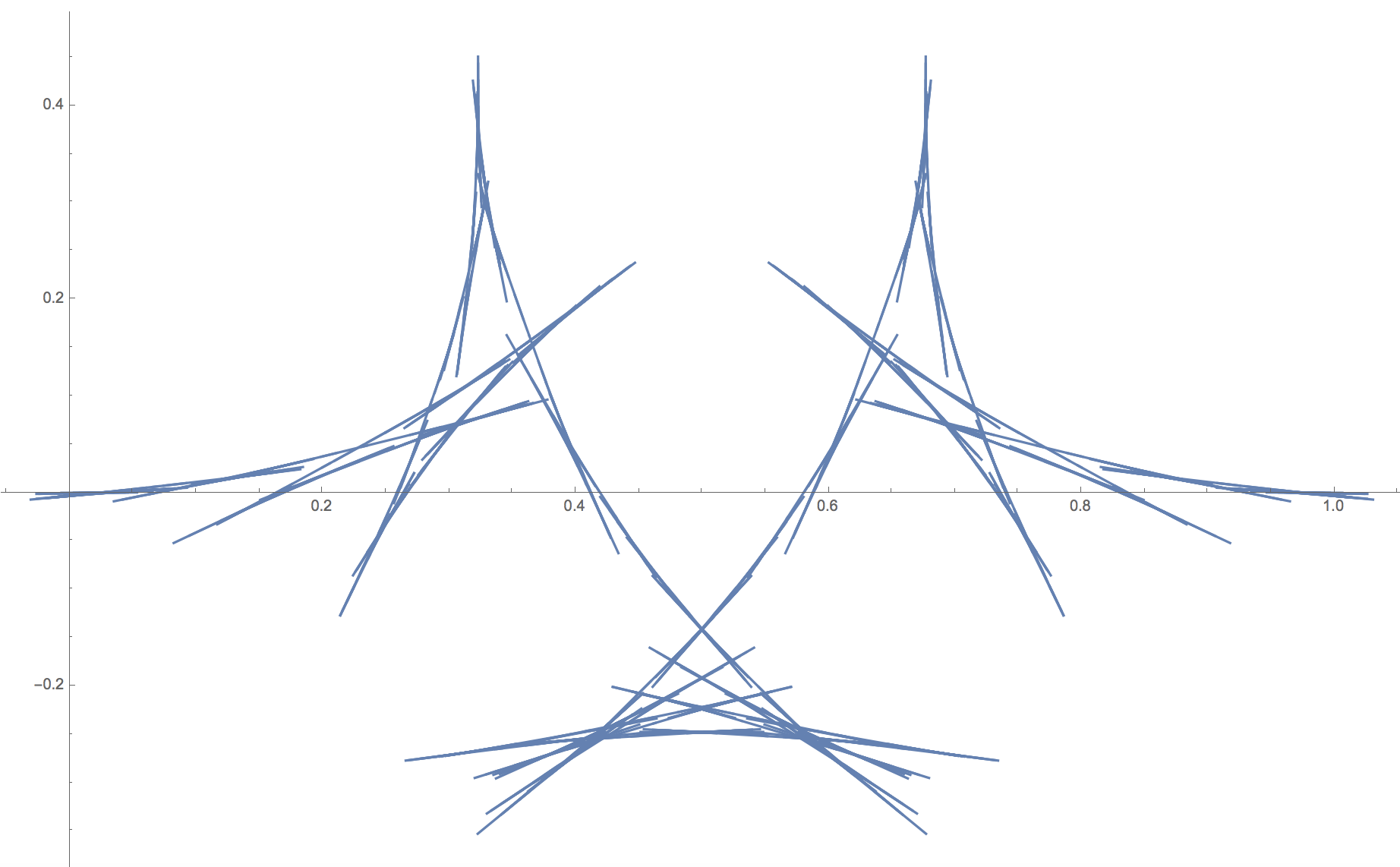}
\qquad    
    \includegraphics[height=0.28\textwidth]{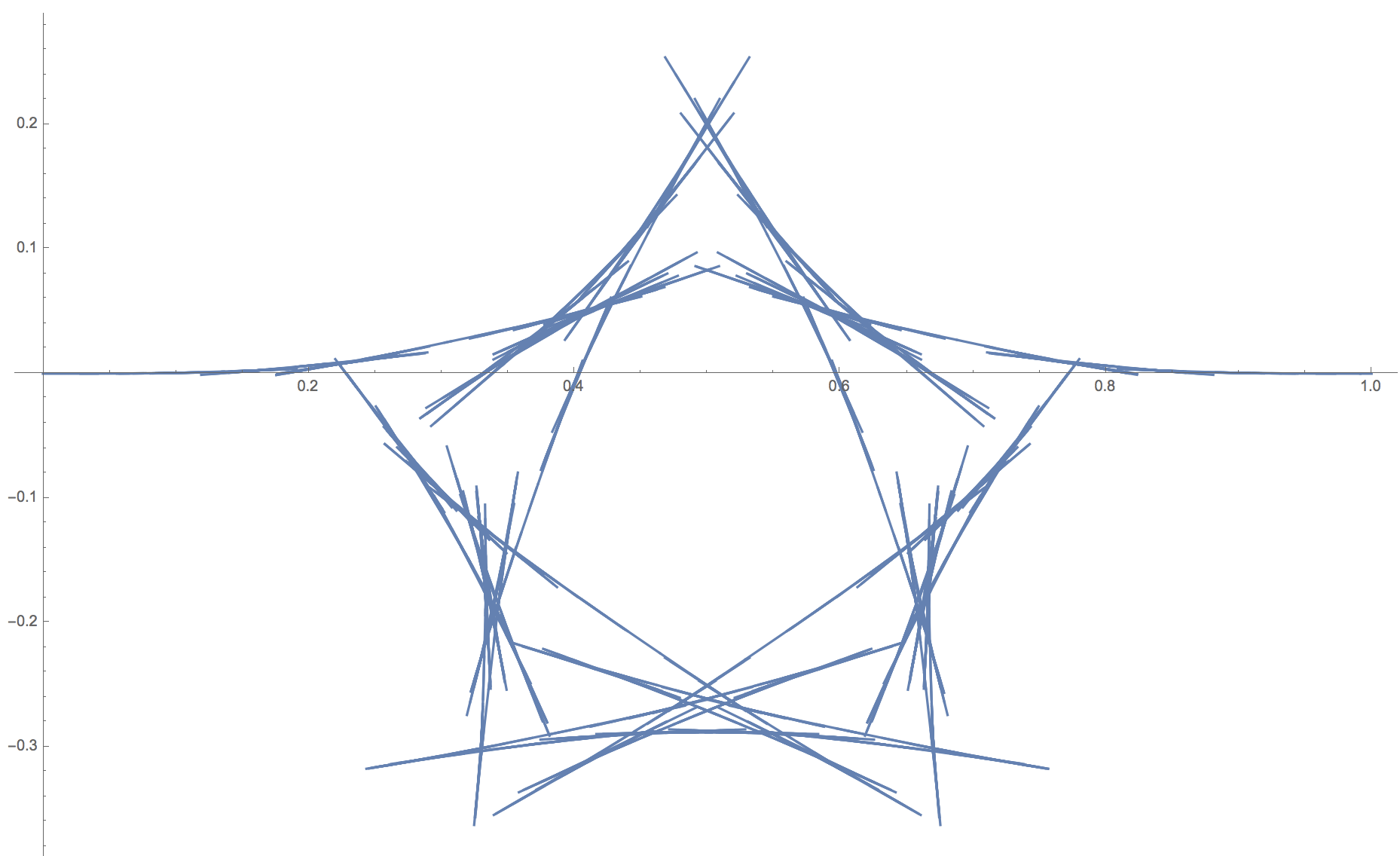}
    \caption{The quadratic Gauss paths for $c = 163841$ and $c = 224737$.}
    \label{fig: gausspaths}
\end{figure}

Pictures don't lie, of course, and provide additional food for thought when, after even just a modest amount of experimentation, stunningly similar pictures start showing up; see Figure~\ref{fig: moregausspaths}.

What is going on here? We can observe very long stretches in which the Legendre symbol $\big(\frac mc\big)$ has a definite (statistical) preference for one of the $\pm 1$ signs over the other, and it is easy to believe that this behavior is guided by $c$ falling in certain residue classes to some small moduli. But why would only a few small moduli seemingly matter, and what is the deal with the sharp reversals?

\begin{figure}[ht]
    \centering
    \includegraphics[height=0.28\textwidth]{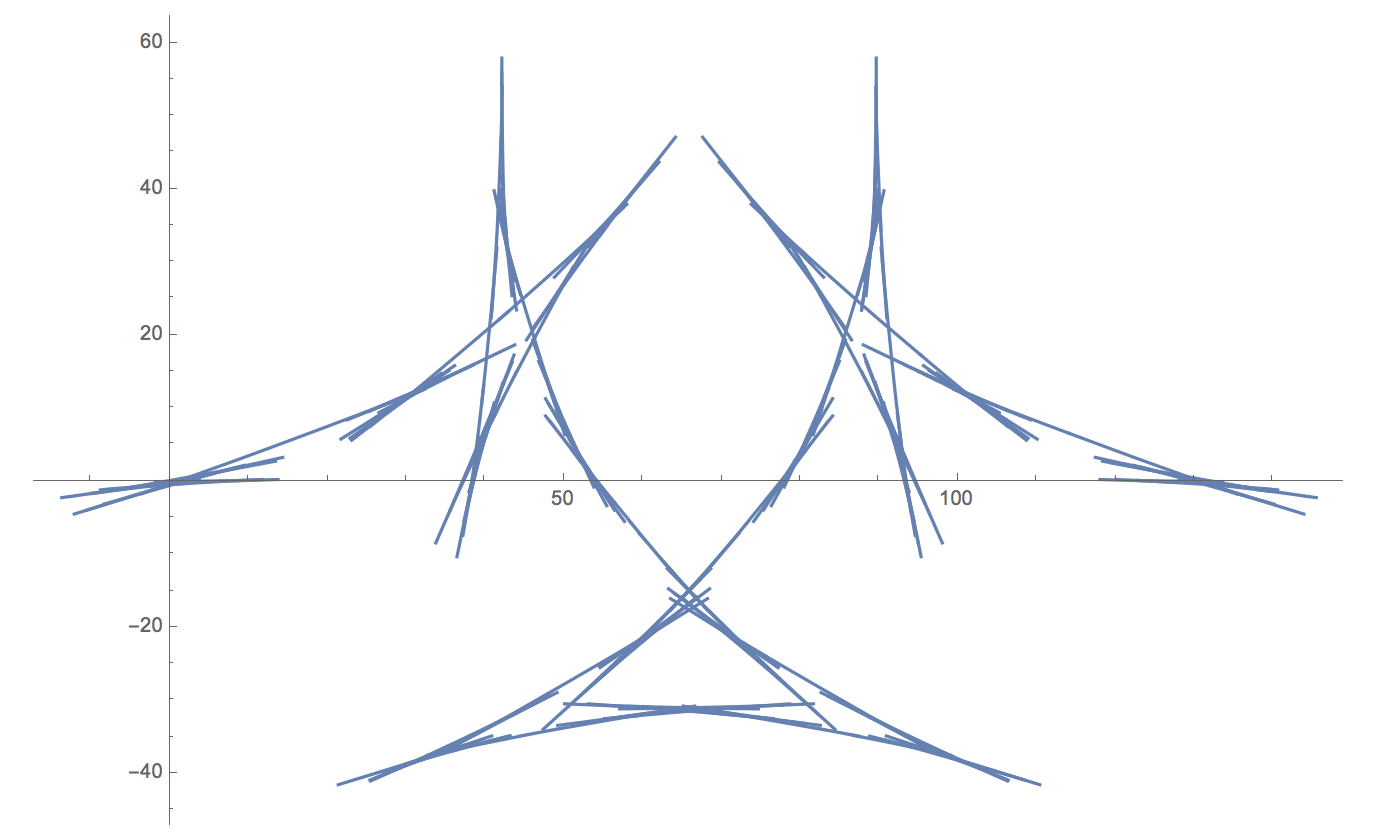}
\qquad    
    \includegraphics[height=0.28\textwidth]{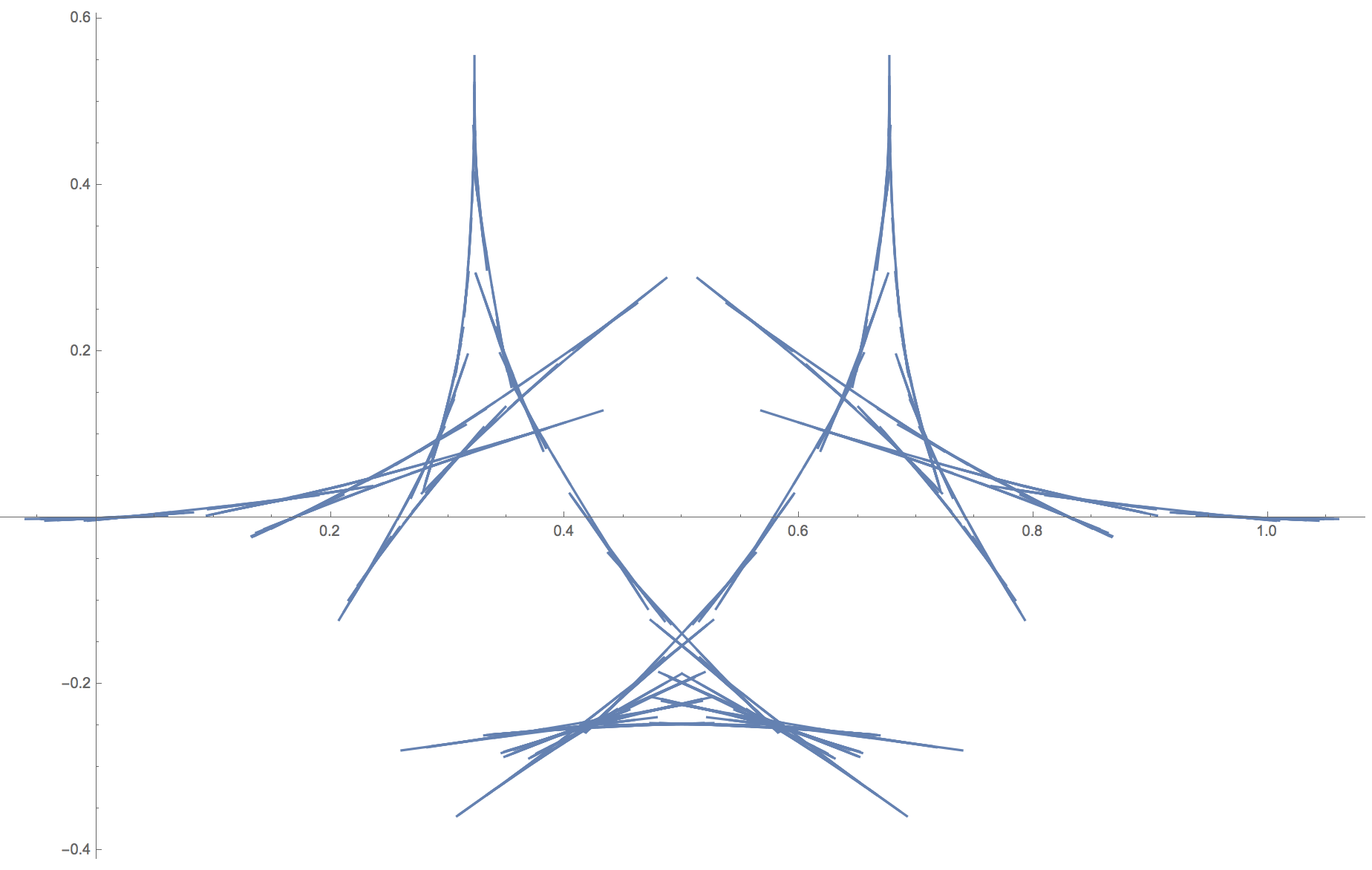}
    \caption{The quadratic Gauss paths for $c = 17393$ and $c = 128201$.}
    \label{fig: moregausspaths}
\end{figure}

\subsection{Limiting distribution}
Let $\cd$ denote the set of positive, square-free integers $c$ for which $c \equiv 1 \bmod{4}$. We now formally define the \emph{quadratic Gauss path} $G(t; c)$ for $c\in\cd$.
For $t = {j}/{(c-1)}$ for some $j \in [0, c-1]\cap\mathbb{Z}$, we let $G(t; c)=g_j$, where $g_j = c^{-1/2}\sum_{m=1}^{j}\Leg{m}{c}e_c(m)$ is the $j$th partial sum of \eqref{gausssum-def}. For ${(j-1)}/{(c-1)} < t < {j}/{(c-1)}$ ($j \in [1, c-1]\cap\mathbb{Z}$), we obtain $G(t; c)$ by linearly interpolating between $g_j$ and $g_{j+1}$. Then, $G(\cdot; c): [0,1] \to \C$ is a continuous function which maps $t \mapsto G(t; c)$.

For every $Q\geqslant 3$, we may consider the sample space $\cd_Q:=[Q, 2Q] \cap \cd$ with the uniform probability measure $m_Q$, and the map
\[ \cd_Q \to C^0([0,1], \C),\quad  c \mapsto G(\cdot; c)\] 
can be viewed as a $C^0([0, 1], \C)$-valued random variable $G_Q$ on this probability space.
We collect some relevant background on probability in Banach spaces in \S\ref{prob-Banach} for convenient reference. Note that $|\mathcal{D}_Q|\sim Q/(3\zeta(2))$; see \eqref{DQ-comput}.

Our first main result, Theorem~\ref{thm1}, establishes that the random variables $(G_Q)$ converge in law, as $Q\to\infty$, to a specific $C^0([0,1],\mathbb{C})$-valued random variable $G^{\ast}$, which we are about to describe. It incorporates completely multiplicative random variables $X_n$, which are the same as those used in the probabilistic model for the Jacobi symbols $(d/n)$ (as $d$ ranges through all fundamental discriminants) described in \cite{GranvilleSoundararajan2003}. For every odd prime $p$, we let $X_p$ be the random variable which takes the value 0 with probability $1/(p+1)$ and takes the values 1 and $-1$ each with probability $p/2(p+1)$; in other words, it is the identity random variable on the sample space $\{0, 1, -1\}$ equipped with the measure $\lambda_p$ defined by 
\begin{equation}
 \lambda_p(1) = \frac{p}{2(p+1)}, \quad \lambda_p(-1) = \frac{p}{2(p+1)}, \quad \lambda_p(0) = \frac{1}{p+1}. \label{eq: lambda-p} \end{equation}
For $p=2$, we let $X_2$ be the usual Bernoulli random variable; that is, the identity random variable on $\{1,-1\}$ with the measure
\begin{equation}
 \lambda_2(1)=\lambda_2(-1)=\frac12.
\label{eq: lambda-2}
\end{equation}
Let $(X_p)$ be a sequence of independent random variables of laws $\lambda_p$, and let $X_m$ be completely multiplicative random variables defined, for $m = \pm p_1^{a_1} \cdots p_k^{a_k}$, as
\begin{equation}
\label{Xm-def}
X_m = X_{p_1}^{a_1} \cdots X_{p_k}^{a_k}.
\end{equation}

\begin{thm}
\label{thm1}
Let $(X_m)$ be a completely multiplicative sequence of random variables of law given by \eqref{eq: lambda-p}--\eqref{Xm-def}.
\begin{enumerate}
\item\label{thm1-item1}
The random Fourier series
\begin{equation}
\label{limiting-random-partial-sums}
G^*(t) := \sum_{n \neq -1, 0} X_n \left( \frac{e((n+1)t) -1}{2 i \pi (n+1)} \right) + t
\end{equation}
converges almost surely to a continuous function and so defines a $C^0([0, 1], \cc)$-valued random variable $G^{\ast}$.
\item\label{thm1-item2}
The sequence of random variables $(G_Q)$ converges in law to $G^{\ast}$ as $Q\to\infty$.
\end{enumerate}
\end{thm}

\subsection{Convergence in probability and the atlas of shapes}
We now turn our attention to the visually observed sensitivity of the Gauss paths $G(\cdot;c)$ on congruence properties of $c$ to small moduli. To this end, for a parameter $Z\geqslant 1$, and let $\bm{\epsilon}_Z=(\epsilon_p)_{p\leqslant Z}$ denote any fixed choice of $\epsilon_p\in\{-1,0,1\}$ over primes $p\leqslant Z$ (with $\epsilon_2\in\{-1,1\}$).

Now, on the one hand, we may consider the sample space
\begin{equation}
\label{DQCQ-def}
\cd_{Q,\bm{\epsilon}_Z}=\Big\{n\in\cd_Q:\Big(\frac pn\Big)=\epsilon_p\text{ for every }p\leqslant Z\Big\}
\end{equation}
equipped with the uniform probability measure and the $C^0([0,1],\mathbb{C})$-valued random variable $G_{Q,\bm{\epsilon}_Z}:c\mapsto G(\cdot;c)$; in other words, the random variable $G_{Q,\bm{\epsilon}_Z}$ is obtained from $G_Q$ by conditioning on the event that $(p/c)=\epsilon_p$ for all $p\leqslant Z$. On the other hand, letting $(X_p)_{p>Z}$ be a sequence of independent random variables of laws $\lambda_p$ as in \eqref{eq: lambda-p}, we may consider the sequence of  completely multiplicative random variables defined as
\[ X_{m,\bm{\epsilon}_Z}=\prod_{\substack{p^{a_p}\exmid m\\p\leqslant Z}}\epsilon_p^{a_p}\prod_{\substack{p^{a_p}\exmid m\\ p>Z}}X_p^{a_p}. \]
In other words, the random variables $X_{m,\bm{\epsilon}_Z}$ are obtained as in \eqref{Xm-def}, but changing the law of $\lambda_p$ for $p\leqslant Z$ to the delta mass on $\epsilon_p$. Then we may consider the random Fourier series
\begin{equation}
\label{Gast-def}
G^{\ast}_{\bm{\epsilon}_Z}(t)=\sum_{n\neq -1,0}X_{n,\bm{\epsilon}_Z}\left( \frac{e((n+1)t) -1}{2 i \pi (n+1)} \right) + t
\end{equation}
and the deterministic Fourier series
\begin{equation}
\label{Gsharp-def}
G^{\sharp}_{\bm{\epsilon}_Z}(t)=\mathbb{E}(G^{\ast}_{\bm{\epsilon}_Z}(t))=\sum_{\substack{n\neq -1,0\\ p\mid n\,\Rightarrow\,p\leqslant Z}}\prod_{p^{a_p}\exmid n}\epsilon_p^{a_p}\left( \frac{e((n+1)t) -1}{2 i \pi (n+1)} \right) + t.
\end{equation}

\begin{figure}[ht]
    \centering
\[ \begin{matrix}
\includegraphics[width=\wdt]{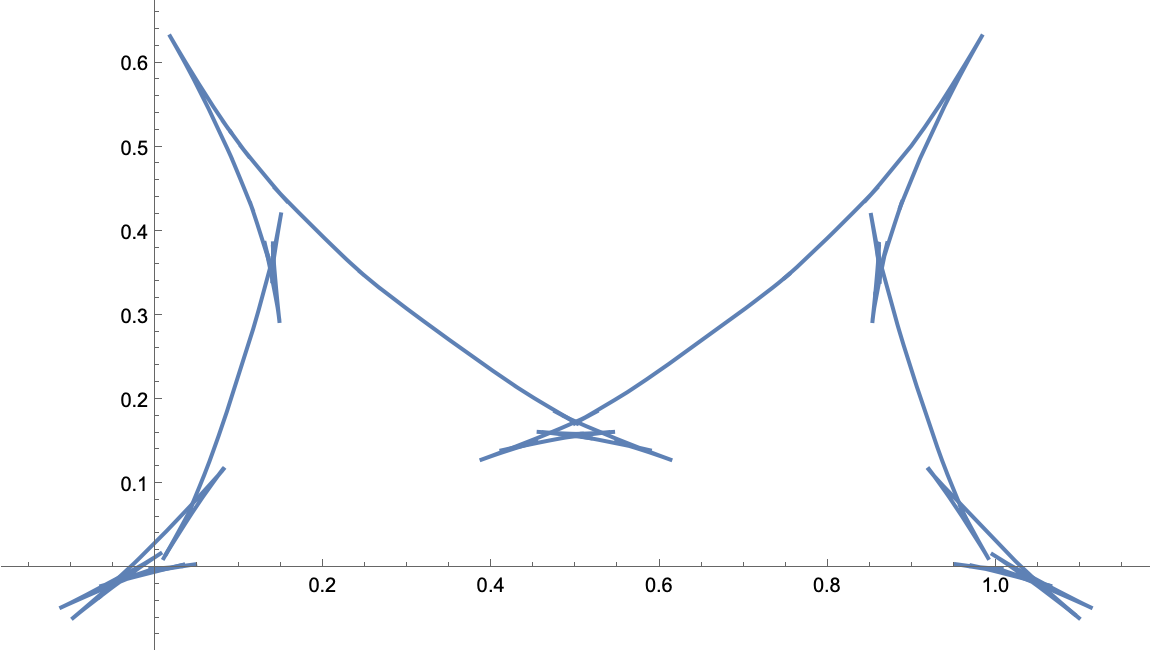} &
\includegraphics[width=\wdt]{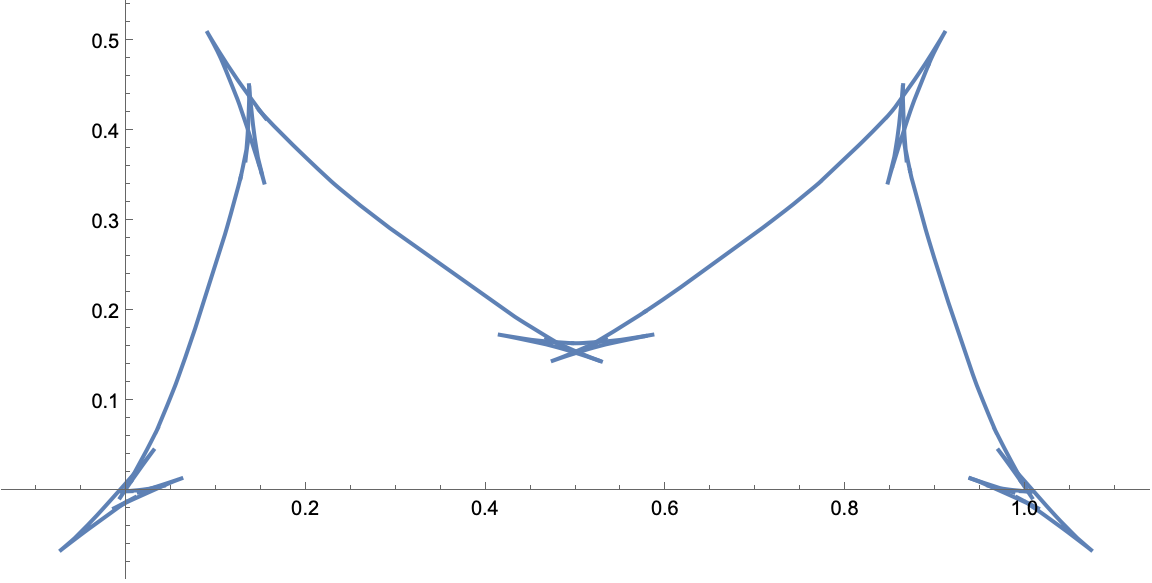} &
\includegraphics[width=\wdt]{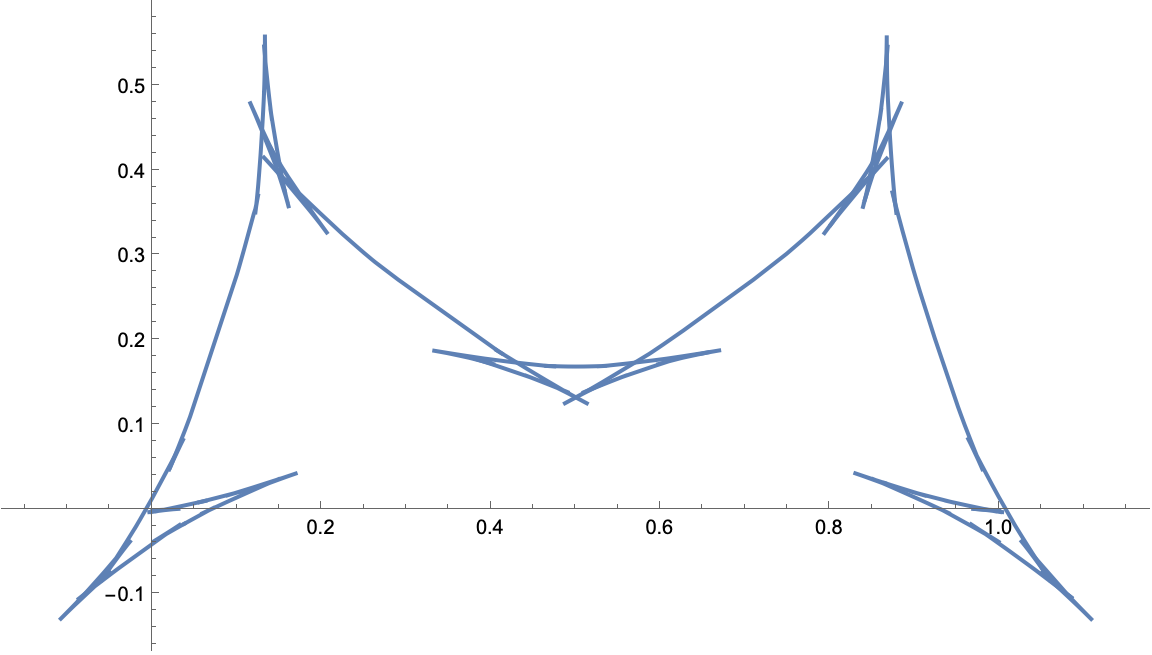}\\
\includegraphics[width=\wdt]{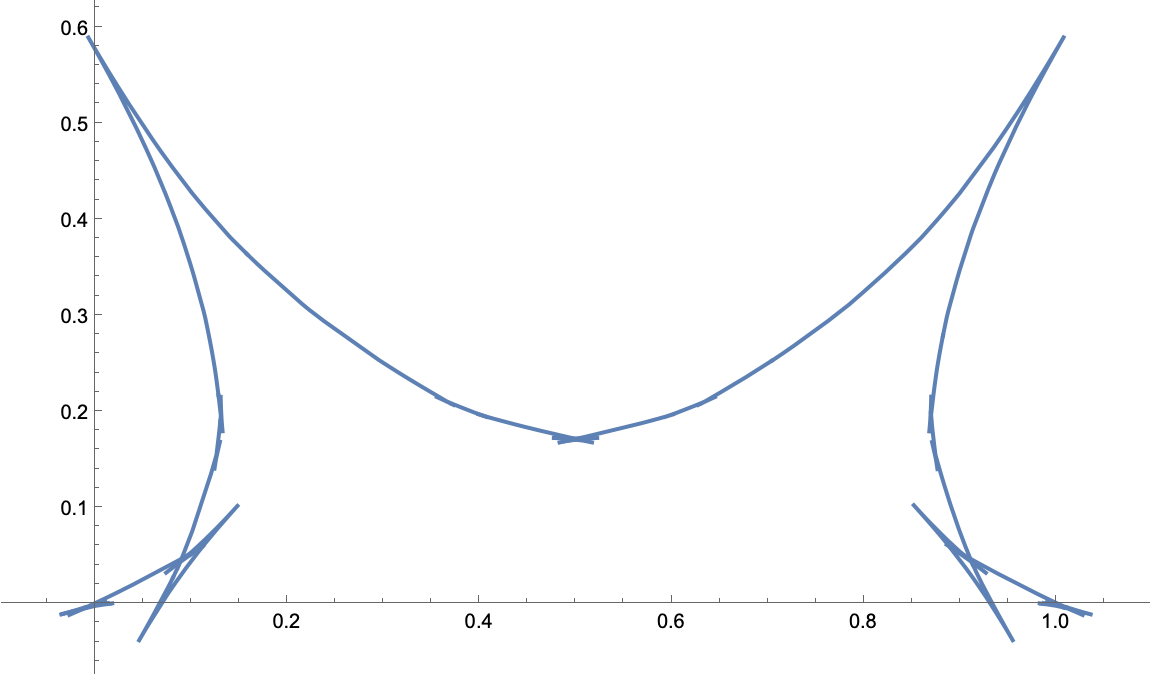} &
\includegraphics[width=\wdt]{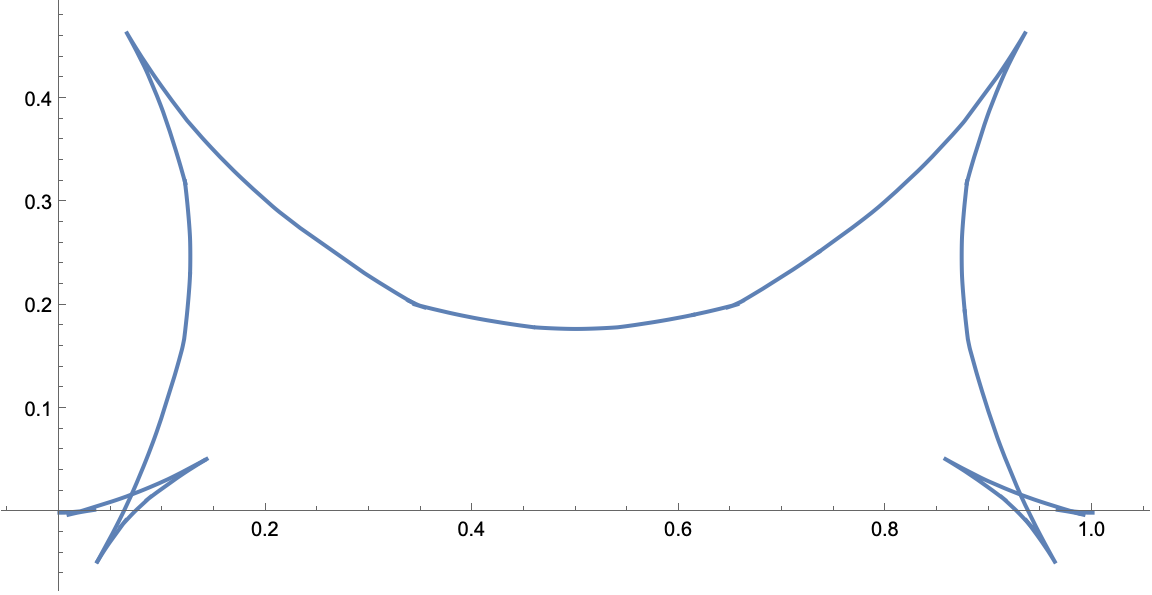} &
\includegraphics[width=\wdt]{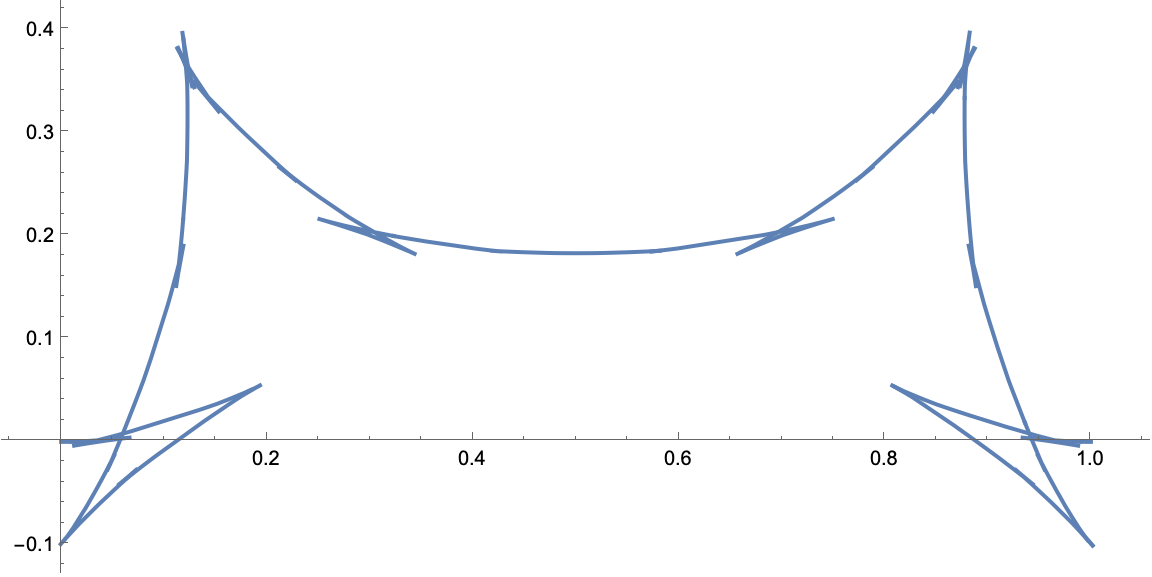}\\
\includegraphics[width=\wdt]{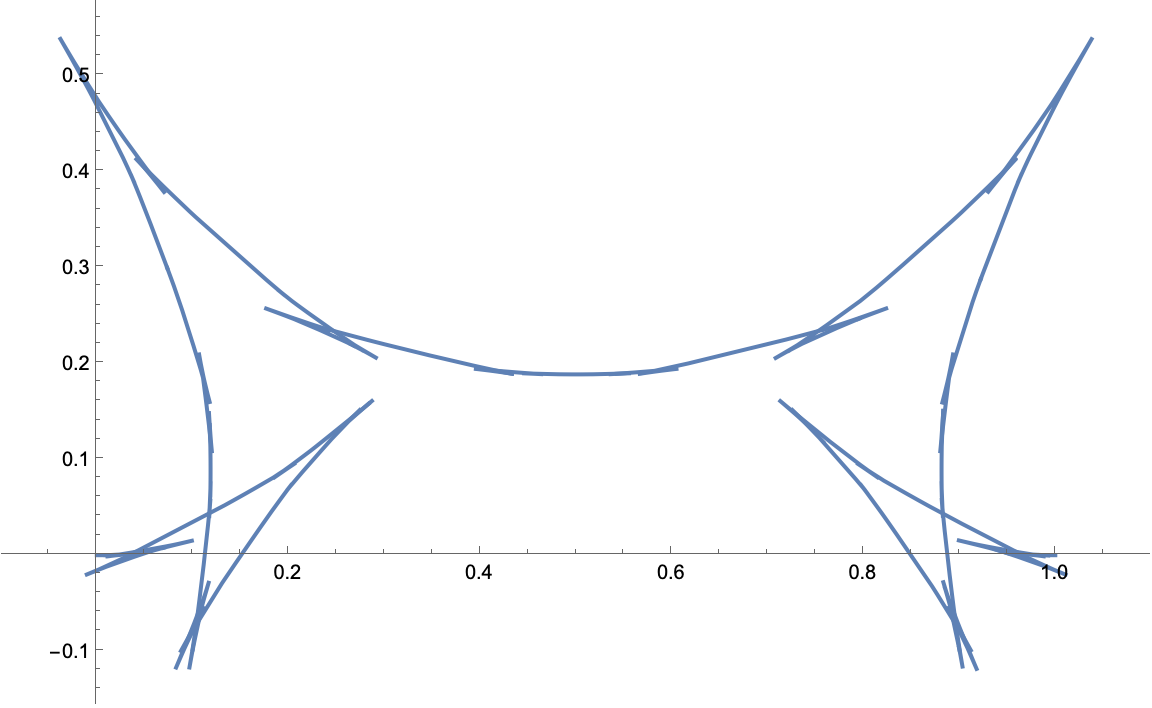} &
\includegraphics[width=\wdt]{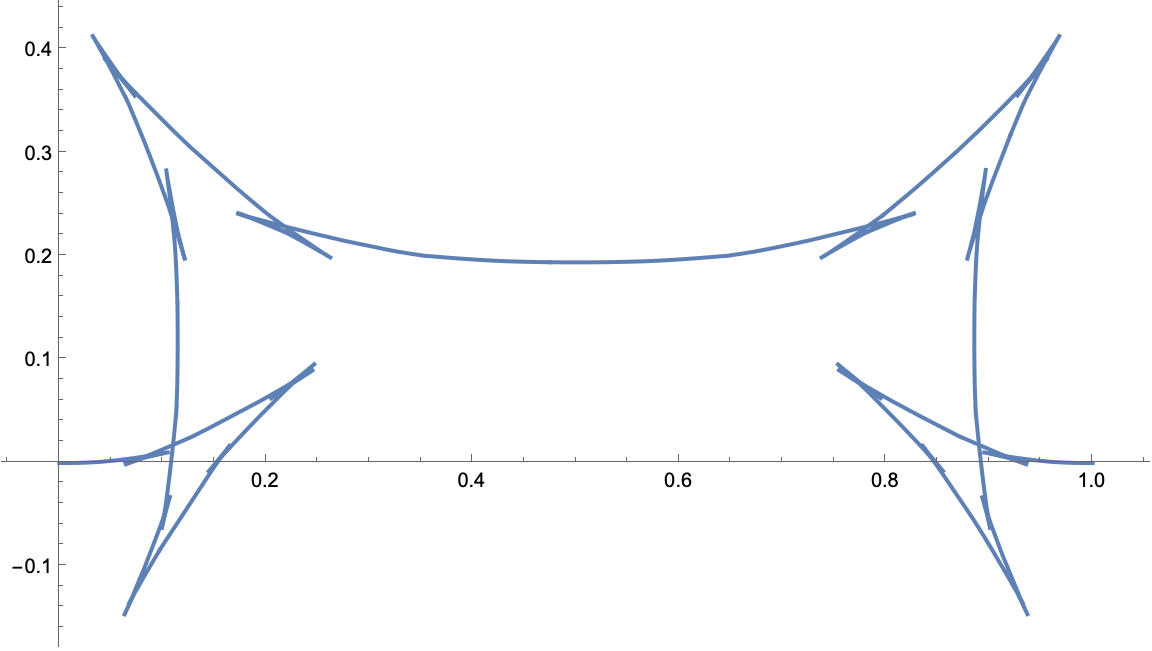} &
\includegraphics[width=\wdt]{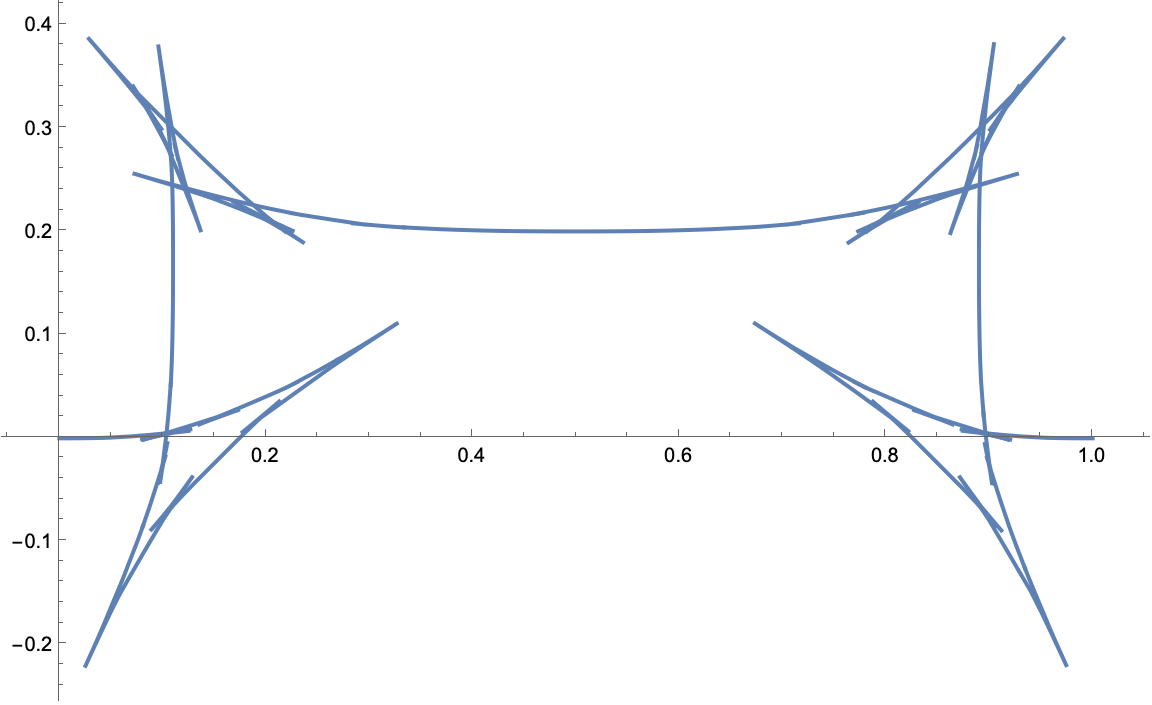}\\
\includegraphics[width=\wdt]{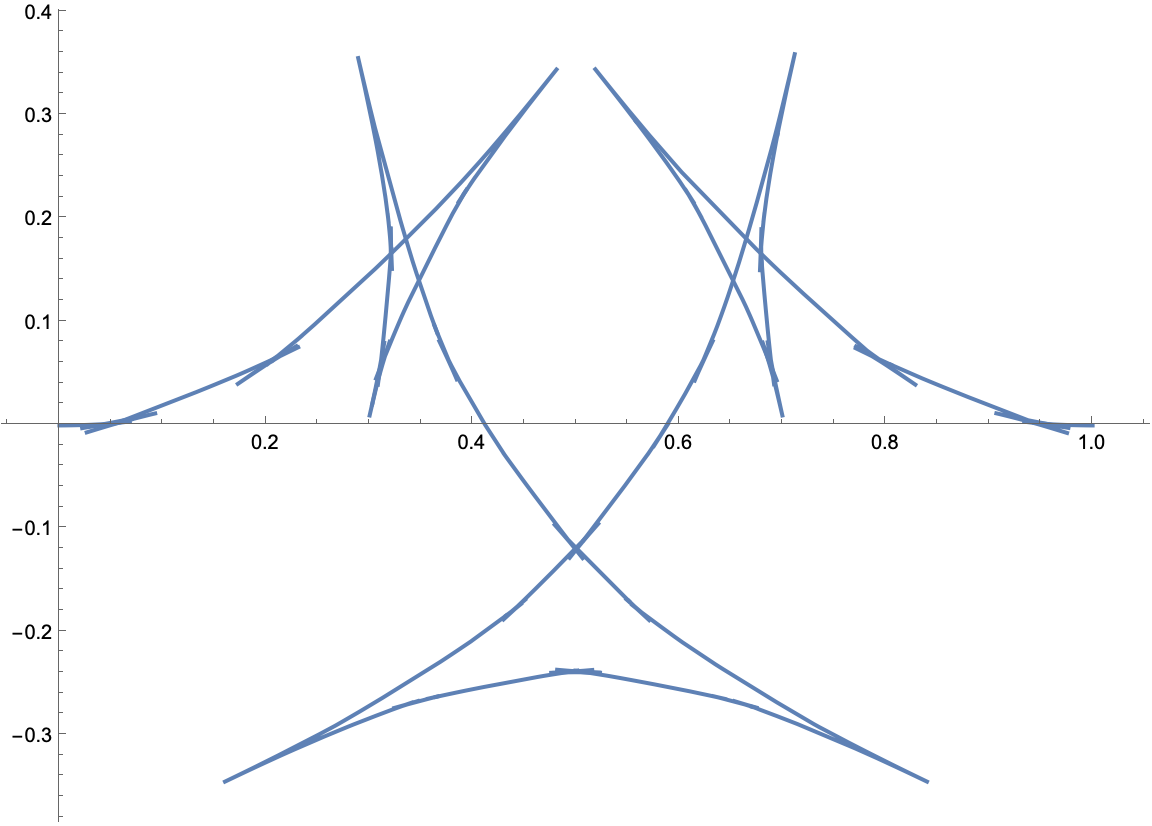} &
\includegraphics[width=\wdt]{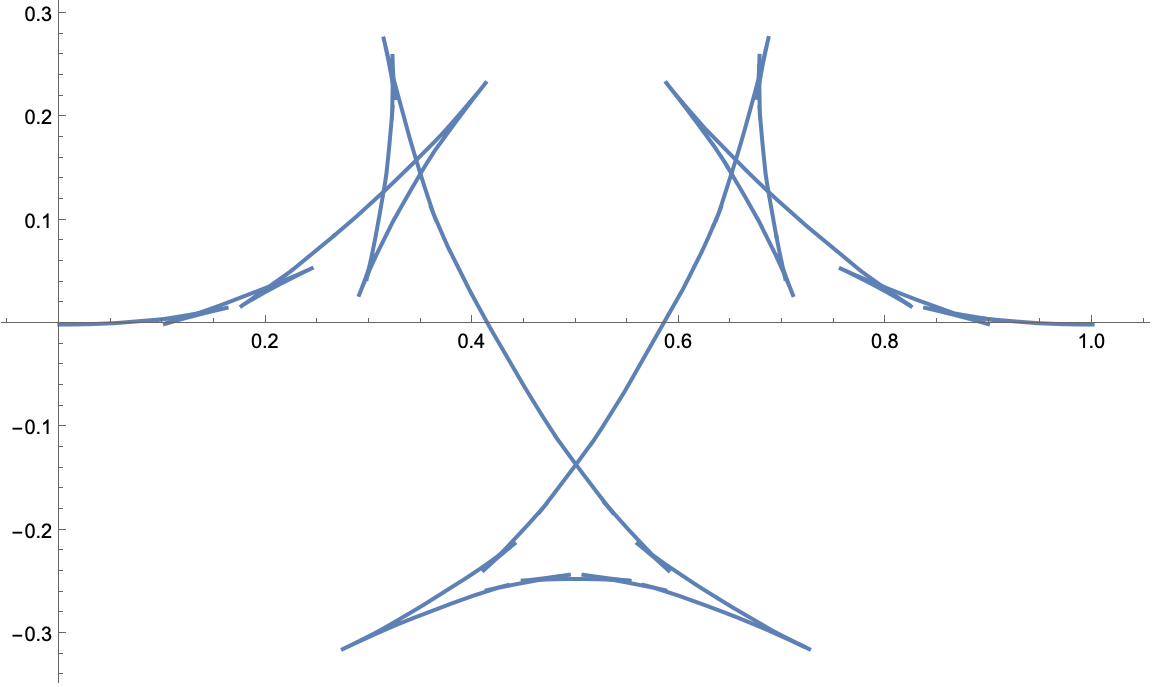} &
\includegraphics[width=\wdt]{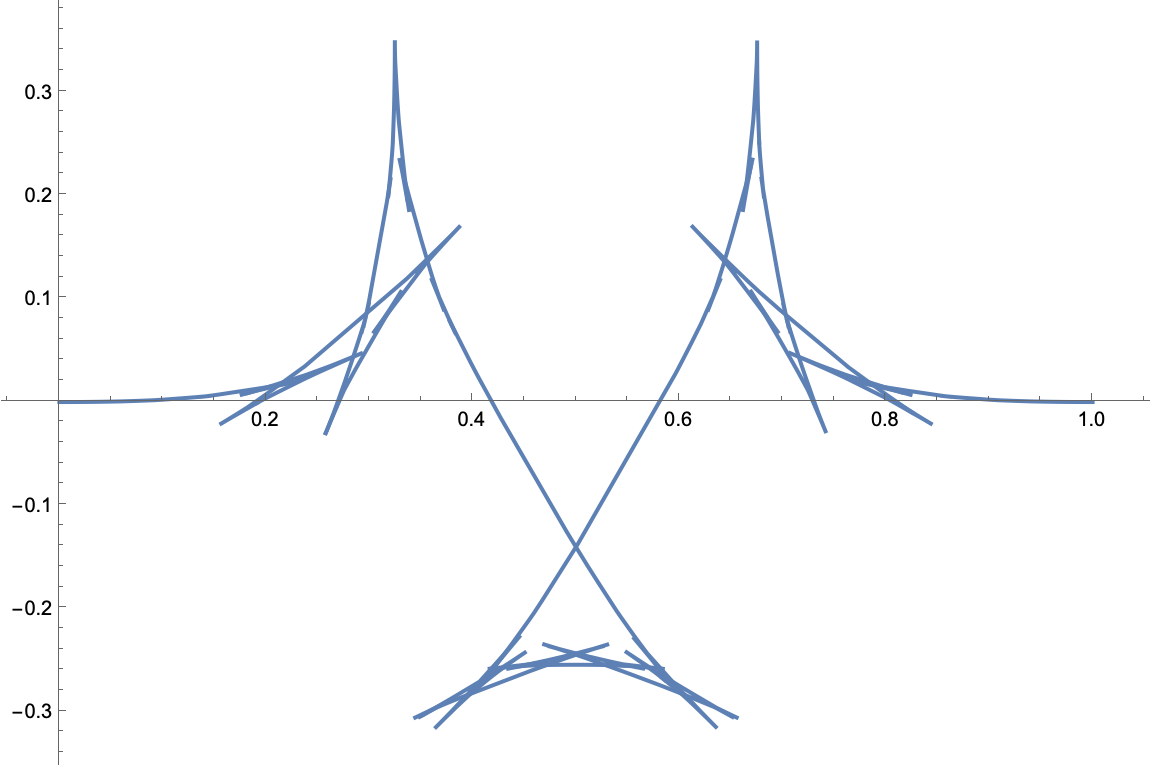}\\
\includegraphics[width=\wdt]{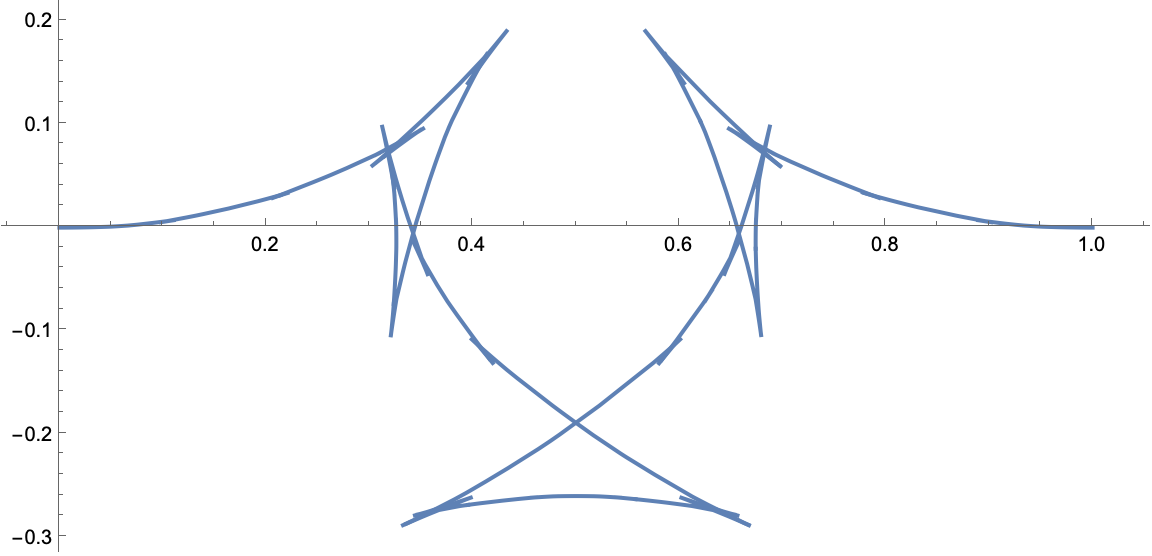} &
\includegraphics[width=\wdt]{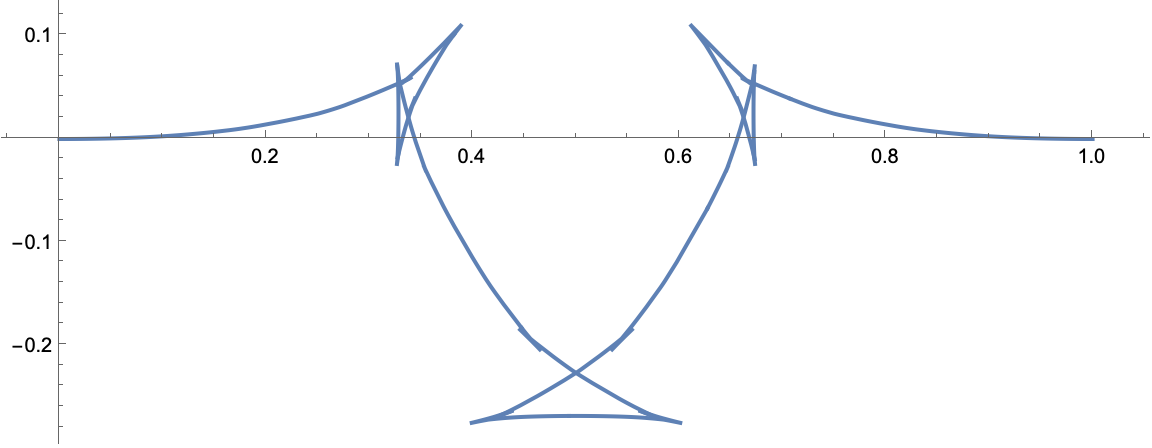} &
\includegraphics[width=\wdt]{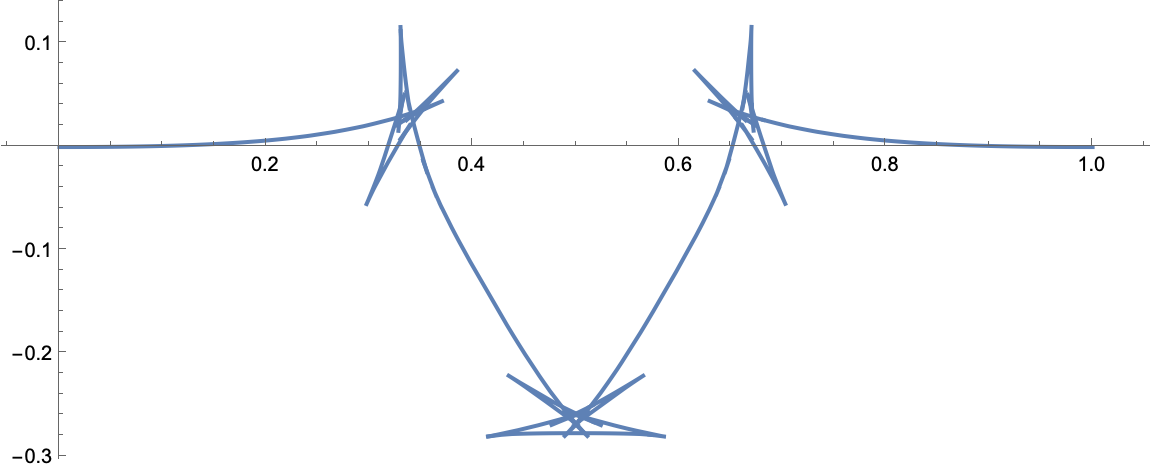}\\
\includegraphics[width=\wdt]{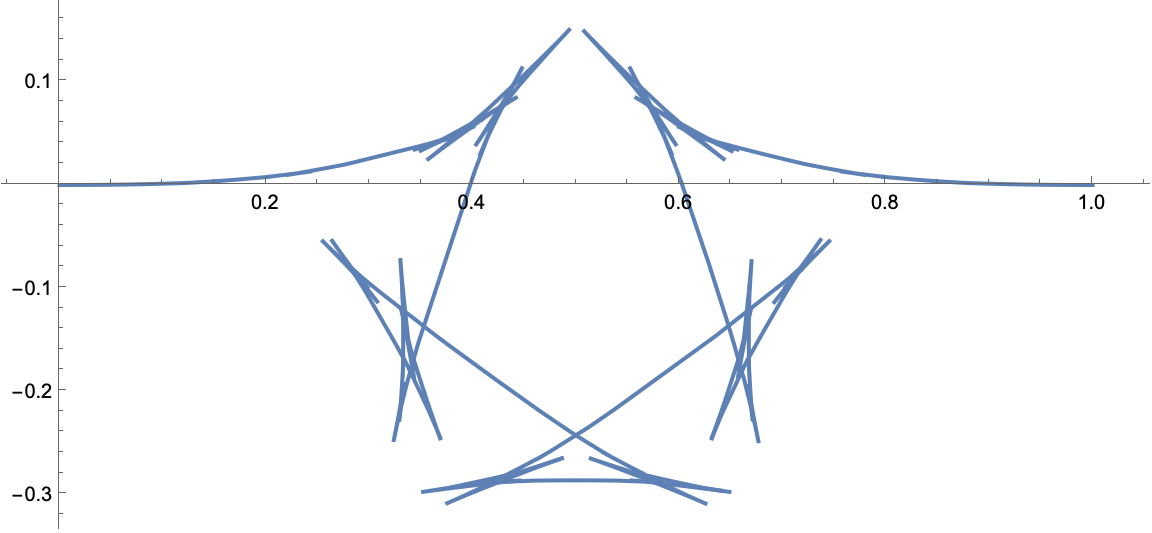} &
\includegraphics[width=\wdt]{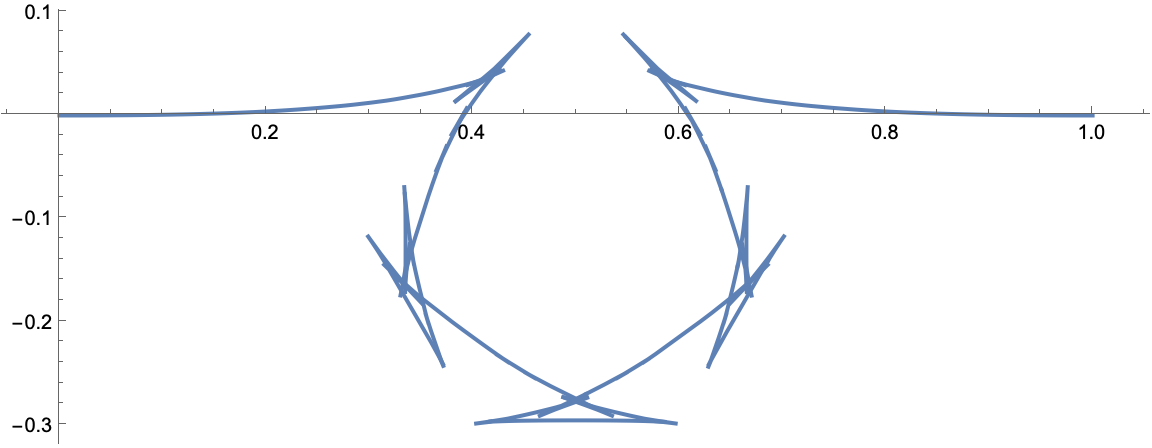} &
\includegraphics[width=\wdt]{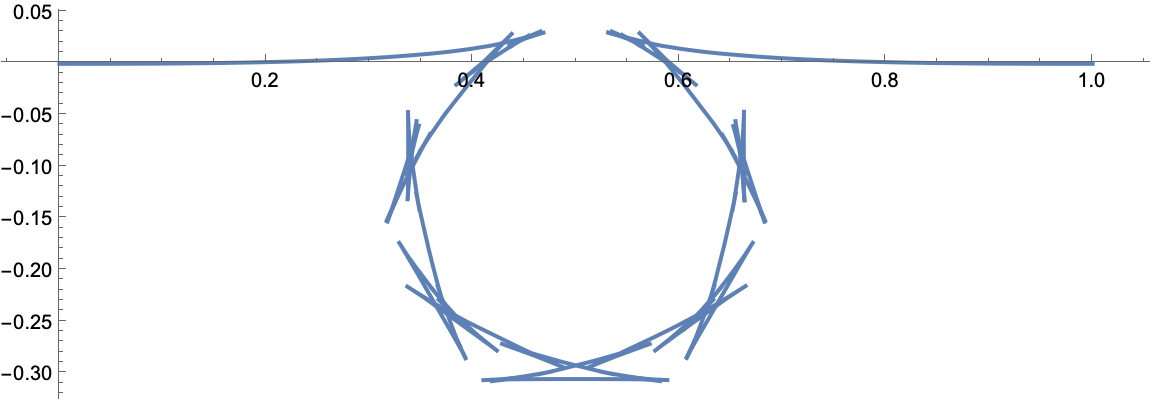}\\
\end{matrix}
 \]
    \caption{The atlas of paths $G^{\sharp}_{\bm{\epsilon}_5}$ over all $18=2\cdot 3\cdot 3$ choices of $(\epsilon_2,\epsilon_3,\epsilon_5)$, arranged in the lexicographic ordering.}
    \label{fig: atlas-5}
\end{figure}

The series defining $G^{\sharp}_{\bm{\epsilon}_Z}(t)$ is akin to the classical everywhere continuous, nowhere differentiable Weierstrass function. A small sample consisting of all possible shapes $G^{\sharp}_{\bm{\epsilon}_5}$ is shown in Figure~\ref{fig: atlas-5}. The following Theorem~\ref{thm2} formalizes the empirical observation that Gauss paths tend to be strongly aligned close to the shapes in the ensemble of the deterministic paths $G^{\sharp}_{\bm{\epsilon}_Z}(t)$ (according to $c\in\mathcal{D}_{Q,\bm{\epsilon}_Z}$ as in \eqref{DQCQ-def} for an increasingly large $Z$), which we may correspondingly term the \emph{atlas of shapes} of quadratic Gauss paths.

\begin{thm}
\label{thm2}
Let $Z\geqslant 1$, and let $\bm{\epsilon}_Z=(\epsilon_p)_{p\leqslant Z}$ denote any fixed choice of $\epsilon_p\in\{-1,0,1\}$ over primes $p\leqslant Z$, with $\epsilon_2\in\{-1,1\}$.
\begin{enumerate}
\item\label{thm2-gastz}
The random Fourier series $G^{\ast}_{\bm{\epsilon}_Z}(t)$ defined in \eqref{Gast-def} converges {a.s.} to a continuous function and defines a $C^0([0,1],\mathbb{C})$-valued random variable $G^{\ast}_{\bm{\epsilon}_Z}$. Moreover, the sequence of random variables $(G_{Q,\bm{\epsilon}_Z})$ defined in \eqref{DQCQ-def} converges in law to $G^{\ast}_{\bm{\epsilon}_Z}$ as $Q\to\infty$.
\item\label{thm2-gsharpZ}
The deterministic Fourier series $G^{\sharp}_{\bm{\epsilon}_Z}(t)$ converges absolutely and uniformly to a continuous function. If $\epsilon_p=-1$ for at least two $p\leqslant Z$, this deterministic path $G^{\sharp}_{\bm{\epsilon}_Z}(t)$ satisfies
\[ \lim_{t\to t_0\pm}\frac{G^{\sharp}_{\bm{\epsilon}_Z}(t)-G^{\sharp}_{\bm{\epsilon}_Z}(t_0)}{e(t_0)(t-t_0)}=\pm\infty\text{ or }\mp\infty \]
on an everywhere dense set of points $t_0\in [0,1]\cap\mathbb{Q}$ satisfying \eqref{nonvanishing-s} and described in Proposition~\ref{main-cor}. If $\epsilon_p=-1$ for exactly one $p\leqslant Z$, the same holds for $t_0=a/q$ satisfying the additional condition \eqref{all-QRs}; see Remark~\ref{remark1}.
\item\label{exceptional-probabilities}
For every $\delta>0$,
\[ \mathbb{P}\Big(\|G_{\bm{\epsilon}_Z}^{\ast}-G_{\bm{\epsilon}_Z}^{\sharp}\|_{\infty}\geqslant\delta\Big)\to 0\quad (Z\to\infty), \]
as well as
\begin{equation}
\label{double-limit}
\lim_{Z\to\infty}\limsup_{Q\to\infty}\mathbb{P}\Big(\|G_{Q,\bm{\epsilon}_Z}-G_{\bm{\epsilon}_Z}^{\sharp}\|_{\infty}\geqslant\delta\Big)=0.
\end{equation}
\end{enumerate}
\end{thm}

Our argument in fact provides an explicit upper bound for the exceptional probabilities in item \eqref{exceptional-probabilities} (after $\limsup_{Q\to\infty}$ in \eqref{double-limit}) of size $\ll\exp(-\delta^2Z^{1/8})$ for $Z\geqslant Z_1(\delta)$ (see \eqref{explicit-conclusion-1} and \eqref{explicit-conclusion-2}). We chose not to optimize this rate, which is already quite rapidly decreasing in $Z$ and thus suggests an explanation for why the experimentally observed shapes (which of course are just a finite and presumably random sample) appear to fall into very few classes dictated by congruence classes modulo several small primes. It bears emphasizing that Theorem~\ref{thm2} does \emph{not} claim that the paths $G(\cdot;c)$ must be uniformly close to $G^{\sharp}_{\bm{\epsilon}_Z}(t)$ for all $c\in\cd_{Q,\bm{\epsilon}_Z}$, nor can such a statement be true. A path $G(\cdot;c)$ will be substantially away from $G_{\bm{\epsilon}_Z}^{\sharp}$ as long as the values of $(p/c)$ over $p>Z$ exhibit a significant bias, and the same is true for a sample of the random path $G^{\ast}_{\bm{\epsilon}_Z}$ if the corresponding sample of $(X_p)_{p>Z}$ is biased. The point is that such an event is rare. Figures~\ref{fig: gsharp-eps} and \ref{fig: m11m111} illustrate Theorem~\ref{thm2} and these points. We also point the reader to Remarks \ref{remark1} and \ref{remark2} and the accompanying Figures \ref{fig: g71} and \ref{fig: g71-zoom}, which illustrate some delicate phenomena for limiting paths $G^{\sharp}_{\bm{\epsilon}_Z}(t)$ when all (or all but one) $\epsilon_p=1$.

\begin{figure}[ht]
    \centering
    \includegraphics[width=0.9\textwidth]{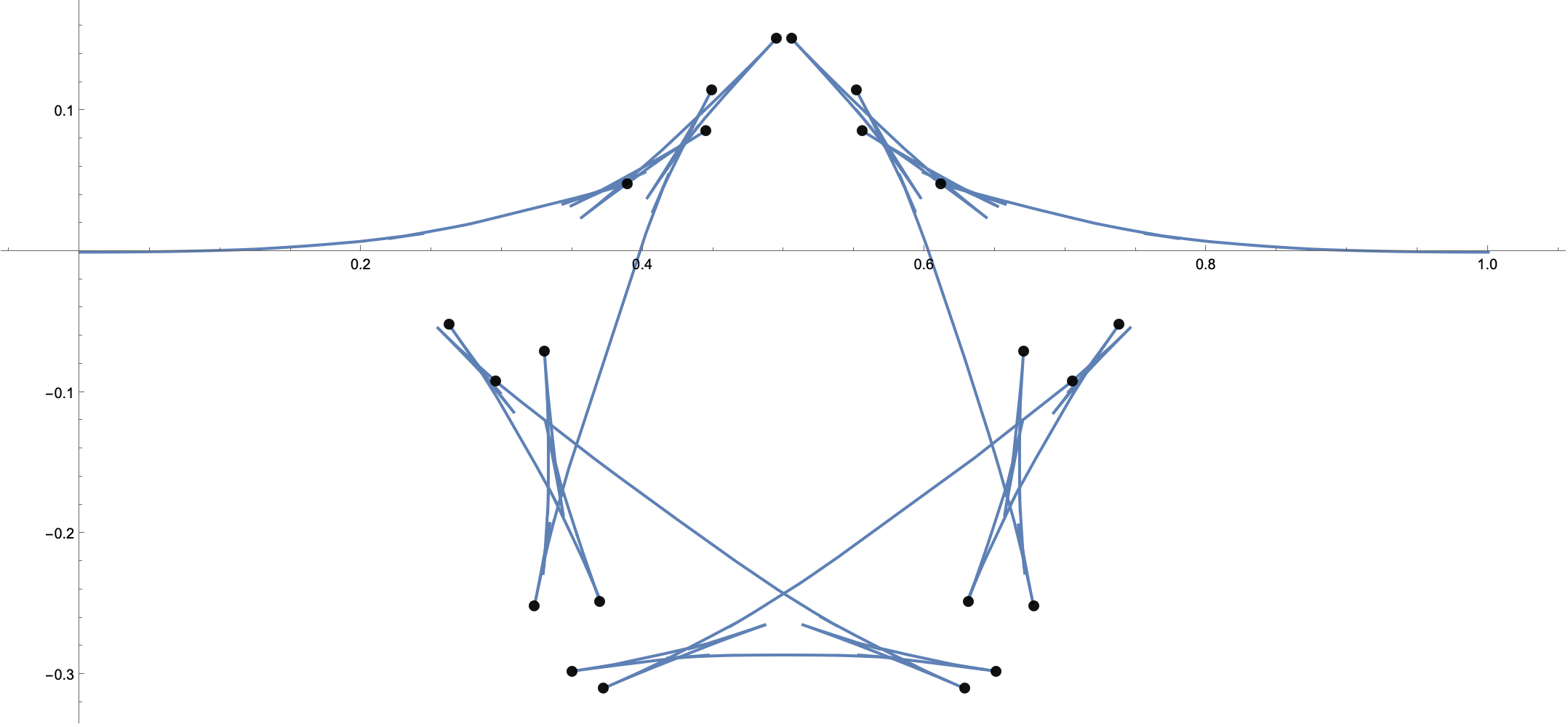}
   \caption{The deterministic path $G^{\sharp}_{\bm{\epsilon}_5}$ for $(\epsilon_2,\epsilon_3,\epsilon_5)=(1,1,-1)$, with marked points at $G^{\sharp}_{\bm{\epsilon}_5}(i/23)$, $1\leqslant i\leqslant 23$; see Proposition~\ref{main-cor}.}
    \label{fig: gsharp-eps}
\end{figure}

\begin{figure}[ht]
    \centering
    \includegraphics[width=0.7\textwidth]{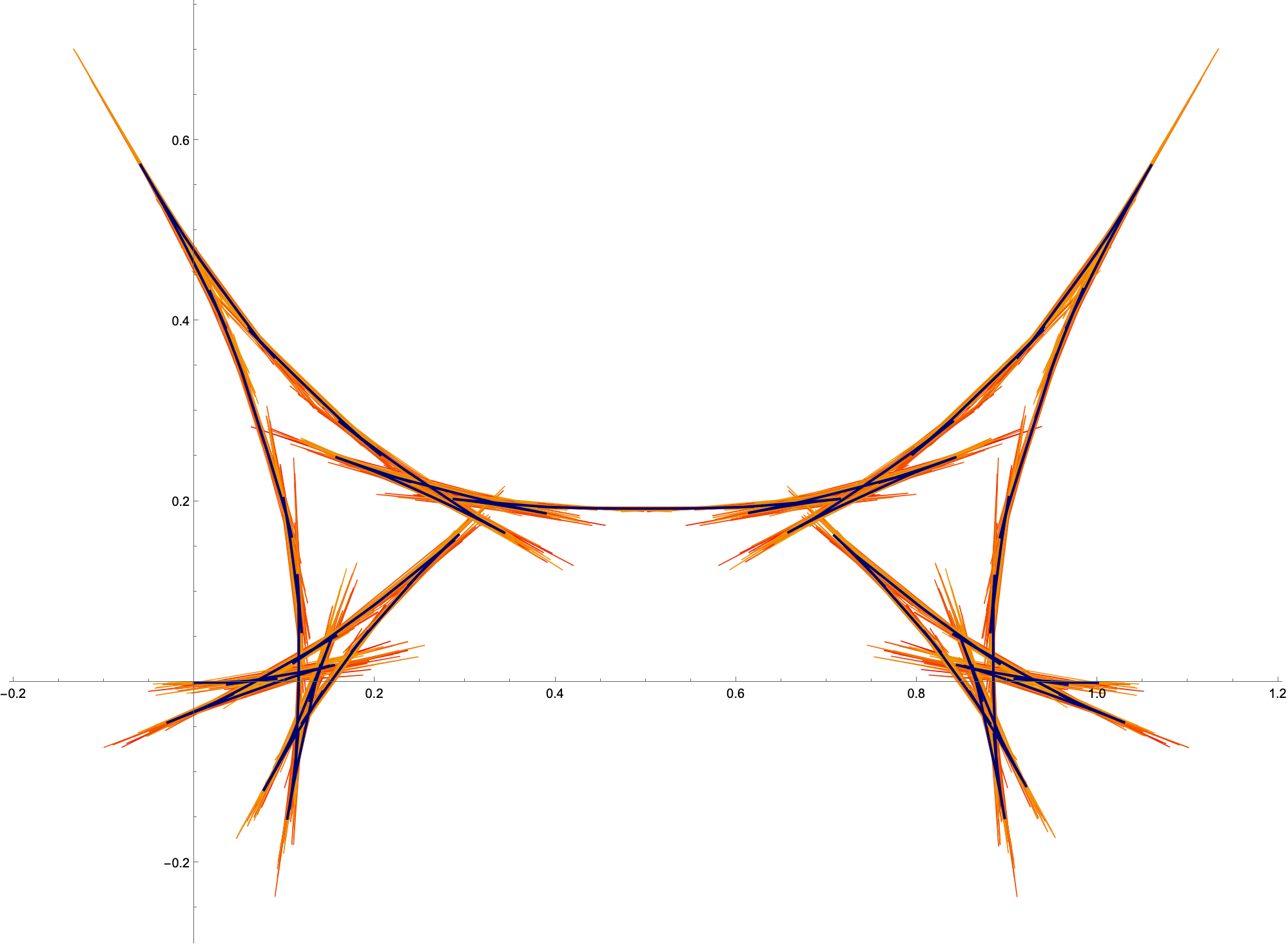}
   \caption{The deterministic path $G^{\sharp}_{\bm{\epsilon}_{11}}$ for $(\epsilon_2,\epsilon_3,\epsilon_5,\epsilon_7,\epsilon_{11})=(-1,1,-1,1,1)$ and 20 randomly chosen Gauss paths from the ensemble $G_{100000,\bm{\epsilon}_{11}}$.}
    \label{fig: m11m111}
\end{figure}

\subsection{Acknowledgement}
This paper grew out of the first author's 2022 thesis, advised by the second author. We were both inspired by and indebted to the existing literature on Kloosterman paths \cite{KowalskiSawin2016,MilicevicZhang2023} and on character paths in the unitary ensemble~\cite{Hussain2022}.
Some of the ingredients in establishing the convergence in law also appear in the work of Hussain and Lamzouri on Legendre paths~\cite{HussainLamzouri2024}, although our main focus on the striking sharp reversals and congruence-guided classification of Gauss paths is very different.

\subsection{Organization of the paper}
We collect some preliminaries on probability theory and estimates on character sums in section~\ref{sec-preliminaries}. The convergence and properties of the limiting random variable, including the proof of Theorem~\ref{thm1}~\eqref{thm1-item1} is studied in section~\ref{sec-rv}. We establish convergence in the sense of finite distributions of $(G_Q)\to G^{\ast}$ using the method of moments in section~\ref{computing-moments-section}, and then prove the convergence in law claim of Theorem~\ref{thm1}\eqref{thm1-item2} in section~\ref{sec-inlaw}. In section~\ref{sec-atlas}, we describe the local asymptotics of the limiting shapes $G^{\sharp}_{\bm{\epsilon}_Z}(t)$ at rational points $t_0\in[0,1]\cap\mathbb{Q}$ and then produce a collection of rational points at which $G^{\sharp}_{\bm{\epsilon}_Z}(t)$ exhibits a cusp. Finally, we prove Theorem~\ref{thm2} in section~\ref{final-section}.

\subsection{Notation}
As is common in analytic number theory, for $z\in\mathbb{C}$ we write $e(z)=e^{2\pi iz}$. We write $f=\mathrm{O}(g)$ or $f\ll g$ (using the notations interchangeably) if there exists a constant $C$, independent of all parameters not explicitly indicated with a subscript, such that $|f|\leqslant Cg$. We also write $f\asymp g$ if $f\ll g$ and $g\ll f$, and $f\sim g$ if $\lim(f/g)=1$, where the direction of the limit is either indicated or clear from the context. We denote by $\epsilon>0$ a positive value, which may differ from line to line but may in each case be taken to be as small as desired.

\section{Preliminaries}
\label{sec-preliminaries}

\subsection{Probability in Banach Spaces}
\label{prob-Banach}
In this section, we collect some definitions and facts pertaining to probability in Banach spaces.

For general facts about Banach spaces and probability, we refer to \cite[Preliminary Chapter, Chapters 1 and 4]{LiQueffelec2018}. In particular, for an arbitrary probability space $(\Omega, \mathcal{A}, \mathbb{P})$, a separable Banach space $E$, and the Borel $\sigma$-algebra $\mathcal{B}$ on $E$, an \emph{$E$-valued random variable} is a map $X: \Omega \to E$ that is $(\mathcal{A}$--$\mathcal{B})$-measurable (in other words, for any $B \in \mathcal{B}$, $X^{-1}(B) \in \mathcal{A}$).

We will be particularly interested in the separable Banach space $C^0([0, 1], \C)$ of complex-valued continuous functions on $[0, 1]$ equipped with the sup-norm. We can define several notions of convergence of random variables on this space. We closely follow the exposition in \cite[Appendix~A]{RicottaRoyer2018} and \cite[\S2.1]{MilicevicZhang2023}, and refer to \cite[\S{}B.11]{Kowalski2021} for proofs.

\begin{defn}[Convergence of Random Variables] \label{def: convergence}
Let $E$ be a Banach space and $(X_n)$ be a sequence of $E$-valued random variables on the probability spaces $(\Omega_n, \mathcal{A}_n, \mathbb{P}_n)$. Let $X$ be an $E$-valued random variable on the probability space $(\Omega, \mathcal{A}, \mathbb{P})$.
\begin{enumerate}
    \item If each $(\Omega_n, \mathcal{A}_n, \mathbb{P}_n) = (\Omega, \mathcal{A}, \mathbb{P})$, we say that $(X_n)$ converges to $X$ \emph{almost surely} if $\mathbb{P}(\{ \omega \in \Omega : \lim_{n \to \infty} X_n(\omega) = X(\omega) \}) = 1$.
    \item We say that $(X_n)$ converges \emph{in law} to $X$ if for every continuous and bounded map $\varphi: E \to \C$, the sequence $\left(\E(\varphi(X_n))\right)$ converges to $\E(\varphi(X))$.
    \item If $E = C^0([0, 1], \C)$, we say that $(X_n)$ converges to $X$ in the sense of \emph{finite distributions} if, for all $k \geq 1$ and all $k$-tuples $(t_1, t_2, \dots, t_k)$ where $0 \leq t_1< \dots < t_k \leq 1$, the sequence of $\C^k$-valued random vectors $(X_n(t_1), \dots, X_n(t_k))$ converges in law to the random vector $(X(t_1), \dots, X(t_k))$.
    \item If $E$ is separable, we say that the sequence $(X_n)$ is \emph{tight} if, for all $\epsilon>0$, there exists a compact subset $K\subseteq E$ such that, for all $n\geqslant 1$, $\mathbb{P}_n(\{X_n\in K\})\geqslant 1-\epsilon$.
\end{enumerate}
\end{defn}
Here, when $E=C^0([0,1],\C)$ and $t \in [0, 1]$, we denote by $X(t)$ the complex-valued random variable which is the evaluation of the random function $X$ at the point $t$, that is, $X(t) = e_t \circ X$, where $e_t: C^0([0, 1], \C) \to \C$ is the evaluation map.

In order to prove convergence in the sense of finite distributions, we can use the \emph{method of moments}. Recall that, for a $\C^k$-valued random vector $X=(X_1, X_2, \dots, X_k)$ and any $k$-tuples $\bm{m} = (m_1, \dots, m_k)$ and $\bm{n} = (n_1, \dots, n_k)$ in $\Z_{\geqslant 0}^k$, the complex moments  of $(X_1, X_2, \dots, X_k)$ are defined as
\[ \mathcal{M}_{\bm{m}, \bm{n}}(X) = \E\bigg( \prod_{i=1}^k \overline{X_i}^{m_i}X_i^{n_i}\bigg). \]
The random variable $X$ is said to be \emph{mild} if there exists a $\delta>0$ such that the power series
\[ \mathop{\sum\sum}_{\bm{m},\bm{n}\in\mathbb{Z}_{\geqslant 0}^k}\mathcal{M}_{\bm{m},\bm{n}}(X)\frac{z_1^{m_1}z_1'{}^{n_1}\cdots z_k^{m_k}z_k'{}^{n_k}}{m_1!n_1!\cdots m_k!n_k!} \]
converges in the disk $\{(z_1,z_1',\dots,z_k,z_k')\in\C^{2k}:|z_i|,|z_i'|\leqslant\delta\}$. A real-valued random variable $X$ is said to be $\sigma^2/$-sub-Gaussian if, for every $t\in\mathbb{R}$, $\mathbb{E}(e^{tX})\leqslant e^{\sigma^2t^2/2}$, and a complex-valued $X=Y+iZ$ is said to be $\sigma^2/2$-sub-Gaussian if $Y$ and $Z$ are. The sum $X=X_1+X_2$ of two independent $\sigma_i^2/2$-sub-Gaussian complex-valued random variables is $(\sigma_1^2+\sigma_2^2)/2$-sub-Gaussian. Moreover, a $\mathbb{C}^k$-valued random variable $X=(X_1,\dots,X_k)$ whose components are all sub-Gaussian is automatically mild. See \cite[\S{}B.5,\S{}B.8]{Kowalski2021} for details.

\begin{prop}[Method of moments, \protect{\cite[Theorem B.5.5(2)]{Kowalski2021}}]
\label{method-of-moments-prop}
Let $(X_n)$ be a sequence of $\cc^k$-valued random vectors, and let $X$ be a mild $\cc^k$-valued random vector. If, for any two $k$-tuples $\bm{m},\bm{n}\in\mathbb{Z}_{\geqslant 0}^k$, their corresponding complex moments $\mathcal{M}_n(\bm{m}, \bm{n})$ and $\mathcal{M}(\bm{m}, \bm{n})$ satisfy 
\[ \mathcal{M}_n(\bm{m}, \bm{n}) \to \mathcal{M}(\bm{m}, \bm{n})\quad (n\to\infty), \]
then $(X_n)$ converges to $X$ in law.
\end{prop}

The notion of tightness provides an effective way to upgrade the convergence in the sense of finite distributions of a sequence of $C^0([0,1],\mathbb{C})$-valued random variables to convergence in law via the following two propositions.

\begin{prop}[Prokhorov's Criterion] \label{prop: Prokhorov}
Let $(X_n)_{n=1}^\infty$ and $X$ be $C^0([0,1], \C)$-valued random variables. If the sequence $(X_n)$ is tight and converges to $X$ in the sense of finite distributions, then $(X_n)$ converges to $X$ in law.
\end{prop}
\begin{prop}[Kolmogorov's Tightness Criterion] \label{prop: Kolmogorov-tightness}
Let $(X_n)$ be a sequence of $C^0([0,1], \allowbreak\C)$-valued random variables. If there exist $\alpha, \delta > 0$ such that for all $0 \leq s, t \leq 1$ and $n \geq 1$, we have 
\[ \E\left( \left| X_n(t) - X_n(s) \right|^\alpha \right) \ll |s-t|^{1+\delta}, \]
then $(X_n)$ is tight.
\end{prop}

\subsection{Character sums and quadratic large sieve}
In this section, we recall several ingredients from classical analytic number theory. One of them is Heath-Brown's quadratic large-sieve inequality.

\begin{prop}[\protect{\cite[Theorem 1]{Heath-Brown1995}}]
\label{HBQLS}
For any $M,N\in\mathbb{N}$ and $a_1,\dots,a_N\in\C$,
\[ \sumast_{m\leqslant M}\bigg|\sumast_{n\leqslant N}a_n\bigg(\frac nm\bigg)\bigg|^2\ll_{\epsilon}(MN)^{\epsilon}(M+N)\sumast_{n\leqslant N}|a_n|^2, \]
with $\sumast$ denoting the sum over positive odd square-free values.
\end{prop}

For our proof of tightness, we will also require the $r=2$ and $r=4$ cases of Burgess' classical bound for short character sums. Here we note that what we really require of the second bound is that it is nontrivial in the range $N\ll c^{1/3-\delta}$ for some small $\delta>0$.

\begin{prop}[\protect{\cite[Theorem 12.6]{IwaniecKowalski2004}}]
\label{Burgess}
For every primitive character of any conductor $c>1$ and any $M$ and $N\geqslant 1$,
\[ \sum_{N<n\leqslant M+N}\chi(n)\ll_{\varepsilon} N^{1/2}c^{3/16+\varepsilon}. \]
\end{prop}

\begin{prop}[\protect{\cite[Theorem 1', special case]{Chang2014}}]
\label{Chang-Burgess}
For every $\kappa>0$, there exists a constant $\delta=\delta_{\kappa}>0$ such that, for every $\alpha\in\mathbb{R}$, every multiplicative character $\chi$ to any square-free modulus $q$, and every interval $I\subset [1,q]$ of size $|I|>q^{1/4+\kappa}$,
\[ \sum_{n\in I}\chi(n)e^{2\pi i\alpha n}\ll_{\kappa} |I|q^{-\delta}. \]
\end{prop}

In fact, \cite[Theorem 1']{Chang2014} allows an arbitrary degree $d$ polynomial phase, explicates the $d$-dependence in the implied constant, and implies that any $\delta_{\kappa}<c\kappa^2d^{-2}$ ($c$ an absolute constant) is admissible. The above estimate is all we need since $\kappa\in(0,\frac1{12})$ (that is, $\frac14+\kappa\in(\frac14,\frac13)$) will eventually be fixed. We note that more precise versions of Proposition~\ref{Chang-Burgess} are available in \cite{Kerr2014,Heath-BrownPierce2015} in various degrees of generality of $q$ and $|I|$.

\subsection{Linear forms in logarithms}
In our investigation of the limiting shapes of Gauss sums in section~\ref{sec-atlas}, we will make use of the following classical and powerful theorem of Baker on linear forms in logarithms.

\begin{prop}[\protect{\cite[Theorem 1]{Baker1977}}]
\label{Baker}
For every finite set of distinct primes $S=\{p_1,\dots,p_k\}$, there exists an (effectively computable) constant $C=C(S)>0$ depending on $S$ only such that, for every $\mathbf{0}\neq\mathbf{b}=(b_0,b_1,\dots,b_k)\in\mathbb{Z}^{k+1}$,
\[ |b_0+b_1\log p_1+\dots+b_k\log p_k|\geqslant\frac1{(e\|\mathbf{b}\|_{\infty})^C}. \]
\end{prop}

\section{The limiting random variable}
\label{sec-rv}

In this section, we establish the properties of the random Fourier series $G^{\ast}(t)$ defined by \eqref{limiting-random-partial-sums}, where we recall that $(X_n)$ is a completely multiplicative sequence of random variables of law given by  \eqref{Xm-def} and \eqref{eq: lambda-p}. We often use $\epsilon_m$ to denote a particular value of $X_m$ and $\Epsilon$ to denote any single infinite choice of $\epsilon_p\in\{-1,0,1\}$ for each prime $p$ (with $\epsilon_2\in\{-1,1\}$) and of the accompanying multiplicatively determined $\epsilon_n=\epsilon_{p_1}^{a_1}\dots\epsilon_{p_k}^{a_k}$ ($|n|=p_1^{a_1}\dots p_k^{a_k}$).

Let $P^{-}(n)$ and $P^{+}(n)$ denote the smallest and largest prime divisors of $n$, respectively. Following \cite[\S2]{Hussain2022}, we first consider the limit
\begin{equation}
\label{arithmetic-limit-equation}
\widetilde{G^{\ast}}(t)=\lim_{y \to \infty} \sum_{\substack{n \neq -1, 0 \\P^+(|n|) \leq y}} X_n \left( \frac{e((n+1)t) -1}{2 i \pi (n+1)} \right) + t.
\end{equation}
Of course, this has the same summands as \eqref{limiting-random-partial-sums}, but in a different order that is more strongly attuned to the multiplicative nature of the random coefficients $X_n$. In \S\ref{properties-limiting-RV-sec}, we will show that this is in fact equivalent to defining $G^*(t)$ as the limit of partial sums, as in \eqref{limiting-random-partial-sums}, and prove some important properties of this random variable. In the first subsection \S\ref{arith-covergence-sec}, we focus on proving the following result.

\begin{prop} \label{thm: convergence}
Let $X_n$ be multiplicative random variables, with $X_p$ for each prime $p$ independent and distributed according to the probability measure $\lambda_p$ defined in (\ref{eq: lambda-p}). Then, the random Fourier series $\widetilde{G^{\ast}}(t)$ in \eqref{arithmetic-limit-equation}
is almost surely the Fourier series of a continuous function.
\end{prop}

\subsection{Arithmetic convergence}
\label{arith-covergence-sec}
In this section, we study in detail the convergence of the random series $\widetilde{G^{\ast}}(t)$ in \eqref{arithmetic-limit-equation} and prove Proposition~\ref{thm: convergence}.

We denote, for $y_2>y_1\geqslant 1$,
\begin{equation}
\label{Ay1-Ay1y2}
A_{y_1} := \{ n \in \nn : P^{-}(n) > y_1 \},\quad A_{y_1}^{y_2}:=\{n\in\mathbb{N}:y_2\geqslant P^{+}(n)\geqslant P^{-}(n)>y_1\},
\end{equation}
and, for every subset $A\subseteq\mathbb{N}$ and $N\in\mathbb{N}$,
\[ A(N) := A\cap (N/e, N], \]
as well as
\begin{equation} \label{eq: S-A}
    S_{A} := \max_{t \in [0, 1]} \left| \sum_{n \in A} \frac{e(nt)}{n} X_n \right|,\quad S_{A}(\mathcal{E}) := \max_{t \in [0, 1]} \left| \sum_{n \in A} \frac{e(nt)}{n} \epsilon_n \right|.
  \end{equation}
The following result an analog of \cite[Proposition 5.2]{BoberGoldmakherGranvilleKoukoulopoulos2018}, with a number of details adjusted to suit our case.

\begin{prop} \label{prop: small-sum}
Let $k\geq 3$ be an integer and let $y_2>y_1\geq k^3$ be  real numbers. With notations as in \eqref{Ay1-Ay1y2} and \eqref{eq: S-A}, we have
\[ \mathbb{E}(S_{A_{y_1}^{y_2}}^{2k}) \ll y_1^{-k/21}. \]
\end{prop}

For positive integers $n,k,N\in\mathbb{N}$, we define 
\[ d_k(n) := \sum_{n_1 \cdots n_k = n} 1,\quad d_k(n; N) := \sum_{\substack{n_1 \cdots n_k = n \\ n_i \in A_{y_1}(N)}} 1.\]
Our proof of Proposition \ref{prop: small-sum} makes use of the following lemma to control sums over rough integers.

\begin{lem}[\protect{\cite[Lemma 5.4]{BoberGoldmakherGranvilleKoukoulopoulos2018}}] \label{lem: divisor-sum}
Let $\epsilon \in (0, 1]$, and let $k \geq 2$ be an integer. For $\sigma  \geq (2+\epsilon)/(2 + 2\epsilon)$ and $y \geq k^{1 + \epsilon}$, we have
\[ \sum_{P^-(n) > y} \frac{d_k(n)^2}{n^{2\sigma}} \ll e^{\mathrm{O}(k/\log k)}. \]
\end{lem}

\begin{proof}[Proof of Proposition~\ref{prop: small-sum}]
We begin by noting that the series defining $S_{A_{y_1}^{y_2}}$ converges comfortably (and uniformly across all samples $\mathcal{E}$ of the random coefficients $(X_n)$) and open with a simple decomposition estimate
\[ S_{A_{y_1}^{y_2}}(\Epsilon)\leqslant \sum_{j > \log y_1} S_{A_{y_1}^{y_2}(e^j)}(\Epsilon). \] Using H\"older's Inequality with $a_j = 1/j^2$, $b_j = j^2S_{A_{y_1}^{y_2}(e^j)}(\Epsilon)$, $p = 2k/(2k-1)$, and $q = 2k$ gives
\begin{align}
    S_{A_{y_1}^{y_2}}(\Epsilon)^{2k}
   &\leqslant \bigg( \sum_{j > \log y_1} \frac{1}{j^{4k/(2k-1)}}\bigg)^{2k-1} \bigg(\sum_{j > \log y_1} j^{4k}S_{A_{y_1}^{y_2}(e^j)}(\Epsilon)^{2k} \bigg) \nonumber \\
    &\ll \frac{1}{(\log y_1)^{2k+1}}\bigg(\sum_{j > \log y_1} j^{4k}S_{A_{y_1}^{y_2}(e^j)}(\Epsilon)^{2k} \bigg),\label{eq: Fred}
\end{align}
via a short calculation with integral comparison. Thus, in order to find a bound for $\mathbb{E}(S_{A_{y_1}^{y_2}}^{2k})$, it suffices to first bound $\mathbb{E}(S_{A_{y_1}^{y_2}(N)}^{2k})$ for $N \geq y_1$. 

Note that, for any $R\in\mathbb{N}$, we may write
\[ S_{A_{y_1}^{y_2}(N)}(\Epsilon) = \max_{t\in [0,1]}\sum_{n \in A_{y_1}^{y_2}(N)} \frac{e(nt)}{n} \epsilon_n = \max_{t\in[0,1]}\sum_{n \in A_{y_1}^{y_2}(N)}\left( \frac{e(n\lceil Rt\rceil/R)}{n} \epsilon_n \right) + \mathrm{O}(N/R). \]
Taking $R = \lfloor N^{21/20} \rfloor$ and applying convexity of $x\mapsto x^{2k}$, we obtain
\begin{align} S_{A_{y_1}^{y_2}(N)}(\Epsilon)^{2k} &\leq 2^{2k-1}\max_{1 \leq r \leq R}\bigg| \sum_{n \in A_{y_1}^{y_2}(N)}\frac{e(nr/R)}{n} \epsilon_n \bigg|^{2k} + \mathrm{O}\bigg(\frac{e^{\mathrm{O}(k)}}{N^{k/10}}\bigg) \nonumber \\
&\leq 2^{2k-1}\sum_{r=1}^R\bigg| \sum_{n \in A_{y_1}^{y_2}(N)}\frac{e(nr/R)}{n} \epsilon_n \bigg|^{2k} + \mathrm{O}\bigg(\frac{e^{\mathrm{O}(k)}}{N^{k/10}}\bigg). \label{eq: Sue}
\end{align}
Therefore, it suffices to bound
\[ S_{N, r, y_1,y_2} := \mathbb{E} \Bigg( \bigg| \sum_{n \in A_{y_1}^{y_2}(N)}\frac{e(nr/R)}{n} X_n \bigg|^{2k} \Bigg). \]

Denoting
\[ \widetilde{d_k}(n; N) := \sum_{\substack{n_1 \cdots n_k = n \\ n_1, \dots, n_k \in A_{y_1}^{y_2}(N)}} \prod_{j =1}^k e(n_jr/R) \]
and expanding, $S_{N,r,y_1,y_2}$ equals
\begin{equation}
\mathbb{E} \Bigg( \Bigg| \!\sum_{\substack{(N/e)^k < n \leq N^k \\ y_1<P^{-}(n)\leqslant P^{+}(n)\leqslant y_2}} \!\!\!\!\!\frac{\widetilde{d_k}(n; N)}{n} X_n \Bigg|^{2} \Bigg)
= \!\!\sum_{\substack{(N/e)^k < n, m \leq N^k \\ y_1<P^{-}(n)\leqslant P^{+}(n)\leqslant y_2\\
y_1<P^{-}(m)\leqslant P^{+}(m)\leqslant y_2}} \!\!\!\!\!\!\!\frac{\widetilde{d_k}(n; N)\overline{\widetilde{d_k}(m; N)}}{mn} \mathbb{E}(X_{mn}). \label{eq: Bob}
\end{equation}

Notice that $\mathbb{E}(X_{mn}) = 0$ when $mn$ is not a square and $|\mathbb{E}(X_{mn})|\leqslant 1$ in any case. Moreover, in (\ref{eq: Bob}), the terms which will survive are those for which $mn$ is a square, $m = uf^2$ and $n = ug^2$ for $f,g\in\mathbb{N}$ and $u$ squarefree. Finally, note that $|\widetilde{d_k}(n; N)| \leq d_k(n; N)$. Therefore, we have
\begin{align*}
    S_{N, r, y_1,y_2} &\leq \sum_{\substack{u \leq N^k \\ P^{-}(u) > y_1}}\mu^2(u)\mathop{\sum\sum}_{\substack{(N/e)^k < ud^2,uf^2 \leq N^k \\ P^{-}(d),P^{-}(f) > y_1}}\frac{d_k(ud^2; N)d_k(uf^2; N)}{u^2d^2f^2}.
\end{align*}
Since $udf \geq (N/e)^{k}$, we further obtain
\begin{align*}
    S_{N, r, y_1,y_2} &\leq \frac{e^{k/2}}{N^{k/2}}\sum_{\substack{u \leq N^k \\ P^{-}(u) > y_1}} \sum_{\substack{(N/e)^k < ud^2 \leq N^k \\ P^{-}(d) > y_1}} \sum_{\substack{(N/e)^k < uf^2 \leq N^k \\ P^{-}(f) > y_1}} \frac{d_k(ud^2; N)d_k(uf^2; N)}{u^{3/2}d^{3/2}f^{3/2}} \\
    &\leq \frac{e^{k/2}}{N^{k/2}} \bigg(\sum_{P^-(u) > y_1} \frac{d_k(u)^2}{u^{3/2}}\bigg)^3,
\end{align*}
by using the elementary inequality $d_k(mn)\leqslant d_k(m)d_k(n)$ and dropping various conditions.
We then use Lemma \ref{lem: divisor-sum} with $\epsilon = 1$ to get, for $y_1\geqslant k^2$,
\[ S_{N,r, y_1,y_2} \ll \frac{e^{k/2}}{N^{k/2}} e^{\mathrm{O}(k/\log k)} \ll \frac{e^{\mathrm{O}(k)}}{N^{k/2}}. \]

Plugging this into (\ref{eq: Sue}) we get
\begin{align}
    \mathbb{E}\big(S_{A_{y_1}^{y_2}(N)}^{2k}\big) \ll \frac{N^{21/20}e^{\mathrm{O}(k)}}{N^{k/2}} + \frac{e^{\mathrm{O}(k)}}{N^{k/20}} \ll \frac{e^{\mathrm{O}(k)}}{N^{k/20}},
\end{align}
since $k \geq 3$. Returning to (\ref{eq: Fred}), we have
\begin{align*}
    \mathbb{E}\big(S_{A_{y_1}^{y_2}}^{2k}\big)
    &\ll \frac{e^{\mathrm{O}(k)}}{\log(y_1)^{2k+1}}\sum_{n=0}^\infty \frac{(\log(y_1) + 1 + n)^{4k}}{y_1^{k/20}e^{nk/20}} \\
    &\leqslant \frac{e^{\mathrm{O}(k)}\log(y_1)^{2k-1}}{y_1^{k/20}}\sum_{n=0}^\infty  \bigg(\frac{(\log(y_1) + 1 + n)^{4}}{\log(y_1)^{4}e^{n/20}}\bigg)^k.
\end{align*}
Keeping in mind that since $y_1\geqslant k^3\geqslant 27$, $(\log y_1+1+n)/\log y_1\leqslant (2+n)$, and so
\[ \mathbb{E}\big(S_{A_{y_1}^{y_2}}^{2k}\big) \ll \frac{\log(y_1)^{2k-1}e^{\mathrm{O}(k)}}{y_1^{k/20}} \ll \frac{1}{y_1^{k/21}}. \qedhere \]
\end{proof}

We now prove that $\widetilde{G^{\ast}}(t)$ defined in \eqref{arithmetic-limit-equation} is almost surely the Fourier series of a continuous function. We use similar methods as in \cite[\S2]{Hussain2022}. We may rewrite every sample as
\begin{equation}
\label{Gtilde-decomp}
\begin{aligned}
&\widetilde{G^{\ast}}(\Epsilon; t):= \lim_{y\to\infty}\sum_{\substack{n \neq -1, 0\\P^{+}(|n|)\leqslant y}} \frac{e((n+1)t) -1}{2 i \pi (n+1)} \epsilon_n + t\\
&\qquad=\lim_{y\to\infty}\frac{e(t)}{2 i \pi}\sum_{\substack{n \neq -1, 0\\P^{+}(|n|)\leqslant y}} \frac{e(nt) -e(-t)}{n} \epsilon_n +t
-\lim_{y\to\infty}\frac{e(t)}{2 i \pi}\sum_{\substack{n \neq -1, 0\\P^{+}(|n|)\leqslant y}} \frac{e(nt) -e(-t)}{n(n+1)} \epsilon_n.
\end{aligned}
\end{equation}
Since the latter series converges absolutely and uniformly, it converges to a continuous function. Hence it suffices to show that
    \begin{equation} \label{eq: little-random-series} \lim_{y \to \infty} \sum_{\substack{n \neq -1, 0 \\P^+(|n|) \leq y}} \frac{e(nt) -e(-t)}{n} X_n \end{equation}
    converges almost surely to a continuous function. We remark that this passage is, in fact, valid in any order of summation.
    
Define
\begin{equation}
\label{Syt}
S_y(t) := \sum_{\substack{n \neq 0, -1 \\ P^+(|n|) \leq y}} \frac{e(nt) -e(-t)}{n} X_n, \quad S_y(\Epsilon; t) = \sum_{\substack{n \neq 0, -1 \\ P^+(|n|) \leq y}} \frac{e(nt) -e(-t)}{n} \epsilon_n.
\end{equation}
Then it is clear, by comparison with the absolutely convergent series $\sum_{P^{+}(n)\leqslant y}(4/n)=4\prod_{p\leqslant y}(1-p^{-1})^{-1}$, that, for every fixed $y>0$, every sample $S_y(\Epsilon;t)$ converges absolutely and uniformly to a continuous function. Since $S_y(\Epsilon; t)$ defines a continuous function for any $y$ and any choice of $\Epsilon$, it suffices to show that the sequence $(S_y)_y$ almost surely converges uniformly, as this will allow us to conclude that $\lim_{y \to \infty}S_y(t)$ is almost surely a continuous function. We do this using Cauchy's Criterion for uniform convergence.

Define 
\begin{equation}
\label{Ry1y2-def}
R_{y_1, y_2}(t) := S_{y_2}(t) - S_{y_1}(t)
= \sum_{\substack{n \neq 0, -1 \\ y_1 < P^+(|n|) \leq y_2}}\frac{e(nt) -e(-t)}{n} X_n.
\end{equation}
The series defining $R_{y_1,y_2}(t)$ converges absolutely and uniformly and may therefore be rearranged at will. Using the multiplicativity of the $X_n$'s we may thus write 
\begin{equation}
\label{Ry1y2-decomp}
R_{y_1,y_2}(t) = \sum_{\substack{n \neq 0 \\  P^+(|n|)\leq y_1}}\frac{X_n}{n} \sum_{\substack{y_1<P^{-}(m)\leqslant P^{+}(m)\leqslant y_2}} \frac{e(mnt) -e(-t)}{m} X_m.
\end{equation}

The following lemma is an analog of \cite[Lemma 2.1]{Hussain2022}, whose proof we closely follow.
\begin{lem}
\label{ry1y2-lemma}
For every $\delta > 0$, there exists a $y_0(\delta)>0$ such that for every $y_2>y_1\geqslant y_0(\delta)$, we have 
\[ \mathbb{P}\left(\| R_{y_1,y_2}(t) \|_{\infty} > \delta \right) \ll \exp\big(-\delta^2y_1^{1/6}\big). \]
\end{lem}

\begin{proof}
As before, we let $\epsilon_m$ be a value of the random variable $X_m$, and let $\Epsilon$ denote the choice of a value $\epsilon_p$ for each $X_p$.
Notice that
\begin{equation}
\label{key-to-Ry1y2}
\bigg| \sum_{\substack{y_1<P^{-}(m)\leqslant y_2}} \frac{e(mnt) -e(-t)}{m} \epsilon_m \bigg| \leq 2S_{A_{y_1}^{y_2}}(\Epsilon),
\end{equation}
where $S_{A_{y_1}^{y_2}}(\Epsilon)$ is defined as in (\ref{eq: S-A}).
Thus,
\begin{equation}
\label{Ry1y2-estimate}
\| R_{y_1,y_2}(\Epsilon; t)  \|_\infty \leqslant 4S_{A_{y_1}^{y_2}}(\Epsilon) \sum_{\substack{n \geq 1 \\ P^+(n) \leqslant y_1}} \frac{1}{n}=4S_{A_{y_1}^{y_2}}(\Epsilon)(e^\gamma\log y_1 + C)
\end{equation}
for some absolute constant $C>0$, by the classical evaluation in \cite[Theorem 2.7]{MontgomeryVaughan2007}.
We also recall that, by Proposition \ref{prop: small-sum}, we have for $k \geq 3$ and $y_1\geq k^3$ the bound
\[ \mathbb{E}\big(S_{A_{y_1}^{y_2}}^{2k}\big) \ll y_1^{-k/21}. \]

Let $\delta(y_1) = \delta/(4e^\gamma\log y_1+C)$, and let 
\[ k = \bigg\lceil \frac{\delta(y_1)^2y_1^{1/3}}{\log y_1} \bigg\rceil. \]
Then, for suitably sufficiently large $y_1\geqslant y_0(\delta)$, the conditions $k\geqslant 3$, $y_1\geqslant k^3$ and $\delta(y_1)>y_1^{-1/11}$ are satisfied, hence
\[ \mathbb{P}(S_{A_{y_1}^{y_2}} > \delta(y_1))
    \leq \frac{\mathbb{E}\big(S_{A_{y_1}^{y_2}}^{2k}\big)}{\delta(y_1)^{2k}}
    \leq \bigg( \frac{\delta(y_1)^{-1}}{y_1^{1/11}} \bigg)^{2\delta(y_1)^2y_1^{1/3}/\log y_1}. \]
    Thus, we have 
\begin{align*}
    \mathbb{P}(\| R_{y_1,y_2} \|_\infty > \delta) &\leqslant \mathbb{P}(S_{A_{y_1}^{y_2}} > \delta(y_1))\ll \bigg(\frac{4e^\gamma\log y_1+ C}{y_1^{1/11}\delta}\bigg)^{\frac{2y_1^{1/3}\delta^2}{\log y_1(4e^\gamma\log y_1 + C)^2}} \\
    &\ll \exp\bigg( -\frac{y_1^{1/3}\delta^2}{8e^{2\gamma}\log^3y_1 + \mathrm{O}(\log^2y_1)} \bigg(\log \frac{y_1^{1/11}\delta}{4e^\gamma\log y_1 + \mathrm{O}(1)} \bigg) \bigg) \\
    &\ll \exp\bigg( - \frac{y_1^{1/3}\delta^2}{\log^3y_1}\bigg) \ll \exp\big( -y_1^{1/6}\delta^2\big)
\end{align*}
for sufficiently large $y_1\geqslant y_0(\delta)$ (adjusting the value of $y_0(\delta)$ if needed).
\end{proof}

We now prove Proposition~\ref{thm: convergence}. 
\begin{proof}
First, we claim that, for every $\delta>0$, 
\begin{equation}
\label{uniform-over-y2}
\mathbb{P}\Big(\sup_{y_2>y_1}\|R_{y_1,y_2}(t)\|_{\infty}>\delta\Big)\ll\exp\big(-\delta^2y_1^{1/7}\big)
\end{equation}
holds for all sufficiently large $y_1>y_0(\delta)$. Indeed, let $C>0$ be the constant provided by Lemma~\ref{ry1y2-lemma}. Let $\delta'=\min(\delta,1)$, $z_1=\max(y_1,\lceil(\delta'/2)^{-C}\rceil)$, and, for $n\geqslant 2$, let
\[ z_n=\lceil z_1(\delta' 2^{-n})^{-C}\rceil. \]
This choice ensures that $z_n\geqslant\lceil(\delta' 2^{-n-1})^{-C}\rceil$ and $z_{n+1}-z_n\asymp z_1(\delta' 2^{-n-1})^{-C}$, and hence by Lemma~\ref{ry1y2-lemma}
\[ \mathbb{P}\Big(\sup_{z_n\leqslant z<z_{n+1}}\|R_{z_n,z}(t)\|_{\infty}>\delta' 2^{-n}\Big)
\ll (\delta' 2^{-n})^{-2C}\exp\big(-z_1^{1/6}(\delta' 2^{-n})^{2-C/6}\big). \]
From this it follows that, for every sufficiently large $y_1>y_0(\delta)$,
\[ \mathbb{P}\Big(\sup_{y_2>y_1}\|R_{y_1,y_2}(t)\|_{\infty}>\delta\Big)\leqslant\sum_{n=1}^{\infty}\mathbb{P}\Big(\sup_{z_n\leqslant y<z_{n+1}}\|R_{z_n,y}\|_{\infty}>\delta' 2^{-n}\Big)\ll\exp\big(-\delta^2y_1^{1/7}\big), \]
as claimed.

Now, let $\delta>0$ be arbitrary. The bound \eqref{uniform-over-y2} combined with $R_{y_1,y_2}=S_{y_2}-S_{y_1}$ shows that
\[ \sum_{y_1 =1}^\infty \mathbb{P}\bigg( \sup_{y_2 > y_1} \| S_{y_2} - S_{y_1}\|_\infty > \delta \bigg) < \infty. \]
By the Borel--Cantelli Lemma (see, for example, \cite[Proposition I.2]{LiQueffelec2018}), this implies that almost surely only finitely many events on the left-hand side occur; in other words, almost surely
\[ \sup_{y_2 > y_1} \| S_{y_2} - S_{y_1}\|_\infty \leqslant \delta \]
holds for all sufficiently large $y_1$. But this means exactly that the sequence $(S_y)$ almost surely converges uniformly. Since each function $S_y$ is continuous, their a.s.\ uniform limit $\widetilde{G^{\ast}}(t)=\lim_{y\to\infty}S_y(t)$ in \eqref{arithmetic-limit-equation} is also a.s.\ a continuous function.
\end{proof}

\subsection{Properties of the limiting random variable}
\label{properties-limiting-RV-sec}
We defined $\widetilde{G^{\ast}}(t)$ in \eqref{arithmetic-limit-equation} as the limit of sum over $\{n\neq -1,0:P^{+}(|n|)\leqslant y\}$. Due to the uniform convergence we just established, $\widetilde{G^{\ast}}(t)$ defines almost surely a continuous function such that the $n$th Fourier coefficient of $\widetilde{G^{\ast}}(t)-t$ (for $n\neq 0,1$) is precisely $X_n/2\pi in=\mathrm{O}(1/n)$.

Now, it is true that, for every $f\in C(\mathbb{R}/\mathbb{Z})$ such that $\hat{f}(n)=\OO(1/n)$, the partial sums of its Fourier series $S_n(f)$ converge uniformly to $f$. For completeness, we reproduce the argument, which we learned from Ullrich~\cite{stackexchange}. Fix for now an arbitrary $\epsilon>0$, let $\psi_{\epsilon}\in C_c^{\infty}(\mathbb{R})$ be a ``trapezoidal'' bump function satisfying $0\leqslant\psi_{\epsilon}\leqslant 1$, $\psi_{\epsilon}(x)=1$ for $x\in [-1,1]$ and $\psi_{\epsilon}(x)=0$ for $|x|\geqslant 1+\epsilon$, and let $F_{\epsilon}$ be its Fourier transform. For $n\in\mathbb{N}$, consider the Schwartz class function $\delta_nF_{\epsilon}(t):=nF_{\epsilon}(nt)$ and its periodization $K_{n,\epsilon}\in C^{\infty}(\mathbb{R}/\mathbb{Z})$ defined by $K_{n,\epsilon}(t)=\sum_{k\in\mathbb{Z}}\delta_nF_{\epsilon}(t+2\pi k)$. On the one hand, using the uniform continuity of $f$, the integrability of $F_{\epsilon}$, and unfolding, we have that
\begin{align*}
f(x)-(K_{n,\epsilon}\star f)(x)&=\bigg(\int_{|h|\leqslant\delta}+\int_{|h|>\delta}\bigg)(f(x)-f(x-h))\delta_nF_{\epsilon}(h)\,\dd h\\
&<\max_{\substack{x\in\mathbb{R}/\mathbb{Z}\\|h|\leqslant\delta}}|f(x)-f(x-h)|\cdot\|F_{\epsilon}\|_1+2\|f\|_{\infty}\cdot\int_{|h|>n\delta}|F_{\epsilon}|<\epsilon
\end{align*}
for sufficiently small $\delta>0$ and then sufficiently large $n\in\mathbb{N}$, uniformly in $x\in\mathbb{R}/\mathbb{Z}$. On the other hand, $K_{n,\epsilon}$ is the trigonometric polynomial $K_{n,\epsilon}(t)=\sum_{j\in\mathbb{Z}}\psi_{\epsilon}(j/n)e(nt)$, and so by the trivial bounds we have that
\[ |S_n(f)(x)-(K_{n,\epsilon}\star f)(x)|\ll\sum_{n\leqslant |j|\leqslant(1+\epsilon) n}\frac1j\ll\epsilon \]
uniformly in $x\in\mathbb{R}/\mathbb{Z}$. Thus $\|S_n(f)-f\|_{\infty}\ll\epsilon$ for sufficiently large $n\geqslant n_0(\epsilon)\in\mathbb{N}$, which precisely establishes that $S_n(f)\,\rightrightarrows\,f$ on $\mathbb{R}/\mathbb{Z}$.

Using the just established fact, the almost surely continuous function $\widetilde{G^{\ast}}(t)-t$ equals {a.s.} the limit of the usual partial sums as in \eqref{limiting-random-partial-sums} and the two definitions will be equivalent, that is,
\begin{equation}
\label{equality-tilde}
G^{\ast}(t)=\widetilde{G^{\ast}}(t)\,\,\text{a.s.}
\end{equation}
Therefore, we use these definitions interchangeably throughout the paper.

Now, let $\Omega=\prod_p\Omega_p$ denote the probability space underlying the sequence $(X_p)$ of independent random variables; that is, $\Omega_p=\{0,1,-1\}$ equipped with the measure $\lambda_p$ appearing in \eqref{eq: lambda-p} for odd prime $p$, and $\Omega_2=\{-1,1\}$ with $\lambda_2$ as in \eqref{eq: lambda-p}. Then, for every $h\in\mathbb{Z}\setminus\{0\}$ we have
\begin{equation}
\label{expectations-eta-def}
\mathbb{E}(X_h)=\begin{cases}\eta(h),&|h|=\square;\\0,&\text{otherwise},\end{cases}\qquad \eta(h)=\prod_{p\mid |h|,\,p>2}\frac{p}{p+1}.
\end{equation}
We also formally set $\eta(0)=0$.

The a.s.\ convergent series $G^{\ast}(t)$ defines a $C^0([0,1],\mathbb{C})$-valued random variable on $\Omega$. The main result of this section, the following Lemma~\ref{evaluation-of-limit-moment}, shows that, for arbitrary fixed $k\in\mathbb{Z}_{\geqslant 0}$ and $\bm{t}=(t_1,\dots,t_k)\in[0,1]^k$,  the $\mathbb{C}^k$-valued random variable $(G^{\ast}(t_1),\dots,G^{\ast}(t_k))$ has complex moments of all orders $\bm{m},\bm{n}\in\mathbb{Z}_{\geqslant 0}^k$:
\begin{equation}
\label{M-ast-moment-def}
\mathcal{M}^{\ast}(\bm{t};\bm{m},\bm{n})=\mathbb{E}\bigg(\prod_{i=1}^k\overline{G^{\ast}(t_i)}^{m_i}G^{\ast}(t_i)^{n_i}\bigg),
\end{equation}
and provides an exact evaluation for these orders. To state the result precisely, the following notations are convenient. Let $\mathcal{H}^{\ast}_{\bm{m},\bm{n}}$ denote the set of all $k$-tuples of vectors $\vec{\bm{h}}=(\vec{h}_1,\dots,\vec{h}_k)$, with each individual $\vec{h}_j=(h_{j, 1}, \dots, h_{j, n_j}, h_{j, n_j+1}, \dots, h_{j, n_j+m_j})\in\mathbb{Z}^{n_j+m_j}$. For every $\vec{\bm{h}}\in\mathcal{H}^{\ast}_{\bm{m},\bm{n}}$ and $\bm{t}\in[0,1]^k$, define
\begin{equation}
\label{H-beta-def}
\begin{aligned}
H(\vec{\bm{h}})&=\prod_{j=1}^k\prod_{\ell=1}^{n_j+m_j}(1-h_{j,\ell}), \\
\beta(\vec{\bm{h}};\bm{t})&=\prod_{j=1}^k\beta(\vec{h}_j;t_j),\quad \beta(\vec{h}_j;t_j)=\prod_{\ell=1}^{n_j}\beta(h_{j,\ell};t_j)\prod_{\ell=1}^{m_j}\overline{\beta(h_{j,n_j+\ell};t_j)}, 
\end{aligned}
\end{equation}
where, for every $h\in\mathbb{Z}$ and $t\in [0,1]$,
\[ \beta(h; t) = \begin{cases}\frac{e(ht) -1}{2 i \pi h}, &h\neq 0,\\ t,&h=0.\end{cases} \]
Then we have the following result.

\begin{lem}
\label{evaluation-of-limit-moment}
For every $t\in [0,1]$, the random variable $G^{\ast}(t)\in\bigcap_{p<\infty}L^p(\Omega)$. Moreover, for every $k\in\mathbb{Z}_{\geqslant 0}$ and every $\bm{t}=(t_1,\dots,t_k)\in [0,1]^k$, the $\mathbb{C}^k$-valued random variable $G^{\ast}(\bm{t}):=((G^{\ast}(t_1),\dots,G^{\ast}(t_k))$ has complex moments $\mathcal{M}^{\ast}(\bm{t};\bm{m},\bm{n})$ as in \eqref{M-ast-moment-def} of all orders $\bm{m},\bm{n}\in\mathbb{Z}_{\geqslant 0}^k$, given by the absolutely convergent sum
\[ \mathcal{M}^{\ast}(\bm{t};\bm{m},\bm{n})=\sum_{\substack{\vec{\bm{h}}\in\mathcal{H}^{\ast}_{\bm{m},\bm{n}}\\|H(\vec{\bm{h}})|=\square}}\beta(\vec{\bm{h}};\bm{t})\eta(H(\bm{h})), \]
with $\beta(\vec{\bm{h}};\bm{t})$ and $\eta(H(\vec{\bm{h}}))$ as in \eqref{expectations-eta-def} and \eqref{H-beta-def}. These moments satisfy $\mathcal{M}^{\ast}(\bm{t};\bm{m},\bm{n})\leqslant C^{m+n}$ for a suitable absolute $C>0$, and so $G^{\ast}(\bm{t})$ is a mild random variable.
\end{lem}

\begin{proof}
Returning to \eqref{Gtilde-decomp}, we denote
\[ S^{\ast}_y(t) := \sum_{\substack{n \neq 0 \\ P^+(|n|) \leq y}} \frac{e((n+1)t) -1}{2\pi i(n+1)} X_n, \quad S^{\ast}_y(\Epsilon; t) = \sum_{\substack{n \neq 0 \\ P^+(|n|) \leq y}} \frac{e((n+1)t)-1}{2\pi i(n+1)} \epsilon_n, \]
with the terms corresponding to $n=-1$ formally interpreted as $t$. In this paragraph, we verify that all results of the previous section remain valid with $S^{\ast}_y(t)$, $S^{\ast}_y(\Epsilon;t)$, and analogously defined $R^{\ast}_{y_1,y_2}(t)$ in place of $S(t)$, $S(\Epsilon;t)$ and $R_{y_1,y_2}(t)$. Indeed, for any choice of $|\delta_n|\leqslant 1$, we may repeat the full proof of Proposition~\ref{prop: small-sum} with the quantities
\[ S_{A,\bm{\delta}}^{\ast}:=\max_{t\in [0,1]}\left|\sum_{n\in A}\frac{e(nt)}{n+\delta_n}X_n\right|,\quad S_{A,\bm{\delta}}^{\ast}(\Epsilon):=\max_{t\in [0,1]}\left|\sum_{n\in A}\frac{e(nt)}{n+\delta_n}\epsilon_n\right|, \]
instead of \eqref{eq: S-A}, with the only substantive change in the proof being that $\tilde{d}_k(n;N)$ needs to be replaced with
\[ \tilde{d}_k^{\ast}(n)=\sum_{\substack{n_1\cdots n_k=n\\n_1,\dots,n_k\in A_{y_1}^{y_2}(N)}}\prod_{j=1}^k\Big(1+\frac{\delta_{n_j}}{n_j}\Big)^{-1}\prod_{j=1}^ke(n_jr/R). \]
This satisfies the estimate $|\tilde{d}_k^{\ast}(n;N)|\leqslant e^{\OO(k)}d_k(n;N)$, which is all that is needed for the proof. As in \eqref{Syt}, $S^{\ast}_y(\Epsilon;t)$ converges absolutely and uniformly to a continuous function. In place of \eqref{Ry1y2-decomp}, we have the decomposition
\[ R^{\ast}_{y_1,y_2}(t)=\sum_{\substack{n\neq 0\\ P^{+}(|n|)\leqslant y_1}}\frac{X_n}n\sum_{\substack{y_1<P^{-}(m)\leqslant P^{+}(m)\leqslant y_2\\mn\neq -1}}\frac{e(mnt)-e(-t)}{m+1/n}X_m. \]
Therefore, denoting $\delta'(m)=1/n$, Lemma~\ref{ry1y2-lemma} remains valid for $R^{\ast}_{y_1,y_2}$ as stated, with the key estimate \eqref{key-to-Ry1y2} in the proof replaced by
\[ \bigg|\sum_{y_1<P^{-}(m)\leqslant P^{+}(m)\leqslant y_2}\frac{e(mnt)-e(-t)}{m+1/n}\epsilon_m\bigg|\leqslant 2S_{A_{y_1}^{y_2},\bm{\delta}'}^{\ast}(\Epsilon), \]
for which the newly adjusted Proposition~\ref{prop: small-sum} provides $\mathbb{E}(S_{A_{y_1}^{y_2},\bm{\delta}'}^{\ast})\ll y_1^{-k/21}$, and the rest of the proof is unchanged.

Using the estimate \eqref{uniform-over-y2} and the fact that $S_z^{\ast}\,\rightrightarrows\,\widetilde{G^{\ast}}$ (uniformly in $t$, as $z\to\infty$) almost surely, we conclude that
\[ \mathbb{P}\Big(\sup_{z_2>z_1\geqslant y}\|S^{\ast}_{z_2}-S^{\ast}_{z_1}\|_{\infty}>1\Big),\,\,
\mathbb{P}\Big(\sup_{z\geqslant y}\|\widetilde{G^{\ast}}-S^{\ast}_z\|_{\infty}>1\Big)\ll\exp(-y^{1/7}/4) \]
holds for all sufficiently large $y$. Letting
\[ E_n=\Big\{\mathcal{E}\in\Omega:\sup_{z\geqslant n}\|(\widetilde{G^{\ast}}-S^{\ast}_z)(\mathcal{E})\|_{\infty}\leqslant 1\Big\}, \]
we have that $(E_n)$ form a (non-strictly) increasing sequence of events with $\mathbb{P}(E_n)\geqslant 1-C\exp(-n^{1/7}/4)$, so that $\chi_{E_n}\,\nearrow\,1$ a.s. In particular, for every fixed $t\in [0,1]$, we have by the Monotone Convergence Theorem for every $p\geqslant 0$
\begin{equation}
\label{upperbound-gtilde-p}
\mathbb{E}\big(\big((|\widetilde{G^{\ast}}(t)|+1)^{2p}\big)=\lim_{n\to\infty}\mathbb{E}\big(\big(|\widetilde{G^{\ast}}(t)|+1)^{2p}\chi_{E_n}\big)\leqslant 3^{2p}\big(\liminf_{n\to\infty}\mathbb{E}(|S^{\ast}_n(t)|^{2p})+1\big).
\end{equation}
Now, for every $n,p\in\mathbb{N}$, we have by rapid convergence
\begin{equation}
\label{upperbound-sn-p}
\begin{aligned}
\mathbb{E}(|S^{\ast}_n(t)|^{2p})&=\mathop{\sum\dots\sum}_{\substack{h_1,\dots,h_{2p}\neq 0\\P^{+}(|h_1\cdots h_{2p}|)\leqslant n\\h_1\cdots h_{2p}=\square}}\prod_{j=1}^p\frac{e((h_j+1)t)-1}{2\pi i(h_j+1)}\frac{e(-(h_{j+p}+1)t)-1}{2\pi i(h_{j+p}+1)}\eta(h_1\cdots h_{2p})\\
&\leqslant 8^p\sum_{m=1}^{\infty}\frac{d_{2p}(m^2)}{m^2}.
\end{aligned}
\end{equation}
Now, for every prime $q$, we have by a simple combinatorial argument that $d_{2p}(q^k)=\binom{2p+k-1}k$, and so
\[ \sum_{k=0}^{\infty}\frac{d_{2p}(q^{2k})}{q^{2k}}=\sum_{k=0}^{\infty}\binom{-1-2p}{2k}\frac1{q^{2k}}=\frac12\Big[\Big(1-\frac1{q^2}\Big)^{-1-2p}+\Big(1+\frac1{q^2}\Big)^{-1-2p}\Big]=e^{\mathrm{O}(p/q^2)}, \]
whence
\begin{equation}
\label{d2-estimate}
\sum_{n=1}^{\infty}\frac{d_{2p}(m^2)}{m^2}=\exp\Big(\sum_q\mathrm{O}\Big(\frac{p}{q^2}\Big)\Big)=e^{\mathrm{O}(p)}.
\end{equation}
Inserting this into \eqref{upperbound-sn-p} and \eqref{upperbound-gtilde-p} completes the proof of
\begin{equation}
\label{Lpnorm-estimate}
\widetilde{G^{\ast}}(t)\in\bigcap_{p<\infty}L^p(\Omega),\quad \|\widetilde{G^{\ast}}(t)\|_p^p=e^{\mathrm{O}(p)},
\end{equation}
where the values of $p\not\in 2\mathbb{Z}_{\geqslant 0}$ are covered by an interpolation argument.

We also claim that $S^{\ast}_n(t)\to\widetilde{G^{\ast}}(t)$ in $L^p(\Omega)$. We proceed by a similar argument. Let $\varepsilon>0$ and $n\in\mathbb{N}$ be arbitrary, and denote
\[ E_{m,\varepsilon}=\Big\{\mathcal{E}\in\Omega:\sup_{z\geqslant m}\|(\widetilde{G^{\ast}}-S^{\ast}_z)(\mathcal{E})\|_{\infty}\leqslant\varepsilon\Big\}. \]
Then the same argument using \eqref{uniform-over-y2} as above shows that $(E_{m,\varepsilon})_m$ form a (non-strictly) increasing sequence of events with $\mathbb{P}(E_{m,\varepsilon})\geqslant 1-C\exp(-\varepsilon^2m^{1/7})$, so that $\chi_{E_{m,\varepsilon}}\,\nearrow\,1$ a.s. as $m\to\infty$, and thus for every $p\geqslant 1$
\begin{align*}
\mathbb{E}\big(|\widetilde{G^{\ast}}(t)-S^{\ast}_n(t)|^{2p}\big)
&=\lim_{m\to\infty}\mathbb{E}\big(|\widetilde{G^{\ast}}(t)-S^{\ast}_n(t)|^{2p}\chi_{E_{m,\varepsilon}}\big)\\
&\leqslant 2^{2p}\Big(\varepsilon^{2p}+\liminf_{m\to\infty}\mathbb{E}\big(|S^{\ast}_m(t)-S^{\ast}_n(t)|^{2p}\big)\Big).
\end{align*}
Now, for $m>n$ we have, arguing as in \eqref{upperbound-sn-p} and below, that
\begin{equation}
\label{ESmSn2p-estimate}
\mathbb{E}(|S^{\ast}_m(t)-S^{\ast}_n(t)|^{2p})\leqslant 8^p\sum_{P^{+}(k)>n}\frac{\tau_{2p}(k^2)}{k^2}\ll_{p,\epsilon}\sum_{q>n}\sum_{k=1}^{\infty}\frac1{(q^2k^2)^{1-\epsilon}}\ll_{p,\epsilon}\frac1{n^{1-\epsilon}}.
\end{equation}
Inserting this into the previous estimate, we conclude that
\[ \mathbb{E}\big(|\widetilde{G^{\ast}}(t)-S^{\ast}_n(t)|^{2p}\big)\ll_p\Big(\varepsilon^{2p}+\frac1{n^{0.99}}\Big). \]
Executing the limits as $\varepsilon\to 0$ and $n\to\infty$ (in either order), we conclude that indeed
\begin{equation}
\label{convergence-in-lp}
\lim_{n\to\infty}\mathbb{E}\big(|\widetilde{G^{\ast}}(t)-S^{\ast}_n(t)|^{2p}\big)=0.
\end{equation}
The same claim is true for all real values $p\geqslant 1$ by interpolation.

Finally, we turn to the complex moments $\mathcal{M}^{\ast}(\bm{t};\bm{m},\bm{n})$, which are all finite by H\"older's inequality. By writing $\widetilde{G^{\ast}}(t_i)^k=(\widetilde{G^{\ast}}(t_i)^k-S^{\ast}_n(t_i)^k)+S^{\ast}_n(t_i)^k$ ($k\in\{m_i,n_i\}$) and expanding the products and complex conjugates, we may write
\begin{equation}
\label{differences}
\begin{aligned}
&\prod_{i=1}^k\overline{\widetilde{G^{\ast}}(t_i)}^{m_i}\widetilde{G^{\ast}}(t_i)^{n_i}-\prod_{i=1}^k\overline{S^{\ast}_n(t_i)}^{m_i}S^{\ast}_n(t_i)^{n_i}=\sum_{\substack{\Sigma_1,\Sigma_2\subseteq\{1,\dots,k\}\\\Sigma_1\cup\Sigma_2\neq\emptyset}}\\
&\quad \prod_{i\in\Sigma_1}\big(\overline{\widetilde{G^{\ast}}(t_i)}^{m_i}-\overline{S^{\ast}_n(t_i)}^{m_i}\big)\prod_{i\not\in\Sigma_1}\overline{S^{\ast}_n(t_i)}^{m_i}\prod_{i\in\Sigma_2}\big(\widetilde{G^{\ast}}(t_i)^{n_i}-S^{\ast}_n(t_i)^{n_i}\big)\prod_{i\not\in\Sigma_2}S^{\ast}_n(t_i)^{n_i}.
\end{aligned}
\end{equation}
Since $\widetilde{G^{\ast}},S^{\ast}_n\in\bigcap_{p<\infty}L^p(\Omega)$, in fact with $\|S^{\ast}_n\|_p=\mathrm{O}_p(1)$ uniformly in $n$ in light of \eqref{upperbound-sn-p}, taking expectations on both sides, factoring the differences of powers, and applying H\"older's inequality and \eqref{convergence-in-lp}, we conclude that
\[ \mathcal{M}^{\ast}(\bm{t};\bm{m},\bm{n})=\mathbb{E}\bigg(\prod_{i=1}^k\overline{\widetilde{G^{\ast}}(t_i)}^{m_i}\widetilde{G^{\ast}}(t_i)^{n_i}\bigg)=\lim_{n\to\infty}\mathbb{E}\bigg(\prod_{i=1}^k\overline{S^{\ast}_n(t_i)}^{m_i}S^{\ast}_n(t_i)^{n_i}\bigg). \]
But this final expectation is straightforward to evaluate; indeed, denoting by $\mathcal{H}^{\ast\ast}_{\bm{m},\bm{n}}$ the set of all tuples $\vec{\bm{h}}=(\vec{h}_1,\dots,\vec{h}_k)$ with $\vec{h}_j=(h_{j,1},\dots,h_{j,n_j},h_{j,n_j+1},\dots,h_{j,n_j+m_j})$
and each $h_{j,\ell}\in\mathbb{Z}\setminus\{0\}$, and $\Pi(\vec{\bm{h}})=\prod_{j=1}^k\prod_{\ell=1}^{n_j+m_j}h_{j,\ell}$,
\begin{align*}
&\mathbb{E}\bigg(\prod_{i=1}^k\overline{S^{\ast}_n(t_i)}^{m_i}S^{\ast}_n(t_i)^{n_i}\bigg)\\
&\qquad=\sum_{\substack{\vec{\bm{h}}\in\mathcal{H}^{\ast\ast}_{\bm{m},\bm{n}}\\ P^{+}(|\Pi(\vec{\bm{h}})|)\leqslant n,\,\,|\Pi(\vec{\bm{h}})|=\square}}\prod_{j=1}^k\prod_{\ell=1}^{n_j}\frac{e((h_{j,\ell}+1)t_j)-1}{2\pi i(h_{j,\ell}+1)}\prod_{\ell=n_j+1}^{n_j+m_j}\frac{e(-(h_{j,\ell}+1)t_j)-1}{2\pi i(h_{j,\ell}+1)}\eta(\Pi(\vec{\bm{h}}))\\
&\qquad =\sum_{\substack{\vec{\bm{h}}\in\mathcal{H}^{\ast}_{\bm{m},\bm{n}}\\ P^{+}(|H(\vec{\bm{h}})|)\leqslant n,\,\,|H(\vec{\bm{h}})|=\square}}\beta(\vec{\bm{h}};\bm{t})\eta(H(\bm{\vec{h}})).
\end{align*}
Since, denoting $u=\sum_{i=1}^km_i+\sum_{i=1}^kn_i$, we comfortably have absolute convergence
\begin{equation}
\label{absolute-convergence-eq}
\sum_{\substack{\vec{\bm{h}}\in\mathcal{H}^{\ast}_{\bm{m},\bm{n}}\\|H(\vec{\bm{h}})|=\square}}|\beta(\vec{\bm{h}};\bm{t})\eta(H(\vec{\bm{h}}))|\leqslant 8^{u}\sum_{k=1}^{\infty}\frac{\tau_{u}(k^2)}{k^2}=\mathrm{O}_{\bm{m},\bm{n}}(1),
\end{equation}
this implies the announced evaluation
\[ \mathcal{M}^{\ast}(\bm{t};\bm{m},\bm{n})=\lim_{n\to\infty}\mathbb{E}\bigg(\prod_{i=1}^k\overline{S^{\ast}_n(t_i)}^{m_i}S^{\ast}_n(t_i)^{n_i}\bigg)=
\sum_{\substack{\vec{\bm{h}}\in\mathcal{H}^{\ast}_{\bm{m},\bm{n}}\\|H(\vec{\bm{h}})|=\square}}\beta(\vec{\bm{h}};\bm{t})\eta(H(\vec{\bm{h}})). \qedhere \]
\end{proof}

\section{Computing the moments}
\label{computing-moments-section}

In this section, we compute asymptotically the complex moments of $G_Q(\bm{t})$, which are given by
\begin{equation}
\label{MQ-moment-def}
\mathcal{M}_Q(\bm{t}; \bm{m}, \bm{n}) = \E \bigg( \prod_{i=1}^k \overline{G_Q(t_i)}^{m_i}G_Q(t_i)^{n_i} \!\bigg)= \!\!\sum_{c \in [Q, 2Q]\cap \cd} \!\!m_Q(c) \prod_{i=1}^k \overline{G(t_i; c)}^{m_i} G(t_i; c)^{n_i},
\end{equation}
where $k$ is a positive integer, $\bm{t} = (t_1, \dots, t_k)$ is a $k$-tuple in $[0, 1]^k$, and $\bm{n} = (n_1, \dots, n_k)$ and $\bm{m} = (m_1, \dots, m_k)$ are $k$-tuples of non-negative integers. Additionally, we denote $m = \sum_{i=1}^k m_i$ and $n = \sum_{i=1}^k n_i$. Specifically we will prove the following evaluation.

\begin{prop}
\label{moments-eval-prop}
For every positive integer $k$ and all $k$-tuples $\bm{t}\in [0,1]^k$ and $\bm{m},\bm{n}\in\mathbb{Z}_{\geqslant 0}^k$, the complex moments $\mathcal{M}_Q(\bm{t};\bm{m},\bm{n})$ of $G_Q(\bm{t})$ in \eqref{MQ-moment-def} satisfy
\[ \mathcal{M}_Q(\bm{t};\bm{m},\bm{n})=\mathcal{M}^{\ast}(\bm{t};\bm{m},\bm{n})+\mathrm{O}_{\epsilon}(Q^{-1/3+\epsilon}), \]
where $\mathcal{M}^{\ast}(\bm{t};\bm{m},\bm{n})$ are the corresponding complex moments of $G^{\ast}(\bm{t})$ in \eqref{M-ast-moment-def}.
\end{prop}

\begin{cor}
\label{finite-distributions-cor}
The sequence of $C^0([0,1],\mathbb{C})$-valued random variables $(G_Q)$ converges in the sense of finite distributions to $G^{\ast}$ as $Q\to\infty$.
\end{cor}

As a preliminary step, we compute by simple sieving that
\begin{equation}
\label{DQ-comput}
|\mathcal{D}_Q|=\sum_{\substack{\alpha^2\leqslant 2Q\\2\nmid\alpha}}\mu(\alpha)\sum_{\substack{c\in [Q,2Q]\\\alpha^2\mid c,\,c\equiv 1\bmod 4}}1=\frac{Q}4\sum_{\substack{\alpha^2\leqslant 2Q\\2\nmid\alpha}}\frac{\mu(\alpha)}{\alpha^2}+\mathrm{O}(Q^{1/2})=\frac{Q}{3\zeta(2)}+\mathrm{O}(Q^{1/2}),
\end{equation}
so that the uniform probability measure $m_Q$ on $\mathcal{D}_Q$ satisfies
\[ m_Q(c)=\frac{3\zeta(2)}{Q}+\mathrm{O}\Big(\frac1{Q^{3/2}}\Big). \]

\subsection{Reduction steps}
\label{reduction-subsec}
We also consider slightly different functions $\widetilde{G}(\cdot,c):[0,1]\to\mathbb{C}$ defined by
\[\widetilde{G}(t; c)= \frac{1}{\sqrt{c}} \sum_{1 \leq x \leq (c-1)t} \Leg{x}{c}e_c(x). \]
The functions $\widetilde{G}(\cdot,c)$ are discontinuous but they agree with the Gauss paths $G(\cdot,c)$ at the points $t=j/(c-1)$ $(0\leq j\leq c-1)$ and (as we will quickly see) stay very close to them, while being technically easier to work with. Indeed, we have the following expansion.

\begin{lem} \label{lem: g-squiggle-val} We have 
\[ \widetilde{G}(t; c) =  \frac1{c^{3/2}} \sum_{h \bmod{c}} G(1-h, c) \sum_{1 \leq x \leq (c-1)t} e_c(hx) \]
and
\[ \widetilde{G}(t;c)\ll\log c. \]
\end{lem}

\begin{proof}
This follows by the completion method, analogously to the case of Kloosterman paths \cite{KowalskiSawin2016,RicottaRoyer2018,MilicevicZhang2023} and discussed in some detail in \cite[Chapter 6]{Kowalski2021}. Indeed, the Parseval identity for the discrete Fourier transform  $f(x) = \Leg{x}{c}e_c(x)$ and $g(x) = 1_{1\leq x\leq j}(x)$ (where $j=\lfloor (c-1)t\rfloor$) gives
\begin{equation}
\label{eq: g-squiggle}
    \widetilde{G}(t; c)= \frac{1}{\sqrt{c}} \sum_{1 \leq x \leq (c-1)} f(x) \overline{g(x)} = \sqrt{c} \sum_{h=0}^{c-1} \widehat{f}(h) \overline{\widehat{g}(h)},
\end{equation}
and we compute
\[     \widehat{f}(h)= \frac{1}{c}\sum_{x=0}^{c-1}\Leg{x}{c}e_c(x)e_c(-hx)= \frac1c G(1-h, c), \quad
\overline{\widehat{g}(h)}= \frac{1}{c}\sum_{x=1}^j e_c(hx). \]
This proves the first identity. Now, inserting the classical bound for $h\neq 0$ (see, for example, \cite[Chapter 3]{Montgomery1994})
\begin{equation}
\label{sharp-cutoff-estimate}
\bigg|\sum_{x=1}^{j} e_c(hx)\bigg| = \bigg|e_c((j+1)h/2) \frac{\sin(\pi jh/c)}{\sin(\pi h/c)}\bigg|\leq\frac{1}{2||h/c||},
\end{equation}
where $|| h/c||$ is the distance between $h/c$ and the nearest integer, and bounding Gauss sums individually, we get
\[ \abs{\widetilde{G}(t; c)}= \bigg| \frac1{c^{3/2}}\sum_{h=0}^{c-1}G(1-h, c)\sum_{x=1}^j e_c(hx) \bigg|
    \leq \frac1c\bigg( j+ \sum_{0 < |h|< c/2} \frac{2c}{|h|}\bigg)\ll \log(c). \qedhere \]
\end{proof}

For every $\bm{t}\in[0,1]^k$ we may consider the complex moments of $\widetilde{G}(\cdot,c)$ given by
\begin{equation}
\label{modified-moments-equation}
\widetilde{\mathcal{M}_Q}(\bm{t}; \bm{m}, \bm{n}) := \sum_{c \in [Q, 2Q]\cap \cd} m_Q(c) \prod_{i=1}^k \overline{\widetilde{G}(t_i; c)}^{m_i} \widetilde{G}(t_i; c)^{n_i}.
\end{equation}

That the moments $\mathcal{M}_Q(\bm{t}; \bm{m}, \bm{n})$ and $\widetilde{\mathcal{M}_Q}(\bm{t}; \bm{m}, \bm{n})$ are very close will follow directly from the following lemma.

\begin{lem} \label{lem: moment-diff-bound} We have
\[\left| \prod_{i=1}^k \overline{G(t_i; c)}^{m_i} G(t_i; c)^{n_i} - \prod_{i=1}^k \overline{\widetilde{G}(t_i; c)}^{m_i} \widetilde{G}(t_i; c)^{n_i} \right| \ll \frac{\log^{(m+n)}(c)}{c^{1/2}}. \]
\end{lem}

\begin{proof}
We first note that
\begin{align*}
    &\prod_{i=1}^k \overline{G(t_i; c)}^{m_i} G(t_i; c)^{n_i} - \prod_{i=1}^k \overline{\widetilde{G}(t_i; c)}^{m_i} \widetilde{G}(t_i; c)^{n_i}=\sum_{\substack{\Sigma_1,\Sigma_2\subseteq\{1,\dots,k\}\\\Sigma_1\cup\Sigma_2\neq\emptyset}}\\
    &\qquad\prod_{i\in\Sigma_1}\big(\overline{G(t_i; c)}^{m_i} - \overline{\widetilde{G}(t_i; c)}^{m_i}\big)\prod_{i\not\in\Sigma_1}\overline{\widetilde{G}(t_i; c)}^{m_i}\prod_{i\in\Sigma_2}\big(G(t_i; c)^{n_i} - \widetilde{G}(t_i; c)^{n_i} \big)\prod_{i\not\in\Sigma_2}\widetilde{G}(t_i; c)^{n_i}.
\end{align*}

We now use the bounds
\begin{equation}
\label{simplebounds}
| G(t; c) - \widetilde{G}(t; c)| \leq \frac{1}{\sqrt{c}},\quad | G(t;c)|,\,|\widetilde{G}(t; c) | \ll \log(c).
\end{equation}
The first of these follows immediately from the definitions of $G(t; c)$ and $\widetilde{G}(t; c)$, and the second one follows from the first one and Lemma~\ref{lem: g-squiggle-val}. Using these bounds we conclude that
\[    \big| G(t_i; c)^{k} - \widetilde{G}(t_i; c)^{k} \big| = \Big| \left( G(t_i; c)- \widetilde{G}(t_i; c) \right) \sum_{\ell=1}^{k} \big( G(t_i; c)^{k-\ell} \widetilde{G}(t_i; c)^{\ell-1} \big) \Big|
    \ll \frac{1}{\sqrt{c}} \log(c)^{k-1}. \]
Consequently, we have for every choice of $\Sigma_1,\Sigma_2\subseteq\{1,\dots,k\}$ with $\Sigma_1\cup\Sigma_2\neq\emptyset$
\begin{align*}
    &\prod_{i\in\Sigma_1}\big(\overline{G(t_i; c)}^{m_i} - \overline{\widetilde{G}(t_i; c)}^{m_i}\big)\prod_{i\not\in\Sigma_1}\overline{\widetilde{G}(t_i; c)}^{m_i}\prod_{i\in\Sigma_2}\big(G(t_i; c)^{n_i} - \widetilde{G}(t_i; c)^{n_i} \big)\prod_{i\not\in\Sigma_2}\widetilde{G}(t_i; c)^{n_i}\\
    &\qquad\ll\prod_{i\in\Sigma_1}\Big(\frac1{\sqrt{c}}\log(c)^{m_i-1}\Big)\prod_{i\not\in\Sigma_1}\log(c)^{m_i}\prod_{i\in\Sigma_2}\Big(\frac1{\sqrt{c}}\log(c)^{n_i-1}\Big)\prod_{i\not\in\Sigma_2}\log(c)^{n_i}\\
    &\qquad\ll \frac{1}{\sqrt{c}}\log(c)^{m+n},
\end{align*}
and the claim follows.
\end{proof}

Having proven Lemma \ref{lem: moment-diff-bound}, it is clear that
\begin{equation}
\label{tilde-non-tilde-moments-eq}
   \left| \widetilde{\mathcal{M}_Q}(\bm{t}; \bm{m}, \bm{n}) - \mathcal{M}_Q(\bm{t}; \bm{m}, \bm{n}) \right| \ll \sum_{c \in [Q, 2Q] \cap \cd} m_Q(c) c^{-1/2}\log(c)^{m+n} \\
    \ll_{\epsilon} Q^{-1/2+\epsilon},
    \end{equation}
so that we may from now on focus on an asymptotic evaluation of $\widetilde{\mathcal{M}_Q}(\bm{t}; \bm{m}, \bm{n})$.

We first rewrite the first conclusion of Lemma \ref{lem: g-squiggle-val} as
\begin{equation}
\label{eq: g-squiggle-alpha}
    \widetilde{G}(t; c)= \frac1c \sum_{h \bmod{c}} G(1-h, c) \alpha_c(h; t), \quad \alpha_c(h; t) := \frac{1}{\sqrt{c}}
    \sum_{1 \leq x \leq (c-1)t} e_c(hx).
\end{equation}
Inserting the expansion \eqref{eq: g-squiggle-alpha} into the definition \eqref{modified-moments-equation} and expanding, we write
\begin{align*}
    \widetilde{\mathcal{M}_Q}(\bm{t}; \bm{m}, \bm{n})
&= \sum_{c \in [Q, 2Q]\cap \cd} m_Q(c)\frac{1}{c^{m+n}} \prod_{j=1}^k \bigg( \overline{\sum_{h_j \bmod{c}}  G(1-h_j, c)\alpha_c(h_j; t_j)} \bigg)^{m_j} \\ &\qquad\times \bigg(\sum_{h_j \bmod{c}} G(1-h_j, c)\alpha_c(h_j; t_j)\bigg)^{n_j} \\
    &= \sum_{c \in [Q, 2Q]\cap \cd} m_Q(c)\frac{1}{c^{m+n}} \sum_{\Vec{h}_1} \dots \sum_{\Vec{h}_k} \alpha_c(\Vec{h}_1; t_1) \cdots \alpha_c(\Vec{h}_k; t_k)\\ &\qquad\times \prod_{\ell =1}^{n_1+m_1} G(1-h_{1, \ell}, c) \cdots \prod_{\ell=1}^{n_k+m_k}G(1-h_{k, \ell}, c).
\end{align*}
Here, $\Vec{h}_j$ ranges over all $(n_j+m_j)$-tuples $\vec{h}_j=(h_{j, 1}, \dots, h_{j, n_j}, h_{j, n_j+1}, \dots, h_{j, n_j+m_j})$, where each $h_{j, \ell}\in(-c/2, c/2)$, and  
\[ \alpha_c(\Vec{h}_j; t_j) = \prod_{\ell=1}^{n_j}\alpha_c(h_{j, \ell};t_j) \prod_{\ell =1}^{m_j} \overline{\alpha_c(h_{j, n_j+\ell};t_j)}.\]

Now, write $\mathcal{H}_c$ for the set of all such $k$-tuples $\vec{\bm{h}}=(\vec{h}_1,\dots,\vec{h}_k)$ and denote
\[ \alpha_c(\vec{\bm{h}};\bm{t})=\prod_{j=1}^k\alpha_c(\vec{h}_j;t_j), \]
recalling also the notations $H(\vec{\bm{h}})$ and $\beta(\vec{\bm{h}};\bm{t})$ for every $\vec{\bm{h}}\in\mathcal{H}\supseteq\mathcal{H}_c$ from \eqref{H-beta-def}.
Recalling that $G(1-h_{j, \ell}, c) = \big(\frac{1-h_{j, \ell}}{c}\big)\sqrt{c}$ and using multiplicativity of Jacobi symbols, we get
\begin{align}
 \label{eq:moment-alpha}
    \widetilde{\mathcal{M}_Q}(\bm{t}; \bm{m}, \bm{n})=  \sum_{c \in [Q, 2Q]\cap \cd} m_Q(c) \frac{1}{c^{(m+n)/2}} \sum_{\vec{\bm{h}}\in\mathcal{H}_c} \alpha_c(\vec{\bm{h}};\bm{t}) \bigg(\frac{H(\vec{\bm{h}})}{c}\bigg).
\end{align}

Next, we show that $\alpha_c(\vec{\bm{h}}; \bm{t})$ may be replaced with $\beta(\vec{\bm{h}};\bm{t})$, an expression independent of $c$, at the cost of a negligible error.  We begin with the following elementary lemma, for an independent proof of which we refer to \cite[Section 2]{KowalskiSawin2016} (where the condition that $c=p$ is a prime is clearly immaterial). We will provide a different argument, which also sets the stage for the proof of the next Lemma~\ref{alpha-beta-sum}.

\begin{lem}
\label{diff-coeff-small}
For $|h| < c/2$, we have
\[ \left| \frac{\alpha_c(h; t)}{\sqrt{c}} - \beta(h; t) \right| \ll \frac{1}{c}. \]
\end{lem}

\begin{proof}
We may write
\[ f(h)=\frac1c\sum_{1\leqslant x\leqslant (c-1)t}e_c(hx)-\int_0^te(h\xi)\,\dd\xi=\sum_{1\leqslant x\leqslant (c-1)t}e_c(hx)F_t(h,c)+G_t(h,c), \]
where
\[ F_t(h,c)=\int_{-1/c}^0(1-e(h\eta))\,\dd\eta,\quad F_t(h,c)\ll\frac h{c^2},\quad G_t(h,c)\ll\frac1c. \]

Now, if $h=0$, the statement of the lemma is trivially true. Otherwise, noting that $F_t(h,c)$ does not depend on $x$, and inserting the classical bound \eqref{sharp-cutoff-estimate}, we conclude that, for $|h|<c/2$,
\[ f(h)=\bigg(\sum_{1\leqslant x\leqslant (c-1)t}e_c(hx)\bigg)F_t(h,c)+G_t(h,c)\ll\frac1{|h|/c}\frac{|h|}{c^2}+\frac1c\ll\frac1c. \qedhere \]
\end{proof}

\begin{lem}
\label{alpha-beta-sum}
For $|h|<c/2$ and $t\in [0,1]$,
\begin{equation}
     \bigg| \sum_{|h| <c/2} \bigg(\frac{1-h}c\bigg) \left(\frac{\alpha_c(h; t)}{\sqrt{c}} - \beta(h; t)\right)\bigg| \ll \frac{\log^2c}{\sqrt{c}}.
\end{equation}
\end{lem}

\begin{proof}
On the one hand, by completion (that is essentially by the P\'olya--Vinogradov inequality), we have as in the proof of Lemma~\ref{lem: g-squiggle-val} for every $x\in\mathbb{Z}$, $\xi\geqslant 0$ the bound
\begin{equation}
\label{Ax-xi-eq}
A_x(\xi):=\sum_{1\leqslant h\leqslant\xi}\bigg(\frac{1-h}c\bigg)e_c(hx)\ll c^{1/2}\log c.
\end{equation}

On the other hand, we may insert the representation for $f(h)$ from the proof of Lemma \ref{diff-coeff-small}, noting that, moreover, $F_t(\cdot,c)$ and $G_t(\cdot,c)$ define functions of a continuous real variable that satisfy
\[ \frac{\dd}{\dd\xi}F_t(\xi,c)\ll\frac1{c^2},\quad \frac{\dd}{\dd\xi}G_t(\xi,c)\ll\frac1c. \]
Moreover, denoting $x_t^{+}=\lfloor (c-1)t\rfloor +1$, we may write
\[ \sum_{1\leqslant x\leqslant (c-1)t}e_c(hx)=\frac{e_c(hx_t^{+})-e_c(h)}{e_c(h)-1}. \]

Using summation by parts, this leads to
\begin{align*}
&\sum_{1\leqslant h<c/2} \bigg(\frac{1-h}c\bigg) \left(\frac{\alpha_c(h; t)}{\sqrt{c}} - \beta(h; t)\right)\\
&\qquad=\sum_{1\leqslant h<c/2}\bigg(\frac{1-h}c\bigg)\big(e_c(hx_t^{+})-e_c(h)\big)\frac{F_t(h,c)}{e_c(h)-1}+\sum_{1\leqslant h<c/2}\bigg(\frac{1-h}c\bigg)G_t(h,c)\\
&\qquad=\sum_{1\leqslant x\leqslant (c-1)t}\int_{1-}^{c/2}F_t(\xi,c)\,\dd\big(A_{x_t^{+}}(\xi)-A_1(\xi)\big)+\int_{1-}^{c/2}G_t(\xi,c)\,\dd A_0(\xi)\\
&\qquad=\bigg(\frac{F_t(\xi,c)}{e_c(\xi)-1}\big(A_{x_t^{+}}(\xi)-A_1(\xi)\big)+G_t(\xi,c)A_0(\xi)\bigg)\bigg|_{1-}^{c/2}\\
&\qquad\qquad-\int_1^{c/2}\big(A_{x_t^{+}}(\xi)-A_1(\xi))\frac{\dd}{\dd\xi}\frac{F_t(\xi,c)}{e_c(\xi)-1}\,\dd\xi-\int_1^{c/2}A_0(\xi)\frac{\dd}{\dd\xi}G_t(\xi,c)\,\dd\xi\\
&\qquad\ll\frac{\log c}{\sqrt{c}}+\sqrt{c}\log c\int_1^{c/2}\frac{\dd\xi}{c\xi}\ll\frac{\log^2c}{\sqrt{c}}.
\end{align*}
The contributions from $-c/2<h<0$ are estimated analogously, and the $h=0$ term is trivially admissible.
\end{proof}

This leads to the following.
\begin{lem}
\label{moments-with-betas-lemma}
We have
\begin{align} \label{eq: new-moments}
  \widetilde{\mathcal{M}_Q}(\bm{t}; \bm{m}, \bm{n}) &=  \sum_{c \in [Q, 2Q] \cap \cd}  m_Q(c) \sum_{\vec{\bm{h}}\in\mathcal{H}_c} 
   \beta(\vec{\bm{h}}; \bm{t}) \bigg(\frac{H(\vec{\bm{h}})}c\bigg)+ O\bigg(\frac{\log^{m+n+1}Q}{Q^{1/2}}\bigg).
\end{align}
\end{lem}

\begin{proof}
Let $\Sigma_{k,\bm{m},\bm{n}}$ be the set of all pairs $(j,\ell)$ such that $1\leqslant j\leqslant k$ and $1\leqslant\ell\leqslant n_j+m_j$. By writing $\alpha_c(h_{j,\ell};t_j)/\sqrt{c}=(\alpha_c(h_{j,\ell},t_j)/\sqrt{c}-\beta(h_{j,\ell};t_j))+\beta(h_{j,\ell};t_j)$ and expanding the product, may write
\[ \alpha_c(\vec{\bm{h}};\bm{t})-\beta(\vec{\bm{h}};\bm{t})=\sum_{\emptyset\neq\Sigma\subseteq\Sigma_{k,\bm{m},\bm{n}}}\prod_{(j,\ell)\in\Sigma}\left(\frac{\alpha_c(h_{j,\ell};t_j)}{\sqrt{c}}-\beta(h_{j,\ell};t_j)\right)^{\ast}\prod_{(j,\ell)\in\Sigma_{k,\bm{m},\bm{n}}\setminus\Sigma}\beta(h_{j,\ell};t_j)^{\ast}, \]
where
\[ \beta(h_{j,\ell},t_j)^{\ast}=\begin{cases} \beta(h_{j,\ell};t_j),&1\leqslant\ell\leqslant n_j,\\ \overline{\beta(h_{j,\ell};t_j)},&n_j+1\leqslant\ell\leqslant m_j,\end{cases} \] 
and similarly for $(\alpha_c(h_{j,\ell},t_j)/\sqrt{c}-\beta(h_{j,\ell};t_j)^{\ast})$.

Using this expansion, we have that
\begin{align*}
&\widetilde{\mathcal{M}_Q}(\bm{t};\bm{m},\bm{n})-\sum_{c \in [Q, 2Q] \cap \cd}  m_Q(c) \sum_{\vec{\bm{h}}\in\mathcal{H}_c} 
   \beta(\vec{\bm{h}}; \bm{t}) \bigg(\frac{H(\vec{\bm{h}})}c\bigg)\\
&\qquad=\sum_{\emptyset\neq\Sigma\subseteq\Sigma_{k,\bm{m},\bm{n}}}
\sum_{c \in [Q, 2Q] \cap \cd}  m_Q(c)
\prod_{(j,\ell)\in\Sigma_{k,\bm{m},\bm{n}}\setminus\Sigma}\sum_{|h_{j,\ell}|<c/2}\left(\frac{1-h_{j,\ell}}{c}\right)\beta(h_{j,\ell};t_j)^{\ast}\\
&\qquad\qquad\times\prod_{(j,\ell)\in\Sigma}\sum_{|h_{j,\ell}|<c/2}\left(\frac{1-h_{j,\ell}}{c}\right)\left(\frac{\alpha_c(h_{j,\ell};t_j)}{\sqrt{c}}-\beta(h_{j,\ell};t_j)\right)^{\ast}.
\end{align*}
In this expression, we bound the sums over $h_{j,\ell}$ corresponding to $(j,\ell)\in\Sigma_{k,\bm{m},\bm{n}}\setminus\Sigma$ trivially, using $\beta(h;t)\ll 1/(1+|h|)$, and those corresponding to $(j,\ell)\in\Sigma$ using Lemma~\ref{alpha-beta-sum}. This gives the announced estimate.
\end{proof}

We would now like to change the order of summation, so that the $c$-sum is inside of the $\vec{\bm{h}}$ sums. We now show that the limits of summation can indeed be changed from $|h_{j, \ell}| < c/2$ to the range $|h_{j, \ell}| < Q/2$ independent of $c$, while only introducing negligible error.

\begin{lem}
\label{moments-order-switched}
Let $\mathcal{H}'$ be the set of all $k$-tuples $\vec{\bm{h}}\!=\!(\vec{h}_1,\dots,\vec{h}_k)$ of vectors $\vec{h}_j\!=\!(h_{j,\ell})_{1\leqslant\ell\leqslant n_j+m_j}$ satisfying $|h_{j,\ell}|<Q/2$. Then,
\begin{align}
    \widetilde{\mathcal{M}_Q}(\bm{t}; \bm{m}, \bm{n}) =   \sum_{\Vec{\bm{h}}\in\mathcal{H}'} \beta(\vec{\bm{h}};\bm{t}) \sum_{c \in [Q, 2Q] \cap \cd}  m_Q(c)\bigg(\frac{H(\Vec{\bm{h}})}{c}\bigg) + \mathrm{O}\bigg(\frac{\log^{m+n+1}Q}{Q^{1/2}}\bigg).\end{align}
\end{lem}

\begin{proof}
Arguing as in the proof of Lemma~\ref{moments-with-betas-lemma}, we find that, for every $c\in [Q,2Q]\cap\mathcal{D}$,
\begin{align*}
\bigg(\sum_{\vec{\bm{h}}\in\mathcal{H}_c}
-\sum_{\vec{\bm{h}}\in\mathcal{H}'}\bigg)\beta(\vec{\bm{h}};\bm{t})\bigg(\frac{H(\Vec{\bm{h}})}{c}\bigg)
&=\sum_{\emptyset\neq\Sigma\subseteq\Sigma_{\bm{k},\bm{m},\bm{n}}}\prod_{(j,\ell)\in\Sigma_{\bm{k},\bm{m},\bm{n}}}\sum_{|h_{j,\ell}|<Q/2}\left(\frac{1-h_{j,\ell}}{c}\right)\beta(h_{j,\ell};t_j)\\
&\qquad\qquad\times\prod_{(j,\ell)\in\Sigma}\sum_{c/2<|h_{j,\ell}|<Q/2}\left(\frac{1-h_{j,\ell}}{c}\right)\beta(h_{j,\ell};t_j).
\end{align*}
We claim that
\[ \sum_{|h|<Q/2}\left(\frac{1-h}c\right)\beta(h;t)\ll\log Q,\quad \sum_{c/2<|h|<Q/2}\left(\frac{1-h}c\right)\beta(h;t)\ll\frac{\log Q}{Q^{1/2}}. \]
The first of these bounds is trivial from $\beta(h;t)\ll 1/(1+|h|)$.

For the second bound, we argue analogously as in the proof of Lemma~\ref{alpha-beta-sum}. Recall the notation and the P\'olya--Vinogradov bound for $A_x(\xi)$ from \eqref{Ax-xi-eq}, and denote $x_t^{\ast}=\lfloor ct\rfloor$. Then, we may rewrite the contribution of $c/2<h<Q/2$ to the second sum as
\[ \sum_{c/2<h<Q/2}\left(\frac{1-h}c\right)\big(e_c(hx_t^{\ast})-1\big)\frac1h+\sum_{c/2<h<Q/2}\left(\frac{1-h}c\right)e(hx_t^{\ast})f(h), \]
where
\[ f(\xi)=\frac{e(\xi(t-x_t^{\ast}/c))}{\xi},\quad \frac{\dd}{\dd\xi}f(\xi)\ll\frac1{Q^2}\quad (\xi\asymp Q). \]
Then summation by parts indeed yields
\begin{align*}
&\sum_{c/2<h<Q/2}\left(\frac{1-h}c\right)\big(e_c(hx_t^{\ast})-1\big)\frac1h=\int_{c/2}^{(Q/2)+}\frac1{\xi}\,\dd\big(A_{x_t^{\ast}}(\xi)-A_0(\xi)\big)\\
&\qquad\qquad=\frac{A_{x_t^{\ast}}(\xi)-A_0(\xi)}{\xi}\bigg|_{c/2}^{(Q/2)+}+\int_{c/2}^{Q/2}\frac{A_{x_t^{\ast}}(\xi)-A_0(\xi)}{\xi^2}\,\dd\xi\ll\frac{\log Q}{Q^{1/2}},\\
&\sum_{c/2<h<Q/2}\left(\frac{1-h}c\right)e(hx_t^{\ast})f(h)=\int_{c/2}^{(Q/2)+}f(h)\,\dd A_{x_t^{\ast}}(\xi)\\
&\qquad\qquad =f(h)A_{x_t^{\ast}}(\xi)\bigg|_{c/2}^{(Q/2)+}-\int_{c/2}^{Q/2}A_{x_t^{\ast}}(\xi)f'(\xi)\,\dd\xi\ll\frac{\log Q}{Q^{1/2}}.
\end{align*}

The terms with $-Q/2<h<-c/2$ are estimated analogously. Putting everything together, summing over all $c\in [Q,2Q]\cap\cd$ with weights $m_Q(c)$, and invoking Lemma~\ref{moments-with-betas-lemma}, we obtain
\[ \widetilde{\mathcal{M}_Q}(\bm{t};\bm{m},\bm{n})=\sum_{c\in [Q,2Q]\cap\cd}m_Q(c)\sum_{\vec{\bm{h}}\in\mathcal{H}'}\beta(\vec{\bm{h}};\bm{t})\bigg(\frac{H(\vec{\bm{h}})}{c}\bigg)+\mathrm{O}\bigg(\frac{\log^{m+n+1}Q}{Q^{1/2}}\bigg). \]
Since the range of summation in $\mathcal{H}'$ does not depend on $c$, we may exchange the order of summation, and the lemma follows.
\end{proof}

\subsection{Isolating the main term and proofs of main results}
In view of Lemma~\ref{moments-order-switched}, we may write
\begin{equation}
\label{moments-main-error-decomp}
\widetilde{\mathcal{M}_Q}(\bm{t};\bm{m},\bm{n})=\widetilde{\mathcal{M}_Q^0}(\bm{t};\bm{m},\bm{n})+\widetilde{\mathcal{M}_Q'}(\bm{t};\bm{m},\bm{n})+\mathrm{O}\bigg(\frac{\log^{m+n+1}Q}{Q^{1/2}}\bigg),
\end{equation}
where
\begin{align*}
\widetilde{\mathcal{M}_Q^0}(\bm{t};\bm{m},\bm{n})&=\sum_{\substack{\Vec{\bm{h}}\in\mathcal{H}'\\|H(\vec{\bm{h}})|=\square}} \beta(\vec{\bm{h}};\bm{t}) \sum_{\substack{c \in [Q, 2Q] \cap \cd\\(c,H(\vec{\bm{h}}))=1}}  m_Q(c),\\
\widetilde{\mathcal{M}_Q'}(\bm{t}; \bm{m}, \bm{n}) &=   \sum_{\substack{\Vec{\bm{h}}\in\mathcal{H}'\\|H(\vec{\bm{h}})|\neq\square}} \beta(\vec{\bm{h}};\bm{t}) \sum_{c \in [Q, 2Q] \cap \cd}  m_Q(c)\bigg(\frac{H(\Vec{\bm{h}})}{c}\bigg).
\end{align*}

The following subsection will be devoted to the proofs of the following two lemmata.

\begin{lem}
\label{main-term-lemma}
\[ \widetilde{\mathcal{M}_Q^0}(\bm{t};\bm{m},\bm{n})=\mathcal{M}^{\ast}(\bm{t};\bm{m},\bm{n})+\mathrm{O}_{\epsilon}(Q^{-1/2+\epsilon}). \]
\end{lem} 

\begin{lem}
\label{error-terms-lemma}
\[ \widetilde{\mathcal{M}_Q'}(\bm{t}; \bm{m}, \bm{n})\ll_{\epsilon} Q^{-1/3+\epsilon}. \]
\end{lem}

Taking Lemmata~\ref{main-term-lemma} and \ref{error-terms-lemma} for granted, we are now ready for the proofs of the main results of this section.

\begin{proof}[Proof of Proposition~\ref{moments-eval-prop}]
Proposition~\ref{moments-eval-prop} follows immediately by combining \eqref{tilde-non-tilde-moments-eq}, \eqref{moments-main-error-decomp}, and Lemmata~\ref{main-term-lemma} and \ref{error-terms-lemma}.
\end{proof}

\begin{proof}[Proof of Corollary~\ref{finite-distributions-cor}]
For every $\bm{t}=(t_1,\dots,t_k)\in[0,1]^k$, the $\mathbb{C}^k$-valued random variable $G^{\ast}(\bm{t})=(G^{\ast}(t_1),\dots,G^{\ast}(t_k))$ is mild by Lemma~\ref{evaluation-of-limit-moment}.
According to Proposition~\ref{method-of-moments-prop}, the convergence in law of $G_Q(\bm{t})$ to $G^{\ast}(\bm{t})$ as $Q\to\infty$ may be verified by checking that the corresponding moments satisfy $\mathcal{M}_Q(\bm{t};\bm{m},\bm{n})\to\mathcal{M}^{\ast}(\bm{t};\bm{m},\bm{n})$ as $Q\to\infty$ for every $\bm{m},\bm{n}\in\mathbb{Z}_{\geqslant 0}^k$. But this follows immediately (in a strong form) from Proposition~\ref{moments-eval-prop}.
\end{proof}

\subsection{Square and non-square contributions}
\label{main-error-subsec}
In this subsection, we prove Lemmata~\ref{main-term-lemma} and \ref{error-terms-lemma}.

\begin{proof}[Proof of Lemma~\ref{main-term-lemma}]
Adapting the sieving argument of \eqref{DQ-comput}, we find that for every $0\neq H=Q^{\mathrm{O}(1)}$,
\begin{align*}
&\sum_{\substack{c\in\cd_Q\\(c,H)=1}}m_Q(c)
=\frac1Q\Big(3\zeta(2)+\mathrm{O}\Big(\frac1{Q^{1/2}}\Big)\Big)\sum_{\substack{\alpha^2\leqslant 2Q\\ (\alpha,2H)=1}}\mu(\alpha)\sum_{\delta\mid H}\mu(\delta)\sum_{\substack{c\in [Q,2Q]\\\alpha^2\delta\mid c,\,c\equiv 1\bmod 4}}1\\
&\quad=3\zeta(2)\prod_{\delta\mid H,\,2\nmid\delta}\frac{\mu(\delta)}{\delta}\sum_{(\alpha,2H)=1}\frac{\mu(\alpha)}{4\alpha^2}+\mathrm{O}_{\epsilon}\Big(\frac{H^{\epsilon}}{Q^{1/2}}\Big)\\
&\quad=3\zeta(2)\prod_{p\mid H,\,p>2}\Big(1-\frac1p\Big)\prod_{p\mid 2H}\Big(1-\frac1{p^2}\Big)^{-1}\frac1{4\zeta(2)}+\mathrm{O}_{\epsilon}\Big(\frac1{Q^{1/2-\epsilon}}\Big)
=\eta(H)+\mathrm{O}_{\epsilon}\Big(\frac{1}{Q^{1/2-\epsilon}}\Big).
\end{align*}
Inserting this into the definition of $\widetilde{\mathcal{M}_Q^0}(\bm{t};\bm{m},\bm{n})$ yields
\[ \widetilde{\mathcal{M}_Q^0}(\bm{t};\bm{m},\bm{n})=\sum_{\substack{\vec{\bm{h}}\in\mathcal{H}'\\|H(\vec{\bm{h}})|=\square}}\beta(\vec{\bm{h}};\bm{t})\bigg(\eta(H(\vec{\bm{h}}))+\mathrm{O}_{\epsilon}\Big(\frac{1}{Q^{1/2-\epsilon}}\Big)\bigg). \]
We then estimate, arguing as in \eqref{absolute-convergence-eq} and \eqref{ESmSn2p-estimate},
\begin{equation}
\label{tails-estimate}
\begin{aligned}
&\sum_{\substack{\vec{\bm{h}}\in\mathcal{H}^{\ast}_{\bm{m},\bm{n}}\setminus\mathcal{H}'\\|H(\vec{\bm{h}})|=\square}}|\beta(\vec{\bm{h}};\bm{t})\eta(H(\vec{\bm{h}}))|\ll_{\epsilon,m,n}\sum_{df^2>Q/2}\mu^2(d)\sum_{df^2\mid k^2}\frac{1}{k^{2-\epsilon}}\\
&\qquad\ll\sum_{df^2>Q/2}\sum_{n=1}^{\infty}\frac1{(dfn)^{2-\epsilon}}
\ll_{\epsilon}\sum_{d,f}\frac1{Q^{1/2-\epsilon}}\frac1{(df)^{1+\epsilon}}\ll_{\epsilon}\frac1{Q^{1/2-\epsilon}},
\end{aligned}
\end{equation}
as well as
\[ \frac{1}{Q^{1/2-\epsilon}}\sum_{\substack{\vec{\bm{h}}\in\mathcal{H}^{\ast}_{\bm{m},\bm{n}}\\|H(\vec{\bm{h}})|=\square}}|\beta(\vec{\bm{h}};\bm{t})|\ll_{m,n}\frac1{Q^{1/2-\epsilon}}. \]
Putting everything together and invoking the evaluation of $\mathcal{M}^{\ast}(\bm{t};\bm{m},\bm{n})$ from Lemma~\ref{evaluation-of-limit-moment} completes the proof.
\end{proof}

\begin{proof}[Proof of Lemma~\ref{error-terms-lemma}]
We begin by noting the (direct) upper bound
\[ |\beta(\vec{\bm{h}};\bm{t})|\ll\prod_{(j,\ell)\in\Sigma_{\bm{k},\bm{m},\bm{n}}}\frac1{1+|h_{j,\ell}|}\leqslant \prod_{(j,\ell)\in\Sigma_{\bm{k},\bm{m},\bm{n}}}\frac1{|1-h_{j,\ell}|}=\frac1{|H(\vec{\bm{h}})|}. \]
Therefore, grouping the terms indexed by $\vec{\bm{h}}\in\mathcal{H}'$ according to the values of $\square\neq |H(\vec{\bm{h}})|\leqslant Q^{\ast}:=(Q/2+1)^{m+n}$, we have
\[ \widetilde{\mathcal{M}_Q'}(\bm{t};\bm{m},\bm{n})\ll\sum_{k\leqslant Q^{\ast}}\frac1{k^2}\sumast_{1<d\leqslant Q^{\ast}/k^2}\frac{\tilde{\tau}_{\mathcal{H}'}(dk^2)}d\bigg|\sum_{\substack{c\in [Q,2Q]\cap\mathcal{D}\\(c,k)=1}}m_Q(c)\bigg(\frac dc\bigg)\bigg|, \]
where, by the divisor bound,
\[ \tilde{\tau}_{\mathcal{H}'}(dk^2):=\#\big\{\vec{\bm{h}}\in\mathcal{H}':|H(\vec{\bm{h}})|=dk^2\big\}\ll_{\epsilon} Q^{\epsilon}. \]

Now, sieving for the conditions $c\in\mathcal{D}$ and $(c,k)=1$, we find that
\[ \sum_{\substack{c\in [Q,2Q]\cap\mathcal{D}\\(c,k)=1}}m_Q(c)\bigg(\frac dc\bigg)=\sum_{\delta\mid k}\sum_{\alpha^2\leqslant 2Q}\mu(\delta)\mu(\alpha)\sum_{2\nmid c\in\big[\frac{Q}{[\alpha^2,\delta]},\frac{2Q}{[\alpha^2,\delta]}\big]}m_Q([\alpha^2,\delta]c)\bigg(\frac{d}{[\alpha^2,\delta]c}\bigg). \]
Since $d>1$ is square-free, $(d/c)$ is a non-principal character of conductor $\asymp d$, we find by the P\'olya--Vinogradov inequality and straightforward estimates that
\[ \sum_{\substack{c\in [Q,2Q]\cap\mathcal{D}\\(c,k)=1}}m_Q([\alpha^2,\delta]c)\bigg(\frac dc\bigg)\ll_{\epsilon}Q^{\epsilon}\sum_{\alpha^2\leqslant 2Q}\min\Big(\sqrt{d},\frac{Q}{\alpha^2}\Big)\ll_{\epsilon} Q^{-1/2+\epsilon}d^{1/4}. \]
Therefore
\[ \widetilde{\mathcal{M}_Q}(\bm{t};\bm{m},\bm{n})_{\leqslant D}:=\sum_{k\leqslant Q^{\ast}}\frac1{k^2}\sumast_{1<d\leqslant D}\frac1d \bigg|\sum_{\substack{c\in [Q,2Q]\cap\mathcal{D}\\(c,k)=1}}m_Q(c)\bigg(\frac dc\bigg)\bigg|\ll_{\epsilon}Q^{-1/2+\epsilon}D^{1/4}. \]

On the other hand, using the Cauchy--Schwarz inequality followed by Heath-Brown's quadratic large sieve (Proposition~\ref{HBQLS}), we find that
\[ \sumast_{D'<d\leqslant 2D'}\frac1d\bigg|\sum_{\substack{c\in [Q,2Q]\cap\mathcal{D}\\(c,k)=1}}m_Q(c)\bigg(\frac dc\bigg)\bigg|\ll\frac1{(D')^{1/2}}\frac{(D'+Q))^{1/2}}{Q^{1/2}}\ll Q^{-1/2}+\frac1{(D')^{1/2}}. \]
Therefore
\[ \widetilde{\mathcal{M}_Q}(\bm{t};\bm{m},\bm{n})_{>D}:=\sum_{k\leqslant Q^{\ast}}\frac1{k^2}\sumast_{D<d\leqslant Q^{\ast}/k^2}\frac1d\bigg|\sum_{\substack{c\in [Q,2Q]\cap\mathcal{D}\\(c,k)=1}}m_Q(c)\bigg(\frac dc\bigg)\bigg|\ll_{\epsilon} Q^{-1/2+\epsilon}+\frac{Q^{\epsilon}}{D^{1/2}}. \]

Making the optimal choice $D=Q^{2/3}$ we conclude that
\[ \widetilde{\mathcal{M}_Q'}(\bm{t};\bm{m},\bm{n})\ll_{\epsilon} Q^{\epsilon}\big(\widetilde{\mathcal{M}_Q}(\bm{t};\bm{m},\bm{n})_{\leqslant D}+\widetilde{\mathcal{M}_Q}(\bm{t};\bm{m},\bm{n})_{>D}\big)\ll_{\epsilon} Q^{-1/3+\epsilon}. \qedhere \]
\end{proof}

We note that we did not try to optimize the final exponent of power savings in Lemma \ref{error-terms-lemma} and that the character sums in the proof can surely be treated more delicately. We opted for brevity since the present power savings suffice for us.

\section{Convergence in law}
\label{sec-inlaw}

The goal of this section is to prove the following statement, which will in turn be used to verify the tightness of the sequence of $C^0([0,1],\mathbb{C})$-valued random variables $(G_Q)$ as $Q\to\infty$.

\begin{prop}
\label{moment-claim}
For every real $\alpha>2$, there exists a $\delta=\delta(\alpha)>0$ such that for every $Q\geqslant 1$ and every $0\leqslant s,t\leqslant 1$,
\[ \frac1Q\sum_{c\in\cd_Q}|G(t;c)-G(s;c)|^{\alpha}\ll_{\alpha} |t-s|^{1+\delta}. \]
\end{prop}

We will prove Proposition~\ref{moment-claim} in \S\ref{tightness-proof-sec}, along with the corollaries for the tightness and convergence in law of the sequence $(G_Q)$, after laying the ground work in \S\ref{tightness-prep-sec} and \S\ref{tightness-ranges-sec}. It will be seen that we can, in fact, choose $\delta(\alpha)=\min(\delta_1(\alpha-2),\delta_2)$ for some two constants $\delta_1,\delta_2>0$, which may in principle be explicated; in particular, any $\delta_2<\frac13$ is allowable with a corresponding suitable $\delta_1>0$.

\subsection{Preparatory lemmata}
\label{tightness-prep-sec}
The first principal arithmetic input into the proof of tightness is the following lemma, which is a simple variation of the Burgess-like bound for short mixed character sums (Proposition~\ref{Chang-Burgess}).

\begin{lem}
\label{Burgess-var}
For every $\kappa<\frac34$, there exists a $\delta=\delta_{\kappa}>0$ such that for every primitive character of any square-free conductor $c>1$ and any $M$ and every $1\leqslant N\leqslant c^{\kappa}$,
\[ \sum_{M<n\leqslant M+N}\chi(n)e^{2\pi in/c}\ll_{\kappa} c^{1/2-\delta}. \]
\end{lem}

\begin{proof}
We begin with an application of completion, as in Lemma~\ref{lem: g-squiggle-val}, obtaining
\begin{align*}
S:=\sum_{M<n\leqslant M+N}\chi(n)e^{2\pi in/c}
&=\frac1c\sum_{h\bmod c}G(1-h;\chi)\sum_{M<x\leqslant M+N}e_c(hx)\\
&=\frac{\epsilon_{\chi}}{c^{1/2}}\bigg(\sum_{\substack{h\bmod c\\\|h/c\|\leqslant 1/N}}+\sum_{\substack{h\bmod c\\\|h/c\|>1/N}}\bigg)\overline{\chi(1-h)}\sum_{M<x\leqslant M+N}e_c(hx),
\end{align*}
where $\epsilon_{\chi}$ is the sign of the Gauss sum for the character $\chi$. Denote the two sums above (including the factor $\epsilon_{\chi}/c^{1/2}$) by $S_1$ and $S_2$. On the one hand, by exchanging the order of summation and applying Proposition~\ref{Chang-Burgess} we have
\[ S_1\ll\frac1{c^{1/2}}N\cdot \frac{c}{N}c^{-\delta}=c^{1/2-\delta}. \]

On the other hand, denoting
\[ S_x(t)=\sum_{0\leqslant h\leqslant t}\overline{\chi(1-h)}e_c(hx),\quad f(t)=\frac{e_c(t)}{1-e_c(t)}, \]
we have by integration by parts  that
\begin{align*}
S_2&=\frac{\epsilon_{\chi}}{c^{1/2}}\int_{(c/N)+}^{(c-c/N)-}f(t)\,\dd (S_{M+N}(t)-S_M(t))\\
&=\frac{\epsilon_{\chi}}{c^{1/2}}f(t)(S_{M+N}(t)-S_M(t))\bigg|_{(c/N)+}^{(c-c/N)-}-\frac{\epsilon_{\chi}}{c^{1/2}}\int_{c/N}^{c-c/N}S(t)f'(t)\,\dd t.
\end{align*}
Now, the conclusion of Proposition~\ref{Chang-Burgess} implies, for $|I|>c^{1/4+\kappa}$, that
\[ \sum_{h\in I}\overline{\chi(1-h)}e_c(hx)\ll_{\kappa}\min(|I|,c-|I|)c^{-\delta}, \]
simply because the complete sum over all $n\bmod c$ is of size at most $\mathrm{O}(c^{1/2})$. This implies that
\[ S_2\ll\frac1{c^{1/2}}N\cdot\frac{c}{N}c^{-\delta}+\frac1{c^{1/2}}\int_{c/N}^{c-c/N}\frac{c}{\min(t,c-t)}\,\dd t\cdot c^{-\delta}\ll_{\kappa} c^{1/2-\delta+\delta/100}. \]
The proof is complete.
\end{proof}

The following variation of Heath-Brown's quadratic large sieve is convenient and probably known, but we could not locate a ready reference.

\begin{lem}
\label{HB-var}
For any $M\in\mathbb{N}$,\! real numbers $0<s<t$,\! and complex numbers $(a_n)_{sM<n\leqslant 2tM}$,
\begin{equation}
\label{HB-var-eq}
\sumast_{M<m\leqslant 2M}\bigg|\sumast_{sm<n\leqslant tm}a_n\Big(\frac nm\Big)\bigg|^2\ll_{\epsilon}M^{\epsilon}\Big[1+M\min\Big(\frac{|t-s|}{\min(1,s)},\Big(\frac{|t-s|}{s}\Big)^{\epsilon}\Big)\Big]\sumast_{sM<n\leqslant 2tM}|a_n|^2.
\end{equation}
\end{lem}

\begin{proof}
Clearly we may assume with out loss of generality that $1/(2M)\leqslant s<t$, and then we may also assume that $s<t\leqslant 2s$, as the general case follows by splitting $[s,t]$ into $\mathrm{O}(1+\log(t/s))$ intervals of this form.

Now, for every $0\leqslant k\leqslant \log_2M+1$, consider the collection $\mathcal{I}_k$ of intervals of length $|t-s|M/2^k$ intersecting $[sM,2tM]$. For every $M<m\leqslant 2M$, we split the interval $(sm,tm]$ into $\mathrm{O}(\log M)$ such intervals (at most two for each value of $k$), using the obvious greedy algorithm, and we estimate the inner sum in \eqref{HB-var-eq} using the Cauchy--Schwarz inequality.

Now, each specific interval $I\in\mathcal{I}_k$ appears in at most $\mathrm{O}(|I|/s+1)$ such decompositions. For, indeed, for $I$ to appear in the decomposition of $[sm,tm]$, one of the points $sm$ or $tm$ must appear in the interval $\tilde{I}$ centered at the midpoint of $I$ and of length $3|I|$; but this forces $m$ to lie in the union $\tilde{J}(I)=\tilde{I}/s\cup\tilde{I}/t$ and, in particular, determines $m$ to within the stated number of choices. Putting this together, we have the estimate
\[ S:=\sumast_{M<m\leqslant 2M}\bigg|\sumast_{sm<n\leqslant tm}a_n\Big(\frac nm\Big)\bigg|^2\ll \log M\sum_{0\leqslant k\ll\log M}\sum_{I\in\mathcal{I}_k}\sumast_{m\in J(I)}\bigg|\sumast_{n\in I}a_n\Big(\frac nm\Big)\bigg|^2. \]
Estimating the inner double sum using Heath-Brown's quadratic large sieve (Proposition~\ref{HBQLS}), we have that
\begin{align*}
S&\ll_{\epsilon}M^{\epsilon}\sum_{0\leqslant k\ll\log M}\sum_{I\in\mathcal{I}_k}\Big(|I|+\frac{|I|}s+1\Big)^{1+\epsilon}\sumast_{n\in I}|a_n|^2\\
&\ll_{\epsilon}M^{\epsilon}\sum_{0\leqslant k\ll\log M}\Big(\frac{|t-s|}s\frac{M}{2^k}+|t-s|\frac{M}{2^k}+1\Big)\sum_{sM<n\leqslant 2tM}|a_n|^2\\
&\ll_{\epsilon}M^{\epsilon}\Big[M\frac{|t-s|}{\min(1,s)}+1\Big]\sum_{sM<n\leqslant 2tM}|a_n|^2.
\end{align*}
This completes the proof.
\end{proof}

Lemma~\ref{HB-var} allows us to prove the following estimate, which will be crucial in our estimates.
\begin{lem}
\label{HB-var-cor}
For every $0\leqslant s<t\leqslant 1$ with $|t-s|\geqslant 1/(2Q)$, we have
\[ \sumast_{Q<c\leqslant 2Q}\bigg|\sumast_{sc<x\leqslant tc}\Big(\frac xc\Big)e^{2\pi ix/c}\bigg|^2\ll_{\epsilon} |t-s|Q^{2+\epsilon}. \]
\end{lem}

\begin{proof}
Clearly we may assume that $s\geqslant 1/(2Q)$ without loss of generality. For $\tau\in[s,t]$, denote
\[ S_c(\tau)=\sumast_{sc<x\leqslant\tau c}\Big(\frac xc\Big). \]
By Lemma~\ref{HB-var}, we have the estimate
\begin{equation}
\label{estimate-from-lemma}
\sumast_{Q<c\leqslant 2Q}|S_c(\tau)|^2\ll_{\epsilon}tQ^{2+\epsilon}\min\Big(\frac{|t-s|}s,\Big(\frac{|t-s|}{s}\Big)^{\epsilon}\Big)\ll_{\epsilon}|t-s|Q^{2+\epsilon},
\end{equation}
where the final estimate follows by separately considering the cases  $|t-s|\leqslant s$ and $|t-s|>s$ and keeping in mind that $s\geqslant 1/(2Q)$.

Now, by summation by parts and the Cauchy--Schwarz inequality, we have
\[ \sumast_{Q<c\leqslant 2Q}\bigg|\sumast_{sc<x\leqslant tc}\Big(\frac xc\Big)e^{2\pi ix/c}\bigg|^2
\ll\sumast_{Q<c\leqslant 2Q}|S_c(t)|^2+\sumast_{Q<c\leqslant 2Q}\frac1{c^2}\bigg|\int_s^tS_c(\tau)e^{2\pi i\tau/c}\,\dd\tau\bigg|^2. \]
Moreover, by the integral Minkowski's inequality,
\[ \bigg(\sumast_{Q<c\leqslant 2Q}\bigg|\int_s^tS_c(\tau)e^{2\pi i\tau/c}\,\dd\tau\bigg|^2\bigg)^{1/2}\leqslant\int_s^t\bigg(\sumast_{Q<c\leqslant 2Q}|S_c(\tau)|^2\bigg)^{1/2}\,\dd\tau. \]
Upon applying \eqref{estimate-from-lemma} in the two previous displays, we conclude that
\[ \sumast_{Q<c\leqslant 2Q}\bigg|\sumast_{sc<x\leqslant tc}\Big(\frac xc\Big)e^{2\pi ix/c}\bigg|^2\ll_{\epsilon}|t-s|Q^{2+\epsilon}+|t-s|^3Q^{2+\epsilon}\ll_{\epsilon} |t-s|Q^{2+\epsilon}, \]
as desired.
\end{proof}

\subsection{Estimates according to ranges}
\label{tightness-ranges-sec}
he following two lemmata are purely analytical in nature.
\begin{lem}
\label{very-short-range}
If $\alpha>0$ and
\[ 0\leqslant |t-s|\leqslant\frac1{c-1}, \]
then
\[ |G(t,c)-G(s;c)|^{\alpha}\leqslant 2^{\alpha}|t-s|^{\alpha/2}=2^{\alpha}|t-s|^{1+(1/2)(\alpha-2)}. \]
\end{lem}

\begin{lem}
\label{non-tilde-to-tilde}
If $\alpha\geqslant 1$ and
\[ |t-s|\geqslant\frac1{c-1}, \]
then
\[ |G(t;c)-G(s;c)|^{\alpha}\ll_{\alpha}|t-s|^{\alpha/2}+|\widetilde{G}(t;c)-\widetilde{G}(s;c)|^{\alpha}. \]
\end{lem}

These statements are analogues of \cite[Lemma 4.2]{RicottaRoyerShparlinski2020} and \cite[Lemma 4.3]{RicottaRoyerShparlinski2020} (as well as of the corresponding statements in \cite{MilicevicZhang2023}). The proofs are essentially verbatim, and so we omit them for brevity, the only notable adaptation being the insertion of \eqref{simplebounds} in place of \cite[(5)]{RicottaRoyerShparlinski2020}.

\begin{lem}
\label{middle-range}
For every $\kappa>\frac14$, there exists a $\delta=\delta_{\kappa}>0$ such that, if $\alpha\geqslant 2$ and if
\[ \frac1{c-1}\leqslant |t-s|\leqslant c^{-\lambda} \]
for some $\lambda\geqslant\kappa$, then
\[ \frac1Q\sum_{c\in\cd_Q}|G(t;c)-G(s;c)|^{\alpha}\ll_{\alpha,\kappa,\epsilon} |t-s|^{1+\delta(\alpha-2)-\epsilon}. \]
\end{lem}

\begin{proof}
Without loss of generality, let $s<t$. By the very definition of $\widetilde{G}(\cdot;c)$,
\[ \widetilde{G}(t;c)-\widetilde{G}(s;c)=\frac1{\sqrt{c}}\sum_{sc <x\leqslant tc}\Big(\frac xc\Big)e^{2\pi ix/c}+\mathrm{O}\Big(\frac1{\sqrt{c}}\Big), \]
where we incur the harmless error term for no other reason than notational simplicity. Thus, using Lemma~\ref{non-tilde-to-tilde} and recalling the condition that $|t-s|\gg 1/Q$, we first have that
\[ V:=\frac1Q\sum_{c\in\cd_Q}|G(t;c)-G(s;c)|^{\alpha}\ll_{\alpha} |t-s|^{\alpha/2}+\frac1{Q^{1+\alpha/2}}\sum_{c\in\cd_Q}\bigg|\sum_{sc<x\leqslant tc}\Big(\frac xc\Big)e^{2\pi ix/c}\bigg|^{\alpha}. \]

Using Lemmata~\ref{Burgess-var} and \ref{HB-var-cor} we conclude that
\begin{align*}
V&\ll_{\alpha,\epsilon}|t-s|^{\alpha/2}+\frac1{Q^{1+\alpha/2}}\big(Q^{1/2-\delta}\big)^{\alpha-2}\sum_{c\in\cd_Q}\bigg|\sum_{sc<x\leqslant tc}\Big(\frac xc\Big)e^{2\pi ix/c}\bigg|^2\\
&\ll_{\alpha,\epsilon}|t-s|^{\alpha/2}+|t-s|Q^{-\delta(\alpha-2)+\epsilon}.
\end{align*}
The statement of the lemma follows upon inputing the condition $Q\ll |t-s|^{-1/\lambda}$, and setting $\min(\frac12,\delta/\lambda)$ as the value of $\delta$.
\end{proof}

\begin{lem}
\label{long-range}
For every even integer $\alpha\geqslant 2$, if
\[ c^{-\lambda}\leqslant |t-s|\leqslant 1 \]
for some $\lambda<\frac13$, then
\[ \frac1Q\sum_{c\in\cd_Q}|G(t;c)-G(s;c)|^{\alpha}\ll_{\alpha,\epsilon} |t-s|^{1+\min(\frac12(\alpha-2),\frac{1-3\lambda}{3\lambda})-\epsilon}. \]
\end{lem}

\begin{proof}
First off, in analogy with \eqref{modified-moments-equation}, we denote
\begin{align*}
\widetilde{\mathcal{M}^{\ast}_Q}([s,t];\alpha)
&:=\sum_{c\in\cd_Q}m_Q(c)|\widetilde{G}(t;c)-\widetilde{G}(s;c)|^{\alpha}\\
&=\sum_{c\in\cd_Q}m_Q(c)\overline{(\widetilde{G}(t;c)-\widetilde{G}(s;c))}^{\alpha/2}(\widetilde{G}(t;c)-\widetilde{G}(s;c))^{\alpha/2}.
\end{align*}
In other words, the moment $\widetilde{\mathcal{M}^{\ast}_Q}([s,t];\alpha)$ is constructed in exactly the same fashion as $\widetilde{\mathcal{M}_Q}(\bm{t};\bm{m},\bm{n})$, with $k=1$ and $m_1=n_1=\alpha/2$, but with $\widetilde{G}(t;c)-\widetilde{G}(s;c)$ in place of each $\widetilde{G}(t_1;c)$. In yet other words, we have a finite (with length and coefficients depending only on the fixed value of $\alpha$) expansion
\[ \widetilde{\mathcal{M}^{\ast}_Q}([s,t];\alpha)=\sum_{k_1=0}^{\alpha/2}\sum_{k_2=0}^{\alpha/2}(-1)^{k_1+k_2}\binom{\alpha/2}{k_1}\binom{\alpha/2}{k_2}\widetilde{\mathcal{M}_Q}((t,s);(\alpha/2-k_1,k_1),(\alpha/2-k_2,k_2)). \]
Applying Lemma~\ref{moments-order-switched}, decomposition \eqref{moments-main-error-decomp}, and Lemma~\ref{error-terms-lemma} to this expansion, we have
\[ \begin{split}
\widetilde{\mathcal{M}^{\ast}_Q}([s,t];\alpha)=\sum_{k_1=0}^{\alpha/2}\sum_{k_2=0}^{\alpha/2}(-1)^{k_1+k_2}\binom{\alpha/2}{k_1}\binom{\alpha/2}{k_2} \sum_{\substack{\vec{\bm{h}}\in\mathcal{H}'\\H(\vec{\bm{h}})=\square}}\beta(\vec{\bm{h}};(t,s))\sum_{\substack{c\in\cd_Q\\(c,H(\vec{\bm{h}}))=1}}m_Q(c)\\
+\mathrm{O}_{\alpha,\epsilon}(Q^{-1/3+\epsilon}),
\end{split} \]
where $\mathcal{H}'$ is the set of all pairs $\vec{\bm{h}}=(\vec{h}_1,\vec{h}_2)$ of vectors $\vec{h}_j=(h_{j,\ell})_{1\leqslant\ell\leqslant\alpha/2}$ satisfying $|h_{j,\ell}|<Q/2$, and
\[ \beta(\vec{\bm{h}};(t,s))=\prod_{\ell=1}^{\alpha/2-k_1}\overline{\beta(h_{1,\ell};t)}
\prod_{m=1}^{k_1}\overline{\beta(h_{1,\alpha/2-k_1+m};s)}
\prod_{\ell=1}^{\alpha/2-k_2}\beta(h_{2,\ell};t)
\prod_{m=1}^{k_2}\beta(h_{2,\alpha/2-k_2+m};s). \]
But then a moment's reflection shows that in fact
\[ \widetilde{\mathcal{M}_Q^{\ast}}([s,t];\alpha)=\mathbb{E}\bigg(\bigg|\sum_{|h|<Q/2}(\beta(h;t)-\beta(h;s))X_{1-h}\bigg|^{\alpha}\bigg)+\mathrm{O}_{\alpha,\epsilon}(Q^{-1/3+\epsilon}). \]
The same result can be arrived at by following the evaluation of $\widetilde{\mathcal{M}_Q}(\bm{t};\bm{m},\bm{n})$ in section~\ref{computing-moments-section}, with $k=1$ and $m_1=n_1=\alpha/2$ and with $\widetilde{G}(t;c)-\widetilde{G}(s;c)$ in place of each $\widetilde{G}(t_1;c)$.

Now, the expectation occurring in this evaluation may be further estimated as
\begin{align*}
\mathbb{E}\bigg(\bigg|\sum_{|h|<Q/2}(\beta(h;t)-\beta(h;s))X_{1-h}\bigg|^{\alpha}\bigg)
&\leqslant\sum_{\substack{|h_1|,\dots,|h_k|<Q/2\\H(\vec{\bm{h}})=\square}}\prod_{j=1}^{\alpha}|\beta(h_j;t)-\beta(h_j;s)|\\
&\ll_{\alpha}\sum_{\substack{|h_1|,\dots,|h_k|<Q/2\\H(\vec{\bm{h}})=\square}}\prod_{j=1}^{\alpha}\frac{|t-s|^{1/2-\epsilon/\alpha}}{(1+|h_j|)^{1/2+\epsilon/\alpha}}\\
&\ll_{\alpha,\epsilon} \sum_{n=1}^{\infty}\frac{|t-s|^{\alpha/2-\epsilon}}{(n^2)^{1/2+\epsilon-\epsilon/2}}\ll_{\alpha,\delta} |t-s|^{\alpha/2-\epsilon},
\end{align*}
bounding $|\beta(h;t)-\beta(h;s)|$ by interpolating between the obvious estimates $|\beta(h;t)-\beta(h;s)|\ll\min(|t-s|,1/(1+|h|))$. The statement of the lemma follows upon inputing the condition $Q\gg |t-s|^{-1/\lambda}$ and recalling that $m_Q(c)\asymp 1/Q$.
\end{proof}

\subsection{Proof of tightness}
\label{tightness-proof-sec}

We are now ready for the main proofs of this section.

\begin{proof}[Proof of Proposition~\ref{moment-claim}]
Fix an arbitrary $\frac14<\kappa=\lambda<\frac13$, and let $\delta=\delta_{\kappa}>0$ be as in the statement of Lemma~\ref{middle-range}. Applying Lemma \ref{very-short-range}, \ref{middle-range}, or \ref{long-range} according to the size of $|t-s|$, we conclude that, for every even integer $\alpha\geqslant 2$ and every $t,s\in [0,1]$,
\[ \frac1Q\sum_{c\in\cd_Q}|G(t;c)-G(s;c)|^{\alpha}\ll_{\alpha,\kappa,\epsilon}|t-s|^{1+\delta(\alpha)-\epsilon}, \]
with
\[ \delta(\alpha)=\min\big(\tfrac12(\alpha-2),\delta_{\kappa}(\alpha-2),\tfrac{1-3\lambda}{3\lambda}\big). \]
This completes the proof of Proposition~\ref{moment-claim} when $\alpha>2$ is an even integer; the claim for other values of $\alpha>2$ follows by interpolation. We also see that we can take $\delta(\alpha)=\min(\delta_1(\alpha-2),\delta_2)$ with (in principle) explicit $\delta_1,\delta_2>0$, and we may obtain any $\delta_2<\frac13$ by taking $\kappa=\lambda$ sufficiently close to $\frac14$.
\end{proof}

Using Kolmogorov's Tightness Criterion (Proposition~\ref{prop: Kolmogorov-tightness}), Proposition~\ref{moment-claim} immediately implies the following statement.

\begin{cor}
\label{tight-cor}
The sequence of $C^0([0,1],\mathbb{C})$-valued random variables $(G_Q)$ is tight as $Q\to\infty$.
\end{cor}

Combining Corollaries~\ref{finite-distributions-cor} and \ref{tight-cor} and applying Prokhorov's Criterion (Proposition~\ref{prop: Prokhorov}), we then obtain the following capstone statement.

\begin{cor}
The sequence of $C^0([0,1],\mathbb{C})$-valued random variables $(G_Q)$ converges in law to $G^{\ast}$ as $Q\to\infty$.
\end{cor}

\section{The atlas of shapes}
\label{sec-atlas}

In this section, we consider the local behavior of the limiting shapes $G^{\sharp}_{\bm{\epsilon}_Z}(t)$ at rational points $t_0\in [0,1]\cap\mathbb{Q}$. We prove the general first-order asymptotics in \S\ref{limits-local} and a refined two-term asymptotic expansion in \S\ref{finer-local-subsec}. These asymptotics relate the local behavior of $G^{\sharp}_{\bm{\epsilon}_Z}(t)$ at $t_0=a/q$ to certain complete exponential sums modulo $q$, which we study in detail in \S\ref{expsums-nv-mult}. In \S\ref{main-cor-subsec}, we combine all these conclusions and prove Proposition~\ref{main-cor}, which provides a collection of rational points $t_0=a/q$ at which the path $G^{\sharp}_{\bm{\epsilon}_Z}(t)$ has a cusp and explains the striking sharp reversals observed in the introduction (Figures \ref{fig: gausspaths}--\ref{fig: atlas-5}).

\subsection{Local behavior of limiting shapes}
\label{limits-local}
We begin by writing for short
\begin{gather*}
 G^{\sharp}_{\bm{\epsilon}_Z}(t)-G^{\sharp}_{\bm{\epsilon}_Z}(t_0)=\sum_{n\in\mathcal{N}^{-}_{\bm{\epsilon}_Z}}\epsilon_n(g_n(t)-g_n(t_0))+(t-t_0)=\sum_{n\in\mathcal{N}_{\bm{\epsilon}_Z}}\epsilon_n(g_n(t)-g_n(t_0)),\\
 \mathcal{N}^{-}_{\bm{\epsilon}_Z}=\{n\in\mathbb{Z}\setminus\{-1,0\}:p\mid n\,\Rightarrow\,p\leqslant Z\text{ and }\epsilon_p\neq 0\},\quad \mathcal{N}_{\bm{\epsilon}_Z}=\mathcal{N}_{\bm{\epsilon}_Z}^{-}\cup\{-1\},\\
\epsilon_n=\prod_{p^{a_p}\exmid n}\epsilon_p^{a_p},\quad g_n(t)=\frac{e((n+1)t)-1}{2i\pi(n+1)}\quad(n\neq -1),\quad g_{-1}(t)=t.
\end{gather*}
We record the simple estimates
\begin{equation}
\label{simple-estimates}
|g_n(t)-g_n(t_0)|\leqslant\min\Big(|t-t_0|,\frac1{\pi|n+1|}\Big).
\end{equation}

Now, fix a $t_0=a/b\in[0,1]\cap\mathbb{Q}$ with $(a,b)=1$, and write
\begin{equation}
\label{b-decomp}
b=b_Zb^Z\quad\text{with}\quad b_Z=\big(b,\prod_{p\leqslant Z}p^{\infty}\big).
\end{equation}
By absolute convergence, we may rewrite
\begin{equation}
\label{decomposition-by-divisors}
G^{\sharp}_{\bm{\epsilon}_Z}=\sum_{d\mid b_Z}G^{\sharp}_{\bm{\epsilon}_Z}[d],\quad\text{where}\quad G^{\sharp}_{\bm{\epsilon}_Z}[d]=\sum_{\substack{n=dn'\in\mathcal{N}_{\bm{\epsilon}_Z}\\(n',b_Z/d)=1}}\epsilon_ng_n(t).
\end{equation}
It will be convenient to write
\begin{equation}
\label{NZ-PZ}
\begin{alignedat}{3}
&\mathcal{N}_{\bm{\epsilon}_Z}^{m-}=\{n'\in\mathcal{N}_{\bm{\epsilon}_Z}:(n',m)=1
\}, &&\quad \mathcal{N}_{\bm{\epsilon}_Z}[d]=\mathcal{N}_{\bm{\epsilon}_Z}^{(b/d)-},\\
&\mathcal{P}_{\bm{\epsilon}_Z}^{m-}=\mathcal{P}_Z\cap\mathcal{N}_{\bm{\epsilon}_Z}^{m-}=\{p\leqslant Z:p\nmid m,\epsilon_p\neq 0\},&&\quad \mathcal{P}_{\bm{\epsilon}_Z}[d]=\mathcal{P}_{\bm{\epsilon}_Z}^{(b/d)-},
\end{alignedat}
\end{equation}
as well as $\mathcal{N}_{\bm{\epsilon}_Z}^{+}=\mathcal{N}_{\bm{\epsilon}_Z}\cap\mathbb{N}$, $\mathcal{N}_{\bm{\epsilon}_Z}^{+}[d]=\mathcal{N}_{\bm{\epsilon}_Z}[d]\cap\mathbb{N}$.
Finally, for a modulus $m\in\mathbb{N}$, we will also consider the normalized complete exponential sums
\begin{equation}
\label{sast-def}
s^{\ast}(a/m;\bm{\epsilon}_Z)=\frac1{(2\varphi(m))^{|\mathcal{P}_{\bm{\epsilon}_Z}^{m-}|}}\sum_{\substack{(m_p)_{p\in\mathcal{P}_{\bm{\epsilon}_Z}^{m-}}\\ 0\leqslant m_p<2\varphi(m)}}e\Big(\frac am\prod_{p\in\mathcal{P}_{\bm{\epsilon}_Z}^{m-}}p^{m_p}\Big)\prod_{p\in\mathcal{P}_{\bm{\epsilon}_Z}^{m-}}\epsilon_p^{m_p}.
\end{equation}

The connection between the exponential sums \eqref{sast-def} and the local behavior of the paths $G^{\sharp}_{\bm{\epsilon}_Z}(t)$ and their decomposition \eqref{decomposition-by-divisors} is given by the following lemma.

\begin{lem}
\label{local-d-lemma}
For every
$t_0=a/b\in [0,1]\cap\mathbb{Q}$ as in \eqref{b-decomp}, and for every $d\mid b_Z$, the function $G^{\sharp}_{\bm{\epsilon}_Z}[d]$ defined in \eqref{decomposition-by-divisors} satisfies, for $t\in [0,1]$,
\begin{align*}
G^{\sharp}_{\bm{\epsilon}_Z}[d](t)-G^{\sharp}_{\bm{\epsilon}_Z}[d](t_0)
&=(t-t_0)\big|\log|t-t_0|\big|^{|\mathcal{P}_{\bm{\epsilon}_Z}[d]|}\\
&\quad\times\Big(e(t_0)
c_{\bm{\epsilon}_Z}[d]\cdot 2\mRe s^{\ast}\Big(\frac{a}{b/d};\bm{\epsilon}_Z\Big)+\bigO_{Z,t_0}\Big(\frac1{|\log(t-t_0)|}\Big)\Big),
\end{align*}
where $\mathcal{P}_{\bm{\epsilon}_Z}[d]$, $s^{\ast}(a/(b/d);\bm{\epsilon}_Z)$, and $c_{\bm{\epsilon}_Z}[d]\neq 0$ are as in \eqref{NZ-PZ}, \eqref{sast-def}, and \eqref{cd-toquote}.
\end{lem}

\begin{proof}
For a $K>0$ to be chosen suitably large later,
we may write
\[ G^{\sharp}_{\bm{\epsilon}_Z}[d](t)-G^{\sharp}_{\bm{\epsilon}_Z}[d](t_0)=S_1[d]+S_2[d], \]
where, separating the terms in \eqref{decomposition-by-divisors} according to whether $n=dn'\leqslant dK$ or $n>dK$, splitting the summands into dyadic ranges of the form $n'\in [K',2K']$ (and denoting dyadic summations over $K'=2^{k'}$, $k'\in\mathbb{Z}_{\geqslant 0}$ by $\sum^{\mathrm{dy}}$), and estimating using \eqref{simple-estimates} and the Mean Value Theorem,
\begin{align*}
S_1[d]&=\sum_{\substack{n=dn'\in\mathcal{N}_{\bm{\epsilon}_Z}\\(n',b_Z/d)=1\\|n|\leqslant dK}}\epsilon_n(g_n(t)-g_n(t_0))
=\int_{t_0}^te(\tau)\sum_{\substack{n'\in\mathcal{N}_{\bm{\epsilon}_Z}[d]\\ |n'|\leqslant K}}\epsilon_{dn'}e(dn'\tau)\,\dd\tau\\
&=(t-t_0)e(t_0)2\mRe\sum_{\substack{n'\in\mathcal{N}_{\bm{\epsilon}_Z}^{+}[d]\\n'\leqslant K}}\epsilon_{dn'}e\Big(\frac{an'}{b/d}\Big)+\bigO\Big(\sumdy_{K'\leqslant K}dK'\big|\mathcal{N}_{\bm{\epsilon}_Z}[d]\cap[K',2K']\big|(t-t_0)^2\Big),\\
S_2[d]&\leqslant\sum_{\substack{n=dn'\in\mathcal{N}_{\bm{\epsilon}_Z}\\(n',b_Z/d)=1\\|n|>dK}}|g_n(t)-g_n(t_0)|\ll\sumdy_{K'\geqslant K}\frac{|\mathcal{N}_{\bm{\epsilon}_Z}[d]\cap[K',2K']|}{dK'}.
\end{align*}
Now, $n'\in\mathcal{N}_{\bm{\epsilon}_Z}^{+}[d]$ if and only if $p\mid n'\,\Rightarrow\,p\in\mathcal{P}_{\bm{\epsilon}_Z}[d]$. Moreover, denoting by $B=[0,2\varphi(b/d)-1)^{\mathcal{P}_{\bm{\epsilon}_Z}[d]}$ the cube with edges indexed by $p\in\mathcal{P}_{\bm{\epsilon}_Z}[d]$ and all edge lengths $2\varphi(b/d)$ and
\begin{align*}
\mathcal{M}_{\bm{\epsilon}_Z}[d](K)&=\Big\{\mathbf{m}=(m_p)\geqslant 0\,(p\in\mathcal{P}_{\bm{\epsilon}_Z}[d]):\sum_{p\in\mathcal{P}_{\bm{\epsilon}_Z}[d]}(\log p)m_p\leqslant\log K\Big\},\\
\mathcal{M}_{\bm{\epsilon}_Z}[d](K)^{\circ}&=\Big\{\mathbf{m}=(m_p)\geqslant 0\,(p\in\mathcal{P}_{\bm{\epsilon}_Z}[d]):2\varphi(b/d)\mid m_p,\,\,\mathbf{m}+B\subseteq\mathcal{M}_{\bm{\epsilon}_Z}[d](K)\Big\},
\end{align*}
we have that $|\mathcal{N}_{\bm{\epsilon}_Z}^{+}[d]\cap[1,K]|=|\mathcal{M}_{\bm{\epsilon}_Z}[d](K)|$ as well as
\begin{align*}
&\sum_{\substack{n'\in\mathcal{N}^{+}_{\bm{\epsilon}_Z}[d]\\n'\leqslant K}}\epsilon_{dn'}e\Big(\frac{an'}{b/d}\Big)
=\epsilon_d\sum_{(m_p)\in\mathcal{M}_{\bm{\epsilon}_Z}[d](K)}
e\bigg(\frac{a}{b^Z(b_Z/d)}\prod_{p\in\mathcal{P}_{\bm{\epsilon}_Z}[d]}p^{m_p}\bigg)\prod_{p\in\mathcal{P}_{\bm{\epsilon}_Z}[d]}\epsilon_p^{m_p}\\
&\qquad=\epsilon_d\sum_{(m_p)\in\mathcal{M}_{\bm{\epsilon}_Z}[d](K)^{\circ}}
\sum_{\substack{0\leqslant\mu_p<2\varphi(b/d)\\(p\in\mathcal{P}_{\bm{\epsilon}_Z}[d])}}e\bigg(\frac{a}{b^Z(b_Z/d)}\prod_{p\in\mathcal{P}_{\bm{\epsilon}_Z}[d]}p^{\mu_p}\bigg)\prod_{p\in\mathcal{P}_{\bm{\epsilon}_Z}[d]}\epsilon_p^{\mu_p}\\
&\qquad\qquad\qquad+\bigO\Big(\big|\mathcal{M}_{\bm{\epsilon}_Z}[d](K)\setminus(\mathcal{M}_{\bm{\epsilon}_Z}[d](K)^{\circ}+B)\big|\Big)\\
&\qquad=\frac{\epsilon_d}{|\mathcal{P}_{\bm{\epsilon}_Z}[d]|!}\frac{(2\varphi(b/d))^{|\mathcal{P}_{\bm{\epsilon}_Z}[d]|}s^{\ast}(a/(b/d);\bm{\epsilon}_Z)}{\prod_{p\in\mathcal{P}_{\bm{\epsilon}_Z}[d]}(2\varphi(b/d)\log p)}{(\log K)^{|\mathcal{P}_{\bm{\epsilon}_Z}[d]|}}+\bigO_{Z,t_0}\big((\log K)^{|\mathcal{P}_{\bm{\epsilon}_Z}[d]|-1}\big).
\end{align*}
Putting everything together, we have proved that
\begin{align*}
G^{\sharp}_{\bm{\epsilon}_Z}[d](t)-G^{\sharp}_{\bm{\epsilon}_Z}[d](t_0)
&=(t-t_0)e(t_0)
c_{\bm{\epsilon}_Z}[d]\cdot
2\mRe s^{\ast}\Big(\frac{a}{b/d};\bm{\epsilon}_Z\Big)(\log K)^{|\mathcal{P}_{\bm{\epsilon}_Z}[d]|}\\
&\qquad+\bigO_{Z,t_0}\Big((\log K)^{|\mathcal{P}_{\bm{\epsilon}_Z}[d]|-1}\Big(|t-t_0|+K(t-t_0)^2+\frac1K\Big)\Big),
\end{align*}
where
\begin{equation}
\label{cd-toquote}
c_{\bm{\epsilon}_Z}[d]=\epsilon_d/\big(|\mathcal{P}_{\bm{\epsilon}_Z}[d]|!\prod\nolimits_{p\in\mathcal{P}_{\bm{\epsilon}_Z}[d]}(\log p)\big).
\end{equation}
The claim of the lemma follows upon choosing $K=|t-t_0|^{-1}$ to balance the error terms.
\end{proof}

In view of  Lemma~\ref{local-d-lemma}, the behavior of the deterministic path $G^{\sharp}_{\bm{\epsilon}_Z}(t)$ close to various rational points is guided at first by whether
\begin{equation}
\label{nonvanishing-s}
\mRe s^{\ast}(a/b;\bm{\epsilon}_Z)\neq 0.
\end{equation}
or not. In particular, close to a rational point $t_0=a/b\in\mathbb{Q}$, the deterministic path $G^{\sharp}_{\bm{\epsilon}_Z}(t)$ splits as in \eqref{decomposition-by-divisors} as the sum of components $G^{\sharp}_{\bm{\epsilon}_Z}[d](t)$ over $d\mid b_Z$, with the component $G^{\sharp}_{\bm{\epsilon}_Z}[d]$ having a logarithmic singularity of order $|\log|t-t_0||^{|\mathcal{P}_{\bm{\epsilon}_Z}[d]|}$ as $t\to t_0$ whenever $\mRe s^{\ast}(a/(b/d);\bm{\epsilon}_Z)\neq 0$. The property of $G^{\sharp}_{\bm{\epsilon}_Z}(t)$ having a logarithmic singularity at $t_0=a/b$ should \emph{not} be confused with this path having a cusp at $t_0$, which is substantially more delicate and studied in \S\ref{finer-local-subsec}.

The nonvanishing condition \eqref{nonvanishing-s} is immediate to numerically check for every specific pair $(a/b,\bm{\epsilon}_Z)$, and in \S\ref{expsums-nv-mult} we offer a more detailed analysis of this fascinating question in more generality. In particular, if $\mRe s^{\ast}(a/b^Z;\bm{\epsilon}_Z)\neq 0$, then the full path has a logarithmic singularity of order $|\log|t-t_0||^{|\mathcal{P}_{\bm{\epsilon}_Z}|}$; if $\mRe s^{\ast}(a/b^Z;\bm{\epsilon}_Z)=0$ but $\mRe s^{\ast}(a/(b^Z(b_Z/d));\bm{\epsilon}_Z)\neq 0$ for some $d\mid b_Z$, then the component $G^{\sharp}_{\bm{\epsilon}_Z}[d](t)$ has a logarithmic singularity of lesser severity as $t\to t_0$, which may or may not be inherited by the full path $G^{\sharp}_{\bm{\epsilon}_Z}(t)$.

\subsection{Finer local information}
\label{finer-local-subsec}
Already Lemma~\ref{local-d-lemma} clearly shows that the local behavior of $G^{\sharp}_{\bm{\epsilon}_Z}$ near $t_0=a/b$ is, at the first order, guided by whether $\mRe s^{\ast}(a/b^Z;\bm{\epsilon}_Z)\neq 0$ or not. This condition is analyzed in more detail in \S\ref{expsums-nv-mult}. In the case of nonvanishing, we see that
\[ \lim_{t\to t_0\pm}\frac{G^{\sharp}_{\bm{\epsilon}_Z}(t)-G^{\sharp}_{\bm{\epsilon}_Z}(t_0)}{e(t_0)(t-t_0)}=\infty\text{ or }-\infty, \]
according to the sign of $\mRe s^{\ast}(a/b^Z;\bm{\epsilon}_Z)\in\mathbb{R}_{\neq 0}$, indicating (in contrast to the conclusion of Theorem~\ref{thm2}\eqref{thm2-gsharpZ}) that the function $G^{\sharp}_{\bm{\epsilon}_Z}(t)$ is not differentiable at $t=t_0=a/b$ \emph{but} the path $G^{\sharp}_{\bm{\epsilon}_Z}$ appears smooth around the point $G^{\sharp}_{\bm{\epsilon}_Z}(t_0)$, in the sense that all of its Dini quotients vanish (to an arbitrarily high degree on the logarithmic scale). Such points are clearly observable (but not specifically marked) in Figure~\ref{fig: gsharp-eps}; check, for example, neighborhoods of $t_0=0,\frac12,1$.

To identify points $t_0=a/b$ where the path $G^{\sharp}_{\bm{\epsilon}_Z}(t)$ exhibits cusp behavior, we need to be able to consider the cases where $\mRe s^{\ast}(a/b^Z;\bm{\epsilon}_Z)=0$ and analyze lower-order local behavior. From now on, throughout the rest of Section~\ref{sec-atlas}, we consider the case $b_Z=d=1$, so that $a/b^Z=a/b$.

Fix once and for all an even nonnegative function $\phi\in C_c^{\infty}(\mathbb{R})$ such that $\phi(x)=1$ for $x\in [-1,1]$ and $\phi(x)=0$ for $x\not\in[-2,2]$. By a familiar repeated integration by parts argument, for every $m\in\mathbb{N}$, the Mellin transform $\widetilde{\phi}(s)$ has a meromorphic continuation to $\mRe(s)>-m$ given by
\begin{equation}
\label{mellin-continuation}
\widetilde{\phi}(s)=\frac{(-1)^m}{s(s+1)\cdots(s+m-1)}\int_0^{\infty}\phi^{(m)}(x)x^{s+m-1}\,\dd x.
\end{equation}
In particular, we have the asymptotic expansions
\begin{alignat*}{3}
&\widetilde{\phi}(s)=\frac{\phi(0)}s+\sum_{\nu=0}^{\infty}c_{\nu,0}(\phi)s^{\nu}\quad &&(s\to 0),\quad c_{\nu,0}(\phi)=-\frac1{(\nu+1)!}\int_0^{\infty}\phi'(x)(\log x)^{\nu+1}\,\dd x,\\
&\tilde{\phi}(s)=\sum_{\nu=0}^{\infty}c_{\nu,\ell}(\phi)(s-\ell)^{\nu}\quad &&(s\to\ell),\quad c_{\nu,\ell}(\phi)=\frac1{\nu!}\int_0^{\infty}\phi(x)x^{\ell-1}(\log x)^{\nu}\,\dd x.
\end{alignat*}
for every $\ell>0$, as well as the uniform bounds
\begin{equation}
\label{rapid-decay}
\widetilde{\phi}(\ell+it)\ll_m2^{\ell}(2/(1+|t|))^m\quad (|t|\gg 1).
\end{equation}

For a large parameter $K>0$, to be suitably chosen later, we consider the (finite!) sum
\[ I_K(\phi)=\sum_{n\in\mathcal{N}_{\bm{\epsilon}_Z}}\epsilon_n(g_n(t)-g_n(t_0))\phi\Big(\frac{|n|}K\Big). \]
Using the absolutely and uniformly convergent Taylor series expansion for $e(n\tau)$, we have
\begin{equation}
\label{IKphi}
\begin{aligned}
&I_K(\phi)=\int_{t_0}^te(\tau)\sum_{n\in\mathcal{N}_{\bm{\epsilon}_Z}}\epsilon_ne(n\tau)\phi\Big(\frac{|n|}K\Big)\dd\tau\\
&=e(t_0)\sum_{n\in\mathcal{N}_{\bm{\epsilon}_Z}}\epsilon_n\sum_{k=0}^{\infty}\frac{(2\pi in)^k}{k!}e(nt_0)\frac{(t-t_0)^{k+1}}{k+1}\phi\Big(\frac{|n|}K\Big)+\OO\Big((t-t_0)^2\!\sum_{n\in\mathcal{N}_{\bm{\epsilon}_Z}}\!\phi\Big(\frac{|n|}K\Big)\Big)\\
&=e(t_0)(t-t_0)\sum_{k=0}^{\infty}\frac{(2\pi i)^k}{(k+1)!}I_{K,k}(\phi)(t-t_0)^k+\OO\big((t-t_0)^2(\log K)^{|\mathcal{P}_{\bm{\epsilon}_Z}|}\big),
\end{aligned}
\end{equation}
where
\[ I_{K,k}(\phi)=\sum_{n\in\mathcal{N}_{\bm{\epsilon}_Z}^{+}}\epsilon_n\big(e(na/b)+(-1)^ke(-na/b)\big)n^k\phi\Big(\frac nK\Big). \]
Denoting as usual $G(\chi)=\sum^{\ast}_{x\bmod b}\chi(x)e(x/b)$ the unnormalized Gauss sum of a (not necessarily primitive) character $\chi$ modulo $b$ and using the discrete and archimedean Mellin transforms, we find that, for any $\sigma>k$,
\begin{equation}
\label{IKk}
\begin{aligned}
I_{K,k}(\phi)
&=\frac{2}{\varphi(b)}\sum_{\substack{\chi\bmod b\\\chi(-1)=(-1)^k}}\overline{G(\chi)}\chi(-a)\sum_{n\in\mathcal{N}_{\bm{\epsilon}_Z}^{+}}\epsilon_n\chi(n)\frac1{2\pi i}\int_{(\sigma)}\widetilde{\phi}(s)n^k(n/K)^{-s}\,\dd s\\
&=\frac{2}{\varphi(b)}\sum_{\substack{\chi\bmod b\\\chi(-1)=(-1)^k}}\overline{G(\chi)}\chi(-a)\frac1{2\pi i}\int_{(\sigma)}\widetilde{\phi}(s)\prod_{p\in\mathcal{P}_{\bm{\epsilon}_Z}}\Big(1-\frac{\epsilon_p\chi(p)}{p^{s-k}}\Big)^{-1}K^s\,\dd s.
\end{aligned}
\end{equation}
The function $\gamma(\chi,s)=\prod_{p\in\mathcal{P}_{\bm{\epsilon}_Z}}(1-\epsilon_p\chi(p)/p^s)^{-1}$ continues to a meromorphic function with a pole at $s=0$ of order at most $|\mathcal{P}_{\bm{\epsilon}_Z}|$ and an asymptotic expansion
\begin{equation}
\label{gammas}
\begin{gathered}
\gamma(\chi,s)=P_{\bm{\epsilon}_Z}\sum_{\nu=0}^{\infty}\gamma_{\nu}(\chi)s^{-|\mathcal{P}_{\bm{\epsilon}_Z}|+\nu}\quad (s\to 0),\qquad
\gamma_0(\chi)=\mathbf{1}_{\substack{\chi(p)=\epsilon_p\\(p\in\mathcal{P}_{\bm{\epsilon}_Z})}},\\ \gamma_1(\chi)=\frac12\mathbf{1}_{\substack{\chi(p)=\epsilon_p\\(p\in\mathcal{P}_{\bm{\epsilon}_Z})}}\sum_{p\in\mathcal{P}_{\bm{\epsilon}_Z}}(\log p)+\sum_{q\in\mathcal{P}_{\bm{\epsilon}_Z}}\mathbf{1}_{\substack{\chi(p)=\epsilon_p\\(p\in\mathcal{P}_{\bm{\epsilon}_Z}^{q-}),\\\chi(q)\neq\epsilon_q}}(\log q)(1-\chi(q)\epsilon_q)^{-1},
\end{gathered}
\end{equation}
where $P_{\bm{\epsilon}_Z}=1/{\prod_{p\in\mathcal{P}_{\bm{\epsilon}_Z}}(\log p)}$. In addition to the pole at $s=0$, we encounter in the evaluation of $I_{K,k}(\phi)$ poles whenever
\[ s_{k,p,\ell,\chi}=k+\Big(\ell+\frac {m_p}{\varphi(b)}\Big)\frac{2\pi i}{\log p},\quad (p\in\mathcal{P}_{\bm{\epsilon}_Z},\,\,\epsilon_p\chi(p)=e^{2\pi i m_p/\varphi(b)},\,\,\ell\in\mathbb{Z}), \]
and these poles are simple and pairwise distinct except possibly for the pole at $s=k$. When $k\neq 0$, we also encounter a distinct pole of $\widetilde{\phi}(s)$ at $s=0$; this then accounts for all the poles of the integrand in $\mRe s>-1$. We will show that the total contributions of these poles converge absolutely and that the contour of integration in \eqref{IKk} may be shifted to $\mRe s=-\delta$ for a suitable $0<\delta<1$, and we will write
\begin{equation}
\label{IKkphi-decomp}
\begin{aligned}
&I_{K,k}(\phi)=I^0_{K,k}(\phi,b)+I^1_{K,k}(\phi,b)+I^2_{K,k}(\phi,b)+I^3_{K,k}(\phi,b)\\
&\quad:=\frac{2(-1)^k}{\varphi(b)}\sum_{\substack{\chi\bmod b\\\chi(-1)=(-1)^k}}\overline{G(\chi)}\chi(a)\times\\
&\qquad\times\bigg(\Res_{s=k}+\sum_{\substack{p\in\mathcal{P}_{\bm{\epsilon}_Z},\,\ell\in\mathbb{Z}\\s_{k,p,\ell,\chi}\neq k}}\Res_{s=s_{k,p,\ell,\chi}}+\delta_{k\neq 0}\Res_{s=0}+\frac1{2\pi i}\int_{(-\delta)}\bigg)(\widetilde{\phi}(s)\gamma(\chi,s-k)K^s).
\end{aligned}
\end{equation}

We begin by evaluating
\begin{align*}
\Res_{s=k}(\widetilde{\phi}(s)\gamma(\chi,s-k)K^s)
&=\delta_{k0}\phi(0)P_{\bm{\epsilon}_Z}\sum_{\nu=0}^{|\mathcal{P}_{\bm{\epsilon}_Z}|}\gamma_{\nu}(\chi)K^k\frac{(\log K)^{|\mathcal{P}_{\bm{\epsilon}_Z}|-\nu}}{(|\mathcal{P}_{\bm{\epsilon}_Z}|-\nu)!}\\
&\qquad+P_{\bm{\epsilon}_Z}\sum_{\nu=0}^{|\mathcal{P}_{\bm{\epsilon}_Z}|-1}\sum_{\substack{\nu_1,\nu_2\geqslant 0\\\nu_1+\nu_2=\nu}}c_{\nu_1,k}(\phi)\gamma_{\nu_2}(\chi)K^k\frac{(\log K)^{|\mathcal{P}_{\bm{\epsilon}_Z}|-1-\nu}}{(|\mathcal{P}_{\bm{\epsilon}_Z}|-1-\nu)!}.
\end{align*}
The total contribution of these residues to $I_{K,k}(\phi)$ in \eqref{IKk} equals
\begin{equation}
\label{IKk0}
I_{K,k}^0(\phi,b)=K^kP_{\bm{\epsilon}_Z}\sum_{\nu=0}^{|\mathcal{P}_{\bm{\epsilon}_Z}|}\iota_{k,\nu}(\phi,b)\frac{(\log K)^{|\mathcal{P}_{\bm{\epsilon}_Z}|-\nu}}{(|\mathcal{P}_{\bm{\epsilon}_Z}|-\nu)!},
\end{equation}
where $\iota_{k,\nu}(\phi,b)$ are arithmetic functions given by
\begin{equation}
\label{iota-j}
\begin{gathered}
\iota_{k,\nu}(\phi,b)=\delta_{k0}\phi(0)j_{\nu,0}(b)+\sum_{\substack{\nu_1,\nu_2\geqslant 0\\\nu_1+\nu_2=\nu-1}}c_{\nu_1,k}(\phi)j_{\nu_2,k}(b),\\
j_{\nu,k}(b)=\frac{2}{\varphi(b)}\sum_{\substack{\chi\bmod b\\\chi(-1)=(-1)^k}}\overline{G(\chi)}\chi(-a)\gamma_{\nu}(\chi),
\end{gathered}
\end{equation}
and we note that $j_{\nu,k}(b)$ depends only on $k\bmod 2$.

We proceed to estimate each of the remaining contributions $I_{K,k}^j(\phi,b)$ ($1\leqslant j\leqslant 3$). At each of the poles $s_{k,p,\ell,m}\neq k$, we have for $q\in\mathcal{P}_{\bm{\epsilon}_Z}$, $q\neq p$, that
\begin{equation}
\label{Baker-lower}
\begin{aligned}
&\Big|1-\frac{\epsilon_q\chi(q)}{q^{s_{k,p,\ell,\chi}-k}}\Big|
=\big|e^{2\pi i(\log q/\log p)(\ell+m_p/\varphi(b))}-e^{2\pi im_q/\varphi(b)}\big|\\
&\qquad\asymp_b\min_{\ell'\in\mathbb{Z}}\big|(\log q)(\ell\varphi(b)+m_p)-(\log p)(\ell'\varphi(b)+m_q)\big|\gg_{Z,b}\frac1{(1+|\ell|)^C}
\end{aligned}
\end{equation}
for a fixed constant $C$ depending on $Z$ only, by using Baker's theorem on linear forms in logarithms (Proposition~\ref{Baker}). Therefore, also using the uniform bound \eqref{rapid-decay} for $\widetilde{\phi}(s)$, the total contribution of these poles is at most
\begin{align*}
I^1_{K,k}(\phi,b)&\ll_{Z,b}\sum_{\substack{\chi\bmod b\\\chi(-1)=(-1)^k}}\sum_{p\in\mathcal{P}_{\bm{\epsilon}_Z}}\sum_{\ell_p\in\mathbb{Z}\setminus\{0\}}\Big|\mathop{\mathrm{Res}}_{s=s_{k,p,\ell,\chi}}\Big(\widetilde{\phi}(s)\prod_{p\in\mathcal{P}_{\bm{\epsilon}_Z}}\Big(1-\frac{\epsilon_p\chi(p)}{p^{s-k}}\Big)^{-1}K^s\Big)\Big|\\
&\ll_N\sum_{p\in\mathcal{P}_{\bm{\epsilon}_Z}}\sum_{\ell_p\in\mathbb{Z}\setminus\{0\}}\frac{(2K)^k2^{N}}{|\ell_p|^N}|\ell_p|^{C'}\ll_Z(2K)^k,
\end{align*}
by choosing $N\in(C'+2,C'+3]$. We emphasize that this is only a preliminary bound, whose primary role is to ensure absolute convergence; we will be estimating the combined contributions of all $I_{K,k}^1(\phi,b)$ far more delicately. Using the same Proposition~\ref{Baker}, we also see that the total contributions of the integrals over the horizontal segments $[\delta+it,\sigma+it]$ may be bounded by $\ll_{Z,b}(\sigma-\delta)(2K)^k/(2+|t|)^{C'-N}\to 0$ over a suitable sequence of $|t|\to\infty$, and thus we may indeed shift the vertical contour past the line $\mRe s=k$.

When $k\neq 0$, we also collect the contribution from the simple pole at $s=0$ equal to
\[ I^2_{K,k}(\phi,b)=\phi(0)\frac{2}{\varphi(b)}\sum_{\substack{\chi\bmod b\\\chi(-1)=(-1)^k}}\overline{G(\chi)}\chi(-a)\prod_{p\in\mathcal{P}_{\bm{\epsilon}_Z}}(1-p^k\epsilon_p\chi(p))^{-1}. \]
Finally, using the uniform bound \eqref{rapid-decay}, the total contribution of the remaining integrals over $\mRe s=-\delta$ is easily seen to be
\[ I^3_{K,k}(\phi,b)\ll_{Z,b}\int_{\mathbb{R}}\frac{C^kK^{-\delta}}{(1+|t|)^2}\,\dd t\ll C^kK^{-\delta}, \]
where $C=\prod_{p\in\mathcal{P}_{\bm{\epsilon}_Z}}p$.

We now return to \eqref{IKphi} and compute the summands in the decomposition
\[ I_K(\phi)=I_K^0(\phi)+I_K^1(\phi)+I_K^2(\phi)+I_K^3(\phi) \]
corresponding to the total contributions of the four summands in \eqref{IKkphi-decomp}. Using \eqref{IKk0}, we compute
\begin{equation}
\label{IKphi-1}
I^0_K(\phi)=e(t_0)(t-t_0)P_{\bm{\epsilon}_Z}\sum_{\nu=0}^{|\mathcal{P}_{\bm{\epsilon}_Z}|}\iota_{\nu}(\phi,b,K(t-t_0))\frac{(\log K)^{|\mathcal{P}_{\bm{\epsilon}_Z}|-\nu}}{(|\mathcal{P}_{\bm{\epsilon}_Z}|-\nu)!},
\end{equation}
where
\[ \iota_{\nu}(\phi,b,K(t-t_0))=\phi(0)j_{\nu,0}(b)+\!\sum_{\substack{\nu_1,\nu_2\geqslant 0\\\nu_1+\nu_2=\nu-1}}\!\big(c^0_{\nu_1}(\phi,K(t-t_0))j_{\nu_2,0}(b)+c_{\nu_1}^1(\phi,K(t-t_0))j_{\nu_2,1}(b)\big), \]
and
\begin{align*}
c^0_{\nu_1}(\phi,K(t-t_0))&=\sum_{2\mid k}\frac{(2\pi iK)^k}{(k+1)!}(t-t_0)^kc_{\nu_1,k}(\phi)=-\frac1{(\nu_1+1)!}\int_0^{\infty}\phi'(x)(\log x)^{\nu_1+1}\,\dd x\\
&\qquad\qquad+\frac1{\nu_1!}\int_0^{\infty}\phi(x)(\log x)^{\nu_1}\psi_0(xK(t-t_0))\frac{\dd x}x,\\
c^1_{\nu_1}(\phi,K(t-t_0))&=\sum_{2\nmid k}\frac{(2\pi iK)^k}{(k+1)!}(t-t_0)^kc_{\nu_1,k}(\phi)=\frac1{\nu_1!}\int_0^{\infty}\!\phi(x)(\log x)^{\nu_1}\psi_1(xK(t-t_0))\frac{\dd x}x,
\end{align*}
with the kernels $\psi_0$ and $\psi_1$ given by
\begin{equation}
\label{kernels}
\psi_0(x)=\sum_{2\mid k\geqslant 2}\frac{(2\pi ix)^k}{(k+1)!}=\frac{\sin 2\pi x}{2\pi x}-1,\quad \psi_1(x)=\sum_{2\nmid k}\frac{(2\pi ix)^k}{(k+1)!}=\frac{\cos 2\pi x-1}{2\pi ix}.
\end{equation}

The total contribution of terms with the factor $j_{\nu_2,1}(b)$ to the sum in \eqref{IKphi-1} is $j_{\nu_2,1}(b)$ times
\begin{align*}
&\int_0^{\infty}\phi(x)\psi_1(xK(t-t_0))\sum_{\nu=1+\nu_2}^{|\mathcal{P}_{\bm{\epsilon}_Z}|}\frac{(\log K)^{|\mathcal{P}_{\bm{\epsilon}_Z}|-\nu}(\log x)^{\nu-1-\nu_2}}{(|\mathcal{P}_{\bm{\epsilon}_Z}|-\nu)!(\nu-1-\nu_2)!}\,\frac{\dd x}x\\
&\qquad=\frac1{(|\mathcal{P}_{\bm{\epsilon}_Z}|-1-\nu_2)!}\int_0^{\infty}\phi(x)(\log(Kx))^{|\mathcal{P}_{\bm{\epsilon}_Z}|-1-\nu_2}\psi_1(Kx(t-t_0))\,\frac{\dd x}x\\
&\qquad=\sgn(t-t_0)P^1_{|\mathcal{P}_{\bm{\epsilon}_Z}|-1-\nu_2}(\log|t-t_0|)+\OO\Big(\frac{(\log K)^{|\mathcal{P}_{\bm{\epsilon}_Z}|-1-\nu_2}}{K|t-t_0|}\Big),
\end{align*}
where $P^1_{\nu}(L)$ is the degree $\nu$ polynomial given by the absolutely convergent integral
\begin{equation}
\label{Pnu1}
P^1_{\nu}(L)=\frac1{\nu!}\int_0^{\infty}(\log\xi-L)^{\nu}\psi_1(\xi)\,\frac{\dd\xi}{\xi}.
\end{equation}
Similarly, the total contribution of terms with the factor $j_{\nu_2,0}(b)$ is $j_{\nu_2,0}(b)$ times
\begin{align*}
&\phi(0)\frac{(\log K)^{|\mathcal{P}_{\bm{\epsilon}_Z}|-\nu_2}}{(|\mathcal{P}_{\bm{\epsilon}_Z}|-\nu_2)!}-\int_0^{\infty}\phi'(x)\sum_{\nu=1+\nu_2}^{|\mathcal{P}_{\bm{\epsilon}_Z}|}\frac{(\log K)^{|\mathcal{P}_{\bm{\epsilon}_Z}|-\nu}}{(|\mathcal{P}_{\bm{\epsilon}_Z}|-\nu)!}\frac{(\log x)^{\nu-\nu_2}}{(\nu-\nu_2)!}\,\dd x\\
&\qquad\qquad +\int_0^{\infty}\phi(x)\psi_0(xK(t-t_0))\sum_{\nu=1+\nu_2}^{|\mathcal{P}_{\bm{\epsilon}_Z}|}\frac{(\log K)^{|\mathcal{P}_{\bm{\epsilon}_Z}|-\nu}}{(|\mathcal{P}_{\bm{\epsilon}_Z}|-\nu)!}\frac{(\log x)^{\nu-1-\nu_2}}{(\nu-1-\nu_2)!}\,\frac{\dd x}x\\
&\qquad=\frac1{(|\mathcal{P}_{\bm{\epsilon}_Z}|-\nu_2)!}\bigg(-\int_0^{\infty}\phi'(x)(\log(Kx))^{|\mathcal{P}_{\bm{\epsilon}_Z}|-\nu_2}\big(1+\psi_0(xK(t-t_0))\big)\,\dd x\\
&\qquad\qquad -K(t-t_0)\int_0^{\infty}(\log(Kx))^{|\mathcal{P}_{\bm{\epsilon}_Z}|-\nu_2}\phi(x)\psi_0'(xK(t-t_0))\,\dd x\bigg)\\
&\qquad=P^0_{|\mathcal{P}_{\bm{\epsilon}_Z}|-\nu_2}(\log|t-t_0|)+\OO\Big(\frac{(\log K)^{|\mathcal{P}_{\bm{\epsilon}_Z}|-\nu_2-1}}{K|t-t_0|}\Big),
\end{align*}
by a little calculation using integration by parts, where $P^0_{\nu}(L)$ is the degree $\nu$ polynomial given by the conditionally convergent improper integral
\begin{equation}
\label{Pnu0}
P^0_{\nu}(L)=-\frac1{\nu!}\int_0^{\infty}(\log\xi-L)^{\nu}\psi'_0(\xi)\,\dd\xi.
\end{equation}
Indeed, the error term we encounter in the final line equals
\begin{align*}
&\frac1{(|\mathcal{P}_{\bm{\epsilon}_Z}|-\nu_2)!}\int_0^{\infty}(\log(Kx))^{|\mathcal{P}_{\bm{\epsilon}_Z}|-\nu_2}\frac{\dd}{\dd x}\big((1-\phi(x))\big(1+\psi_0(xK(t-t_0))\big)\,\dd x\\
&\qquad\ll_Z\int_1^{\infty}\frac1{xK|t-t_0|}(\log(Kx))^{|\mathcal{P}_{\bm{\epsilon}_Z}|-\nu_2-1}\,\frac{\dd x}x\ll\frac{(\log K)^{|\mathcal{P}_{\bm{\epsilon}_Z}|-\nu_2-1}}{K|t-t_0|}.
\end{align*}

Putting everything together into \eqref{IKphi-1}, we find that
\begin{align*}
I_K^0(\phi)&=e(t_0)(t-t_0)P_{\bm{\epsilon}_Z}\sum_{\nu=0}^{|\mathcal{P}_{\bm{\epsilon}_Z}|}j_{\nu,0}(b)P^0_{|\mathcal{P}_{\bm{\epsilon}_Z}|-\nu}(\log|t-t_0|)\\
&\qquad +e(t_0)|t-t_0|P_{\bm{\epsilon}_Z}\sum_{\nu=0}^{|\mathcal{P}_{\bm{\epsilon}_Z}|-1}j_{\nu,1}(b)P^1_{|\mathcal{P}_{\bm{\epsilon}_Z}|-1-\nu}(\log|t-t_0|)+\OO_{b,Z}\Big(\frac{(\log K)^{|\mathcal{P}_{\bm{\epsilon}_Z}|-1}}{K}\Big).
\end{align*}

Next, we proceed to estimate $I_K^1(\phi)$, the combined contribution of the poles at $s=s_{k,p,\ell,\chi}\neq k$. Using \eqref{mellin-continuation} with any $m\geqslant 1$, we find that, for every $s_{k,p,\ell,\chi}=k+i\gamma$, $k\geqslant 0$, $\gamma=\gamma_{p,\ell,\chi}\neq 0$,
\begin{align*}
&\Res_{s=k+i\gamma}\Big(\widetilde{\phi}(s)\prod_{p\in\mathcal{P}_{\bm{\epsilon}_Z}}\Big(1-\frac{\epsilon_p\chi(p)}{p^{s-k}}\Big)^{-1}K^s\Big)=\frac1{\log p}\prod_{\substack{q\in\mathcal{P}_{\bm{\epsilon}_Z}\\q\neq p}}\Big(1-\frac{\epsilon_q\chi(q)}{q^{i\gamma}}\Big)^{-1}K^{k+i\gamma}\times\\
&\qquad\qquad\times\frac{(-1)^m}{(k+i\gamma)(k+i\gamma+1)\cdots(k+i\gamma+m-1)}\int_0^{\infty}\phi^{(m)}(x)x^{k+i\gamma+m-1}\,\dd x.
\end{align*}
Using the elementary integral expression for the beta function and substituting into \eqref{IKkphi-decomp}, we find that
\[ I_K^1(\phi)=e(t_0)(t-t_0)\frac{(-1)^m}{(m-1)!}\sum_{\kappa\in\{0,1\}}\frac{2}{\varphi(b)}\sum_{\substack{\chi\bmod b\\\chi(-1)=(-1)^{\kappa}}}\overline{G(\chi)}\chi(-a)J_{\kappa,\chi,K}(\phi), \]
where
\[ J_{\kappa,\chi,K}(\phi)=\mathop{\sum\sum}_{\substack{p\in\mathcal{P}_{\bm{\epsilon}_Z},\,\ell\in\mathbb{Z}\\\gamma_{p,\ell,\chi}\neq 0}}
\frac1{\log p}\prod_{\substack{q\in\mathcal{P}_{\bm{\epsilon}_Z}\\q\neq p}}\Big(1-\frac{\epsilon_q\chi(q)}{q^{i\gamma_{p,\ell,\chi}}}\Big)^{-1}J_{\kappa,\chi,K}^{(p,\ell)}(\phi), \]
and
\begin{align*}
J_{\kappa,\chi,K}^{(p,\ell)}(\phi)&=\int_0^1y^{k+i\gamma_{p,\ell,\chi}-1}(1-y)^{m-1}
\int_0^{\infty}\phi^{(m)}(x)x^{k+i\gamma_{p,\ell,\chi}+m-1}\times\\
&\qquad\times
\sum_{k\equiv\kappa\,(2)}\frac{(2\pi i)^k}{(k+1)!}(t-t_0)^kK^{k+i\gamma_{p,\ell,\chi}}\,\dd x\,\dd y\\
&=\int_0^1\int_0^{\infty}y^{-1}(1-y)^{m-1}\phi^{(m)}(x)x^{m-1}(xKy)^{i\gamma_{p,\ell,\chi}}\widetilde{\psi_{\kappa}}(2\pi(t-t_0)xKy)\,\dd x\,\dd y,
\end{align*}
where $\widetilde{\psi_0}=\psi_0+1$ and $\widetilde{\psi_1}=\psi_1$, with the kernels $\psi_0$ and $\psi_1$ as in \eqref{kernels}. Using the integration by parts in the $x$-variable $N$ times, we find that
\begin{align*}
J_{\kappa,\chi,K}^{(p,\ell)}(\phi)&=\frac{(-1)^N}{(i\gamma_{p,\ell,\chi})\cdots(i\gamma_{p,\ell,\chi}+N-1)}\int_0^1\int_0^{\infty}y^{-1}(1-y)^{m-1}\times\\
&\qquad\times (xKy)^{i\gamma_{p,\ell,\chi}}x^{N-1}\Big(\frac{\dd}{\dd x}\Big)^{(N)}\big(\phi^{(m)}(x)x^{m}\widetilde{\psi_{\kappa}}(2\pi(t-t_0)xKy)\big)\,\dd x\,\dd y\\
&\ll_{m,N}\frac{(|t-t_0|K)^{N-1}}{|\gamma_{p,\ell,\chi}|^N}. \end{align*}
For clarity, we note that this calculation can be performed equally well with any $m\geqslant 1$ (even with $m=1$); if we choose $m$ sufficiently large, then a similar conclusion can be reached by integration by parts in the $y$-variable $N\leqslant m-1$ times. Inserting these estimates above and using Baker's bound as in \eqref{Baker-lower}, we find that
\[ I^1_K(\phi)\ll_{m,N,b,Z} |t-t_0|\sum_{\chi\bmod b}\mathop{\sum\sum}_{\substack{p\in\mathcal{P}_{\bm{\epsilon}_Z},\,\ell\in\mathbb{Z}\\\gamma_{p,\ell,\chi}\neq 0}}(1+|\ell|)^{C'}\frac{(|t-t_0|K)^{N-1}}{(1+|\ell|)^N}\ll_{b,Z}|t-t_0|(|t-t_0|K)^{C''}, \]
by choosing $N\in(C'+2,C'+3]$ and $C''=C'+2+\mathbf{1}_{|t-t_0|K\geqslant 1}$.

Finally, we address the contributions of $I_K^2(\phi)$ and $I_K^3(\phi)$. Indeed, $I_K^2(\phi)$ is a smooth function given by the absolutely and convergent series
\[ I_K^2(\phi)=e(t_0)(t-t_0)\sum_{k=1}^{\infty}\frac{(2\pi i)^k}{(k+1)!}I_{K,k}^2(\phi)(t-t_0)^k\ll_{Z,b}|t-t_0|^{2}\phi(0), \]
and
\[ I_K^3(\phi)\ll|t-t_0|\sum_{k=0}^{\infty}\frac{(2\pi C)^k}{(k+1)!}|t-t_0|^kK^{-\delta}\ll |t-t_0|K^{-\delta}, \] where we may choose, say, $\delta=\frac12$ for simplicity.

In total, we have proved that
\begin{align*}
I_K(\phi)&=
e(t_0)(t-t_0)P_{\bm{\epsilon}_Z}\sum_{\nu=0}^{|\mathcal{P}_{\bm{\epsilon}_Z}|}j_{\nu,0}(b)P^0_{|\mathcal{P}_{\bm{\epsilon}_Z}|-\nu}(\log|t-t_0|)\\
&\qquad +e(t_0)|t-t_0|P_{\bm{\epsilon}_Z}\sum_{\nu=0}^{|\mathcal{P}_{\bm{\epsilon}_Z}|-1}j_{\nu,1}(b)P^1_{|\mathcal{P}_{\bm{\epsilon}_Z}|-1-\nu}(\log|t-t_0|)\\
&\qquad +\OO_{b,Z}\Big(|t-t_0|\Big(\frac{(\log K)^{|\mathcal{P}_{\bm{\epsilon}_Z}|-1}}{K|t-t_0|}+(|t-t_0|K)^{C''}+|t-t_0|+K^{-1/2}\Big)\Big).
\end{align*}
Critically, this yields a nontrivial asymptotic when $K$ is appreciably larger than $1/|t-t_0|$. Now, arguing as in the proof of Lemma~\ref{local-d-lemma}, we can also estimate
\begin{align*}
G^{\sharp}_{\bm{\epsilon}_Z}(t)-G^{\sharp}_{\bm{\epsilon}_Z}(t_0)-I_K(\phi)
&= \sum_{n\in\mathcal{N}_{\bm{\epsilon}_Z}}\epsilon_n(g_n(t)-g_n(t_0))\Big(1-\phi\Big(\frac{|n|}K\Big)\Big)\\
&\ll\mathop{\sum\nolimits^{\mathrm{dy}}}\limits_{K'\geqslant K}
\sum_{\substack{n\in\mathcal{N}_{\bm{\epsilon}_Z}\\K'\leqslant |n|<2K'}}\frac1{K'}\ll\frac{(\log K)^{|\mathcal{P}_{\bm{\epsilon}_Z}|-1}}{K}.
\end{align*}

Choosing
\[ K=\frac1{|t-t_0|}|\log |t-t_0||^{\delta},\quad \delta=\frac{|\mathcal{P}_{\bm{\epsilon}_Z}|-1}{C'+4}, \]
we conclude that
\begin{equation}
\label{concluding-asymptotic}
\begin{aligned}
G^{\sharp}_{\bm{\epsilon}_Z}(t)-G^{\sharp}_{\bm{\epsilon}_Z}(t_0)&=e(t_0)c^{+}(t_0)\cdot(t-t_0)\ell^{|\mathcal{P}_{\bm{\epsilon}_Z}|}\\
&\qquad +e(t_0)\big(c'(t_0)\cdot(t-t_0)+c^{-}(t_0)\cdot |t-t_0|\big)\ell^{|\mathcal{P}_{\bm{\epsilon}_Z}|-1}\\
&\qquad+\OO_{b,Z}\big(|t-t_0|\ell^{|\mathcal{P}_{\bm{\epsilon}_Z}|-1-\delta}\big),
\end{aligned}
\end{equation}
where $\ell=|\log|t-t_0||$, and the constants $c^{+}(t_0)$, $c'(t_0)$, and $c^{-}(t_0)$ (which also depend on $\bm{\epsilon}_Z$) may be read off from \eqref{gammas}, \eqref{iota-j}, \eqref{Pnu0}, and \eqref{Pnu1} as
\begin{equation}
\label{constants}
\begin{aligned}
&c^{+}(t_0)=
c_{\bm{\epsilon}_Z}s^{+}(a/b,\bm{\epsilon}_Z),\quad c^{-}(t_0)=
-c_{\bm{\epsilon}_Z}|\mathcal{P}_{\bm{\epsilon}_Z}|\frac{\pi}{2i}s^{-}(a/b,\bm{\epsilon}_Z),\\
&c'(t_0)=
c_{\bm{\epsilon}_Z}|\mathcal{P}_{\bm{\epsilon}_Z}|\Big(c^{\prime +}_{\bm{\epsilon}_Z}s^{+}(a/b,\bm{\epsilon}_Z)+\sum_{q\in\mathcal{P}_{\bm{\epsilon}_Z}}(\log q)s_q^{+}(a/b,\bm{\epsilon}_Z)\Big),\\
\end{aligned}
\end{equation}
where we also used the classical {evaluations
\[ \int_0^{\infty}\frac{\cos 2\pi x-1}{2\pi x}\frac{\dd x}x=-\frac{\pi}2,\quad \int_0^{\infty}(\log x)\Big(\frac{\sin 2\pi x}{2\pi x}-1\Big)'\,\dd x=\log 2\pi+\gamma-1, \]
kept in mind from \eqref{cd-toquote} the notation $c_{\bm{\epsilon}_Z}=c_{\bm{\epsilon}_Z}[1]=1/(|\mathcal{P}_{\bm{\epsilon}_Z}|!\prod_{p\in\mathcal{P}_{\bm{\epsilon}_Z}}(\log p))=P_{\bm{\epsilon}_Z}/|\mathcal{P}_{\bm{\epsilon}_Z}|!$ and additionally denoted
\[ c^{\prime +}_{\bm{\epsilon}_Z}=-(\log 2\pi+\gamma-1)+(1/2)\sum\nolimits_{p\in\mathcal{P}_{\bm{\epsilon}_Z}}\log p,\]
noted that $\ell=-\log|t-t_0|$ in the ranges in which the leading terms in \eqref{concluding-asymptotic} are dominant,} and set, for $\varepsilon\in\{\pm\}=\{\pm 1\}$ and $q\in\mathcal{P}_{\bm{\epsilon}_Z}$,
\begin{equation}
\label{sepsilons}
\begin{aligned}
s^{\varepsilon}(a/b,\bm{\epsilon}_Z)&=\frac2{\varphi(b)}\sum_{\substack{\chi\bmod b\\\chi(-1)=\varepsilon\\\chi(p)=\epsilon_p\,(p\in\mathcal{P}_{\bm{\epsilon}_Z})}}\overline{G(\chi)}\chi(-a)\\
s^{\varepsilon}_q(a/b,\bm{\epsilon}_Z)&=\frac2{\varphi(b)}\sum_{\substack{\chi\bmod b\\\chi(-1)=\varepsilon\\\chi(p)=\epsilon_p\,(p\in\mathcal{P}_{\bm{\epsilon}_Z}^{q-})\\\chi(q)\neq\epsilon_q}}(1-\chi(q)\epsilon_q)^{-1}\overline{G(\chi)}\chi(-a).
\end{aligned}
\end{equation}

For future reference we record the following simple lemma. In particular, it confirms that (as it must be) the leading constants in \eqref{concluding-asymptotic} and Lemma~\ref{local-d-lemma} match. To simplify the notation, we write $\bm{\epsilon}_Z^{q-}=\bm{\epsilon}_Z\mathbf{1}_{\{q\}^c}$ for the sequence $(\epsilon_p\delta_{p\neq q})_{p\leqslant Z}$, and we introduce the related exponential sum
\begin{equation}
\label{sq-tilde}
\widetilde{s}_q(a/b;\bm{\epsilon}_Z)=\frac{-\varphi(b)}{(2\varphi(b))^{|\mathcal{P}_{\bm{\epsilon}_Z}|}}\sum_{\substack{(m_p)_{p\in\mathcal{P}_{\bm{\epsilon}_Z}}\\0\leqslant m_p<2\varphi(b)}}\Big\{\frac{m_q+1/2}{\varphi(b)}\Big\}e\Big(\frac ab\prod_{p\in\mathcal{P}_{\bm{\epsilon}_Z}}p^{m_p}\Big)\prod_{p\in\mathcal{P}_{\bm{\epsilon}_Z}}\epsilon_p^{m_p},
\end{equation}
where $\{x\}$ is the familiar sawtooth function defined by $\{x\}=x-\lfloor x\rfloor-1/2$.

\begin{lem}
\label{constants-lemma}
For every $t_0=a/b\in [0,1]\cap\mathbb{Q}$ with $(a,b)=(\prod_{p\in\mathcal{P}_{\bm{\epsilon}_Z}}p,b)=1$, the constants $c^{+}(t_0),c'(t_0),c^{-}(t_0)\in\mathbb{R}$ shown in \eqref{constants} satisfy
\begin{align*}
c^{+}(t_0)&=c_{\bm{\epsilon}_Z}\cdot 2\mRe s^{\ast}(a/b;\bm{\epsilon}_Z),\quad c^{-}(t_0)=-c_{\bm{\epsilon}_Z}|\mathcal{P}_{\bm{\epsilon}_Z}|\cdot\pi\mIm s^{\ast}(a/b;\bm{\epsilon}_Z),\\
c'(t_0)&=c_{\bm{\epsilon}_Z}|\mathcal{P}_{\bm{\epsilon}_Z}|\cdot 2\mRe\Big(c^{\prime +}_{\bm{\epsilon}_Z}s^{\ast}(a/b;\bm{\epsilon}_Z)+\sum_{q\in\mathcal{P}_{\bm{\epsilon}_Z}}(\log q)\widetilde{s}_q(a/b;\bm{\epsilon}_Z)\Big).
\end{align*}
\end{lem}

\begin{proof}
From orthogonality of characters, we have that
\begin{align*}
s^{\varepsilon}(a/b,\bm{\epsilon}_Z)=&\frac2{\varphi(b)}\sum_{\chi\bmod b}\overline{G(\chi)}\chi(-a)\Big(\frac12(1+\chi(-1)\varepsilon)\Big)\prod_{p\in\mathcal{P}_{\bm{\epsilon}_Z}}\Big(\frac1{2\varphi(b)}\sum_{0\leqslant m_p<2\varphi(b)}(\chi(p)\epsilon_p)^{m_p}\Big)\\
=&\frac1{(2\varphi(b))^{|\mathcal{P}_{\bm{\epsilon}_Z}|}}\bigg(\sum_{\substack{(m_p)_{p\in\mathcal{P}_{\bm{\epsilon}_Z}}\\0\leqslant m_p<2\varphi(b)}}e\Big(\frac ab\prod_{p\in\mathcal{P}_{\bm{\epsilon}_Z}}p^{m_p}\Big)\prod_{p\in\mathcal{P}_{\bm{\epsilon}_Z}}\epsilon_p^{m_p}\\
&+\varepsilon\sum_{\substack{(m_p)_{p\in\mathcal{P}_{\bm{\epsilon}_Z}}\\0\leqslant m_p<2\varphi(b)}}e\Big(-\frac ab\prod_{p\in\mathcal{P}_{\bm{\epsilon}_Z}}p^{m_p}\Big)\prod_{p\in\mathcal{P}_{\bm{\epsilon}_Z}}\epsilon_p^{m_p}\bigg)=s^{\ast}(a/b;\bm{\epsilon}_Z)+\varepsilon s^{\ast}(-a/b;\bm{\epsilon}_Z).
\end{align*}
The first two statements of the lemma follows by substituting this expression into \eqref{constants}.

On the other hand, by using the geometric series expansion and l'H\^opital's rule, we can evaluate
\begin{align*}
\tilde{s}_q(a/b,\epsilon_q)&:=\sum_{\substack{\chi\bmod b\\\chi(q)\neq\epsilon_q}}(1-\chi(q)\epsilon_q)^{-1}\overline{G(\chi)}\chi(-a)\\
&=\lim_{r\to 1-}\bigg(\sum_{\chi\bmod b}\frac1{1-r\chi(q)\epsilon_q}\overline{G(\chi)}\chi(-a)-\sum_{\substack{\chi\bmod b\\\chi(q)=\epsilon_q}}\frac1{1-r}\overline{G(\chi)}\chi(-a)\bigg)\\
&=\lim_{r\to 1-}\bigg(\sum_{k=0}^{\infty}r^{k\varphi(b)}\sum_{n=0}^{\varphi(b)-1}\sum_{\chi\bmod b}r^n\chi(q)^n\epsilon_q^n\overline{G(\chi)}\chi(-a)\\
&\qquad\qquad\qquad-\frac1{1-r}\frac1{\varphi(b)}\sum_{\chi\bmod b}\sum_{n=0}^{\varphi(b)-1}\chi(q)^n\epsilon_q^n\overline{G(\chi)}\chi(-a)\bigg)\\
&=\lim_{r\to 1-}\bigg(\varphi(b)\sum_{n=0}^{\varphi(b)-1}\frac{r^n}{1-r^{\varphi(b)}}\epsilon_q^ne\Big(\frac{aq^n}{b}\Big)-\frac1{1-r}\sum_{n=0}^{\varphi(b)-1}\epsilon_q^ne\Big(\frac{aq^n}b\Big)\bigg)\\
&=-\varphi(b)\sum_{n\bmod\varphi(b)}\Big\{\frac{n+1/2}{\varphi(b)}\Big\}\epsilon_q^ne\Big(\frac{aq^n}b\Big).
\end{align*}
Using this evaluation, we can then argue as above that
\begin{align*}
s^{\varepsilon}_q(a/b,\bm{\epsilon}_Z)=&\frac2{\varphi(b)}\frac1{2(2\varphi(b))^{|\mathcal{P}_{\bm{\epsilon}_Z}|-1}}\sum_{\substack{(m_p)_{p\in\mathcal{P}_{\bm{\epsilon}_Z}^{q-}}\\0\leqslant m_p<2\varphi(b)}}\Big(\tilde{s}_q\Big(\frac{a}{b}\prod_{p\in\mathcal{P}_{\bm{\epsilon}_Z}^{q-}}p^{m_p},\epsilon_q\Big)\prod_{p\in\mathcal{P}_{\bm{\epsilon}_Z}^{q-}}\epsilon_p^{m_p}\\
&+\varepsilon\tilde{s}_q\Big(-\frac{a}{b}\prod_{p\in\mathcal{P}_{\bm{\epsilon}_Z}^{q-}}p^{m_p},\epsilon_q\Big)\prod_{p\in\mathcal{P}_{\bm{\epsilon}_Z}^{q-}}\epsilon_p^{m_p}\Big)=\widetilde{s}_q(a/b;\bm{\epsilon}_Z)+\varepsilon\widetilde{s}_q(-a/b;\bm{\epsilon}_Z),
\end{align*}
and the third statement again follows by invoking \eqref{constants}.
\end{proof}

Putting everything together completes the proof of the following proposition, the crowning achievement of this subsection.

\begin{prop}
\label{two-terms-prop}
There exists a $\delta>0$, depending only on $Z$, such that, for every $t_0=a/b\in[0,1]\cap\mathbb{Q}$ with $(a,b)=1$ and $(\prod_{p\in\mathcal{P}_{\bm{\epsilon}_Z}}p,b)=1$, the function $G^{\sharp}_{\bm{\epsilon}_Z}(t)$ satisfies
\begin{align*}
G^{\sharp}_{\bm{\epsilon}_Z}(t)-G^{\sharp}_{\bm{\epsilon}_Z}(t_0)&=e(t_0)c^{+}(t_0)\cdot(t-t_0)\ell^{|\mathcal{P}_{\bm{\epsilon}_Z}|}\\
&\qquad +e(t_0)\big(c'(t_0)\cdot(t-t_0)+c^{-}(t_0)\cdot |t-t_0|\big)\ell^{|\mathcal{P}_{\bm{\epsilon}_Z}|-1}\\
&\qquad+\OO_{b,Z}\big(|t-t_0|\ell^{|\mathcal{P}_{\bm{\epsilon}_Z}|-1-\delta}\big),
\end{align*}
for certain constants $c^{+}(t_0),c'(t_0),c^{-}(t_0)\in\mathbb{R}$ shown in \eqref{constants} and Lemma~\ref{constants-lemma} and $\ell=|\log|t-t_0||$.
\end{prop}

\subsection{Nonvanishing and multiplicativity of exponential sums}
\label{expsums-nv-mult}

The exponential sums $s^{\ast}(a/b;\bm{\epsilon}_Z)$ are related (though not always in a straightforward fashion) to the generalized Gauss power sums, defined for any rational number $a/q$ with $(a,q)=1$, $d\mid\varphi(q)$, and $\iota\in\{0,1\}$ as 
\[ \sigma_d^{\iota}(a/q)=\sumast_{x\bmod q}\Big(\frac xq\Big)^{\iota}e\Big(\frac{ax^d}q\Big). \]
We summarize this relationship in the following lemma. To succinctly state the multiplicativity property of the sums $s^{\ast}(a/m;\bm{\epsilon}_Z)$, we introduce for $\delta\mid\varphi(m)$ and $m\mid m'$ slightly more general sums
\[ s(a/m;\bm{\epsilon}_Z)[\delta,m']=\sum_{\substack{(m_p)\,(p\leqslant Z,\,\epsilon_p\neq 0,\,p\nmid m')\\0\leqslant m_p<2\varphi(m)}}e\Big(\frac am\prod_{\substack{p\leqslant Z,\\\epsilon_p\neq 0,\,p\nmid m'}}p^{\delta m_p}\Big)\prod_{\substack{p\leqslant Z\\\epsilon_p\neq 0,\,p\nmid m'}}\epsilon_p^{m_p}, \]
so that $s^{\ast}(a/m,\bm{\epsilon}_Z)=s(a/m,\bm{\epsilon}_Z)[1,m]$.

\begin{lem}
\label{big-lemma-sums}
\begin{enumerate}
\item\label{gp-sum} For $q=p^k$ for an odd prime $p$ and $k\geqslant 1$,
\[ \sigma_d^{\iota}(a/q)\neq 0\quad\text{if and only if}\quad p^{k-1}\mid d\mid\varphi(p^k)/(1+\iota'), \]
where $\iota'=1$ if $\iota=1$ and $2\nmid k$, and $\iota'=0$ otherwise. In particular, for $q=p$:
\begin{itemize}
\item $\sigma_d^{\iota}(a/p)\neq 0$ for $d\mid (p-1)/(1+\iota)$,
\item $\sigma_d^{\iota}(a/p)\in\mathbb{R}$ whenever $-1\in(\mathbb{Z}/p\mathbb{Z})^{\times d(1+\iota)}$,
\item $\sigma_d^1(a/p)\in i\mathbb{R}$ whenever $2\nmid d$ and $p\equiv 3\bmod 4$,
\item for $d=1$, $\sigma_1^0(a/p)=-1$ and $\sigma_1^1(a/p)=(a/p)\sqrt{p}G(p)$.
\end{itemize}

\item\label{item2} If $2\nmid\ord_mp$ for some $p\leqslant Z$ with $p\nmid m$ and $\epsilon_p=-1$, then $s^{\ast}(a/m;\bm{\epsilon}_Z)=0$.  Further, if $m=q=\ttp^k$ for an odd prime $\ttp$ and $k\geqslant 2$, $s^{\ast}(a/q;\bm{\epsilon}_Z)=0$ unless $\ord_qp\mid(\ttp-1)$ for all $p\leqslant Z$ with $\epsilon_p\neq 0$.
 
 \item\label{item-mult}
For $m=m_1m_2$ with $(m_1,m_2)=1$ and $\delta=(\varphi(m_1),\varphi(m_2))$,
\[ s(a/m;\bm{\epsilon}_Z)[\delta',m']=\frac1{4\delta}\sum_{\substack{(\mu_p)\,(p\leqslant Z,\\\epsilon_p\neq 0,\,p\nmid m')\\0\leqslant\mu_p<2\delta}}\bm{\epsilon}_Z^{\bm{\mu}}s\Big(\frac{a\overline{m_2}\bm{p}_Z^{\delta'\bm{\mu}}}{m_1};|\bm{\epsilon}_Z|\Big)[2\delta\delta',m']s\Big(\frac{a\overline{m_1}\bm{p}_Z^{\delta'\bm{\mu}}}{m_2};|\bm{\epsilon}_Z|\Big)[2\delta\delta',m']. \]
where $m_i\overline{m_i}\equiv 1\pmod{m_{3-i}}$ and
\[ \bm{\epsilon}_Z^{\bm{\mu}}=\prod_{\substack{p\leqslant Z\\\epsilon_p\neq 0,\,p\nmid m'}}\epsilon_p^{\mu_p},\quad \bm{p}_Z^{\delta'\bm{\mu}}=\prod_{\substack{p\leqslant Z\\\epsilon_p\neq 0,\,p\nmid m'}}p^{\delta'\mu_p}. \]

\item\label{item3} For $q$ an odd prime power, let (noting that then $\mathcal{P}_{\bm{\epsilon}_Z}\subseteq (\mathbb{Z}/q\mathbb{Z})^{\times d_q}$)
\[ d_q=\frac{\varphi(q)}{\lcm\big[(\ord_qp)_{p\in\mathcal{P}_{\bm{\epsilon}_Z}}\big]},\quad \mathcal{P}_{\bm{\epsilon}_Z}^{+}(q)=\mathcal{P}_{\bm{\epsilon}_Z}\cap (\mathbb{Z}/q\mathbb{Z})^{\times 2d_q},\,\, \mathcal{P}_{\bm{\epsilon}_Z}^{-}(q)=\mathcal{P}_{\bm{\epsilon}_Z}\setminus\mathcal{P}_{\bm{\epsilon}_Z}^{+}(q). \]
Then,
\[ s^{\ast}(a/q,\bm{\epsilon}_Z)=\begin{cases} (1/\varphi(q))\sigma_{d_q}^0(a/q),&\mathcal{P}_{\bm{\epsilon}_Z}^{-}(q)=\emptyset,\,\bm{\epsilon}=\{1\},\\ 
(1/\varphi(q))\sigma_{d_q}^{\iota}(a/q),&\bm{\epsilon}|_{\mathcal{P}_{\bm{\epsilon}_Z}^{-}(q)}=\{(-1)^{\iota}\},\,\bm{\epsilon}|_{\mathcal{P}_{\bm{\epsilon}_Z}^{+}(q)}=\{1\},\\
0,&\text{otherwise}.\end{cases} \]

\end{enumerate}
\end{lem}

\begin{proof}
Item \eqref{gp-sum} is essentially elementary (and probably well known). For $k=1$, we may write
\[ \sigma_d^{\iota}(a/p)=\sumast_{x\bmod p}(x/p)^{\iota}e(1/p)^{ax^d}. \]
If $\iota=1$, then changing variables $x\mapsto xg^{(p-1)/d}$ (for $g$ an arbitrary primitive root modulo $p$) shows that $\sigma_d^1(a/p)=0$ unless $2\mid(p-1)/d$, that is, $d\mid (p-1)/2$.

Now, the root of unity $\zeta_p=e(1/p)$ generates the cyclotomic field $\mathbb{Q}(\zeta_p)$, in which $\varpi_p=1-\zeta_p$ is a prime of absolute norm $p\sim\varpi_p^{p-1}$. From $\zeta_p\equiv 1\pmod{\varpi_p}$ we conclude that
\[ \sigma_d^0(a/p)\equiv p-1\equiv -1\pmod{\varpi_p}, \]
and thus in particular $\sigma_d^0(a/p)\neq 0$. For $\iota=1$ and $d\mid (p-1)/2$, we have by a slightly more involved argument that
\begin{align*}
\sigma_d^0(a/p)=\sumast_{x\bmod p}(1-\varpi_p)^{ax^d}
&\equiv \sumast_{x\bmod p}\sum_{0\leqslant i\leqslant (p-1)/d}\binom{ax^d}i(-\varpi_p)^i\\
&\equiv (p-1)\Big(1+\frac{(-a\varpi_p)^{(p-1)/d}}{((p-1)/d)!}\Big)\pmod{\varpi_p^{(p-1)/d+1}},
\end{align*}
from which we conclude that
\[ \sigma_d^1(a/p)=-\sigma_d^0(a/p)+\sigma_{2d}^0(a/p)\neq 0\pmod{\varpi_p^{(p-1)/(2d)+1}}, \]
so in particular $\sigma_d^1(a/p)\neq 0$. As for the realness claim, if $\iota=1$ and $-1\equiv y^{2d}\bmod p$, we see by making a change of variable $x\mapsto xy^2$ that
\[ \sigma_1^1(a/p)=\sumast_{x\bmod p}\Big(\frac xp\Big)e\Big(\frac{ax^dy^{2d}}p\Big)=\overline{\sigma_1^1(a/p)}. \]
The case of $\iota=0$ and $-1\in(\mathbb{Z}/p\mathbb{Z})^{\times d}$ is similar (even easier). That $\sigma_d^1(a/p)=-\overline{\sigma_d^1(a/p)}$ when $2\nmid d$ and $p\equiv 3\bmod 4$ follows by a change of variable $x\mapsto -x$. Finally, the claims identifying $\sigma_1^{\iota}(a/p)$ as the Ramanujan and Gauss sums modulo $p$ are immediate.

For $k\geqslant 2$, we may write $d=p^{\kappa}\delta$ for some $0\leqslant\kappa\leqslant k-1$ and $\delta\mid(p-1)$. Denoting by $g$ a (fixed but otherwise arbitrary) primitive root modulo $p^k$, $g^{p^{\kappa}}$ is a primitive root modulo $p^{k-\kappa}$, from which it is easy to see that
\[ \sigma_d^{\iota}(a/q)=\sum_{0\leqslant i<\varphi(p^k)}\Big(\frac{g^{p^{\kappa}\delta i}}{p^k}\Big)e\Big(\frac{ag^{p^{\kappa}\delta i}}{p^k}\Big)=p^{\kappa}\sumast_{x\bmod p^{k-\kappa}}\Big(\frac{x}{p}\Big)^{\iota'}e\Big(\frac{ax^{\delta}}{p^{k-\kappa}}\Big). \]
When $\kappa=k-1$, we find that $\sigma_d^{\iota}(a/q)$ equals $p^{k-1}\sigma_{\delta}^{\iota'}(a/p)$ and thus doesn't vanish by what we already proved. On other hand, for $\kappa\leqslant k-2$, it follows from the $p$-adic method of stationary phase (see, for example, \cite[Lemma 1]{MilicevicZhang2023}) that the above sum vanishes, since no summands satisfy the stationary phase condition $\delta\cdot ax^{\delta-1}\equiv 0\pmod{p^{\lfloor (k-\kappa)/2\rfloor}}$.

We proceed to item \eqref{item2}. If $2\nmid\ord_mp$ for some $p\leqslant Z$ with $p\nmid m$ and $\epsilon_p=-1$, then we see by shifting variables $m_p\mapsto m_p+\ord_mp$ in \eqref{sast-def} that $s^{\ast}(a/m;\bm{\epsilon}_Z)=0$. Now, let $m=\ttp^k$ for some $k\geqslant 2$, and write $\delta=\ord_{\ttp}p$, $p^{\delta}=1+\ttp^ef$ for some $p\leqslant Z$ with $\epsilon_p\neq 0$, $p\neq\ttp$, $e\geqslant 1$, $\ttp\nmid f$. If $2\nmid\delta$ and $\epsilon_p=-1$, then, noting that $p^{\delta\ttp^{k-e}}\equiv 1\pmod{\ttp^k}$ and changing variables $m_p\mapsto m_p+\delta\ttp^{k-e}$ in \eqref{sast-def}, we again see that $s^{\ast}(a/m;\bm{\epsilon}_Z)=0$; thus, from now on we may assume that $2\mid\delta$ or $\epsilon_p=1$.
If $e\leqslant k-2$, then for $1\leqslant\kappa\leqslant k-e$ we have that
\begin{gather*}
p^{m_p+\delta \ttp^{\kappa}t}
\equiv p^{m_p}\exp_{\ttp}\big(\ttp^{\kappa}t\log_{\ttp}(1+\ttp^ef)\big)
\equiv p^{m_p}+f_1(m_p)\cdot\delta\ttp^{\kappa}t\pmod{\ttp^{2\kappa+2e}},\\
f_1(m_p)=\ttp^e\cdot p^{m_p}\delta^{-1}\frac{\log_{\ttp}(1+\ttp^ef)}{\ttp^e},\quad \ttp^e\exmid f_1(m_p).
\end{gather*}
Using the above with any $\kappa\geqslant k/2-e$ and applying the $\ttp$-adic method of stationary phase (see \cite[Lemma 1]{MilicevicZhang2023} and note that $\epsilon_p^{m_p+\delta\ttp^{\kappa}t}=\epsilon_p^{m_p}$), we conclude that the summation in \eqref{sast-def} may be restricted to $m_p$ satisfying $\ttp^{k-\kappa}\mid f_1(m_p)$; thus, picking any $\max(1,k/2-e)\leqslant\kappa\leqslant k-e-1$ we conclude that $s^{\ast}(a/m;\bm{\epsilon}_Z)=0$. If $e=k-1$, only a minor tweak is needed:
\[ p^{m_p+(\ttp-1)t}=p^{m_p}(1+\ttp^{k-1}f)^{((\ttp-1)/\delta)t}\equiv p^{m_p}+p^{m_p+k-1}((\ttp-1)/\delta)f\cdot t\bmod{\ttp^k}, \]
whence by shifting $m_p\mapsto m_p+(\ttp-1)t$ in \eqref{sast-def} we analogously find that
\[ s^{\ast}(a/m;\bm{\epsilon}_Z)=\frac1{\ttp}\sum_{\substack{(m_p)^{\ast}_{p\leqslant Z}\\0\leqslant m_p<2\varphi(\ttp^k)}}e\Big(\frac a{\ttp^k}\prodast_{p\leqslant Z}p^{m_p}\Big)\prodast_{p\leqslant Z}\epsilon_p^{m_p}\sum_{t\bmod\ttp}e\Big(\frac a{\ttp}\prodast_{p\leqslant Z}p^{m_p}\cdot\frac{\ttp-1}{\delta}f\cdot t\Big)=0. \]
In conclusion, $s^{\ast}(a/\ttp^k;\bm{\epsilon}_Z)=0$ unless $p^{\delta}\equiv 1\pmod{\ttp^k}$, which is to say that $\ord_qp\mid (\ttp-1)$, and this condition must hold for every $p\leqslant Z$ with $\epsilon_p\neq 0$.

Item \eqref{item-mult} is a direct consequence of the Chinese Remainder Theorem. Indeed, the value of each summand in \eqref{sast-def} depends only on $m_p$ modulo $2\varphi(m_1)\varphi(m_2)/\delta$. Writing every $0\leqslant m_p<2\varphi(m_1)\varphi(m_2)/\delta$ as
\[ m_p=\mu_p+2\delta\big(\,\overline{\varphi(m_2)/\delta}(\varphi(m_2)/\delta)k_{p1}+\overline{\varphi(m_1)/\delta}(\varphi(m_1)/\delta)k_{p2}\,\big) \]
where $0\leqslant\mu_p<2\delta$ and $0\leqslant k_{pi}<\varphi(m_i)/\delta$ and inverses are modulo $\varphi(m_i)/\delta$,
we have that
\[ \prod_{\substack{p\leqslant Z,\\\epsilon_p\neq 0,\,p\nmid m'}}p^{\delta' m_p}\equiv \prod_{\substack{p\leqslant Z,\\\epsilon_p\neq 0,\,p\nmid m'}}p^{\delta' \mu_p}\cdot \prod_{\substack{p\leqslant Z,\\\epsilon_p\neq 0,\,p\nmid m'}} p^{\delta\delta' u_ik_{pi}}\bmod{m_i}, \]
where $u_i=\overline{\varphi(m_{3-i})/\delta}(\varphi(m_{3-i})/\delta)$ are independent of $k_{pi}$ and satisfy $(u_i,\varphi(m_i)/\delta)=1$. Thus,
\[ a\bm{p}_Z^{\delta'\bm{m}}\equiv am_2\overline{m_2}\bm{p}_Z^{\delta'\bm{\mu}}\bm{p}_Z^{2\delta\delta'u_1\bm{k}_1}+am_1\overline{m_1}\bm{p}_Z^{\delta'\bm{\mu}}\bm{p}_Z^{2\delta\delta'u_2\bm{k}_2}\pmod{m_1m_2}; \]
the claim follows immediately from this upon summing $e(a\bm{p}_Z^{\delta'\bm{m}}/m)\bm{\epsilon}_Z^{\bm{m}}$ over $0\leqslant\mu_p<2\delta$ and $0\leqslant k_{pi}<\varphi(m_i)/\delta$, which encounters $(1/\delta)s(a/m;\bm{\epsilon}_Z)[\delta',m']$ and the product of $\bm{\epsilon}_Z^{\bm{\mu}}$ with two sums $(1/2\delta)s(a\overline{m_{3-i}}\bm{p}_Z^{\delta'\bm{\mu}}/m_i;|\bm{\epsilon}_Z|)[2\delta\delta',m']$.

Finally, we turn our attention to item \eqref{item3}. Fixing an arbitrary primitive root $g$ modulo $q$ and writing $p=g^{k_p}$, $\epsilon_p=(-1)^{\varepsilon_p}$, we have that $d_q=\gcd[(k_p,\varphi(q))_{p\in\mathcal{P}_{\bm{\epsilon}_Z}}]$, and, by definition,
$(2\varphi(q))^{|\mathcal{P}_{\bm{\epsilon}_Z}|}s^{\ast}(a/q;\bm{\epsilon}_Z)$ equals
\begin{align*}
&\sum_{0\leqslant k<\varphi(q)}e(ag^k/q)\sum_{\substack{(m_p)^{\ast}_{p\leqslant Z}:0\leqslant m_p<2\varphi(q)\\\sum^{\ast}_{p\leqslant Z}k_pm_p\equiv k\bmod{\varphi(q)}}}(-1)^{\sum^{\ast}_{p\leqslant Z}m_p\varepsilon_p}\\
&\qquad=\Bigg(\sum_{\substack{0\leqslant k<\varphi(q)\\ d_q\mid k}}
(-1)^{\sum^{\ast}_{p\leqslant Z}m_p^{\circ}(k)\varepsilon_p}e(ag^k/q)\Bigg)
\Bigg(\sum_{\substack{(m_p)^{\ast}_{p\leqslant Z}:0\leqslant m_p<2\varphi(q)\\\sum^{\ast}_{p\leqslant Z}k_pm_p\equiv 0\bmod\varphi(q)}}(-1)^{\sum^{\ast}_{p\leqslant Z}m_p\varepsilon_p}\Bigg),
\end{align*}
where, for every $0\leqslant k<\varphi(q)$ with $d_q\mid k$, we denote by $(m_p^{\circ}(k))$ an arbitrary particular solution of the congruence $\sum^{\ast}_{p\leqslant Z}k_pm_p^{\circ}(k)\equiv k\bmod\varphi(q)$. The indexing set in the latter sum forms an additive group modulo $2\varphi(q)$ (a disjoint union of $2^{|\mathcal{P}_{\bm{\epsilon}_Z}|}$ additive groups modulo $\varphi(q)$) of combined order $M:=(2\varphi(q))^{|\mathcal{P}_{\bm{\epsilon}_Z}|-1}\cdot 2d_q$, and the sum equals $M$ or $0$ according to whether the implication
\[ \textstyle\sum^{\ast}_{p\leqslant Z}k_pm_p\equiv 0\bmod{\varphi(q)}\quad\Rightarrow\quad \sum^{\ast}_{p\leqslant Z}m_p\varepsilon_p\equiv 0\bmod 2 \]
holds or not. Let $2^{\varphi_q}\exmid\varphi(q)$; then, by adjusting the values of $m_p$ modulo $\varphi(q)/2^{\varphi_q}$ using the Chinese Remainder Theorem, we see that the above implication holds if and only if we have a valid implication
\[ \textstyle\sum^{\ast}_{p\leqslant Z}k_pm_p\equiv 0\bmod{2^{\varphi_q}}\quad\Rightarrow\quad \sum^{\ast}_{p\leqslant Z}m_p\varepsilon_p\equiv 0\bmod 2. \]
Further, denote $2^{\delta_q}\exmid d_q$, so that $\mathcal{P}_{\bm{\epsilon}_Z}^{-}=\{p\leqslant Z:2^{\delta_q}\exmid k_p\}\neq\emptyset$. If $\delta_q=\varphi_q$, the above can clearly hold only if all $\varepsilon_p=0$; otherwise, by dividing through the first congruence by $2^{\delta_q}$ and
adjusting the values of $(m_p)_{p\in\mathcal{P}_{\bm{\epsilon}_Z}^{-}}$ by even amounts, the above implication holds if and only if
\[\textstyle\sum_{p\in\mathcal{P}_{\bm{\epsilon}_Z}^{-}}m_p\equiv 0\bmod 2\quad\Rightarrow\quad \sum_{p\in\mathcal{P}_{\bm{\epsilon}_Z}^{-}}m_p\varepsilon_p+\sum_{p\in\mathcal{P}_{\bm{\epsilon}_Z}\setminus\mathcal{P}_{\bm{\epsilon}_Z}^{-}}m_p\varepsilon_p\equiv 0\bmod 2. \]
The latter plainly holds if and only if there exists a $\varepsilon\in\{0,1\}$ such that $\varepsilon_p=\varepsilon$ for all $p\in\mathcal{P}_{\bm{\epsilon}_Z}^{-}$ and $\varepsilon_p=0$ for all $p\in\mathcal{P}_{\bm{\epsilon}_Z}\setminus\mathcal{P}_{\bm{\epsilon}_Z}^{-}$, and in this case
\[ (2\varphi(q))^{|\mathcal{P}_{\bm{\epsilon}_Z}|}s^{\ast}(a/q;\bm{\epsilon}_Z)=\sum_{\substack{0\leqslant k<\varphi(q)\\d_q\mid k}}(-1)^{\varepsilon(k/2^{\delta_q})}e\Big(\frac{ag^k}{q}\Big)=\frac{M}{d_q}\sumast_{x\bmod q}\Big(\frac xq\Big)^{\varepsilon}e\Big(\frac{ax^{d_q}}q\Big). \qedhere\]
\end{proof}

\subsection{Locating a dense set of singularities}
\label{main-cor-subsec}

In this subsection, we use our results from \S\S\ref{limits-local}--\ref{expsums-nv-mult} to prove the following culminating proposition of Section~\ref{sec-atlas}, which provides a collection of points $t_0\in[0,1]\cap\mathbb{Q}$ at which the path $G^{\sharp}_{\bm{\epsilon}_Z}(t)$ has a cusp.

\begin{prop}
\label{main-cor}
Assume that $\bm{\epsilon}_Z$ is not identically zero or one.
Let $q>Z$ be an odd prime such that $q\equiv 3\bmod 4$ and
\begin{equation}
\label{ArP}
\Big(\frac pq\Big)=\epsilon_p\quad (p\leqslant Z,\,\epsilon_p\neq 0).
\end{equation}
Then, for every $1\leqslant a\leqslant q-1$,
\[ \mRe s^{\ast}(a/q;\bm{\epsilon}_Z)=0,\quad \mIm s^{\ast}(a/q;\bm{\epsilon}_Z)\neq 0, \]
and, for $t_0=a/q$, the path $G^{\sharp}_{\bm{\epsilon}_Z}(t)$ satisfies
\begin{equation}
\label{local-asymp-cor}
\frac{G^{\sharp}_{\bm{\epsilon}_Z}(t)-G^{\sharp}_{\bm{\epsilon}_Z}(t_0)}{t-t_0}\sim c_{\pm}(t_0)\cdot e(t_0)|\log(t-t_0)|^{|\mathcal{P}_{\bm{\epsilon}_Z}|\-1}\quad (t\to t_0\pm)
\end{equation}
where $c_{\pm}(t_0)\in\mathbb{R}$ are given explicitly in \eqref{cpm}.

The curve $G^{\sharp}_{\bm{\epsilon}_Z}(t)$ has a cusp at $t=t_0$ as long as $c_{+}(t_0)c_{-}(t_0)<0$. Specifically:
\begin{enumerate}
\item\label{cor-item1} If $\epsilon_p=-1$ for at least two $p\leqslant Z$, then \eqref{local-asymp-cor} holds with
\[ c_{\pm}(t_0)=\pm c^{-}(t_0) \]
with $c^{-}(t_0)=1/(|\mathcal{P}_{\bm{\epsilon}_Z}|-1)!\prod_{p\in\mathcal{P}_{\bm{\epsilon}_Z}}(\log p))\cdot\mIm s^{\ast}(a/q;\bm{\epsilon}_Z)\in\mathbb{R}_{\neq 0}$,
and the path $G^{\sharp}_{\bm{\epsilon}_Z}(t)$ has a cusp at $t=t_0$. In particular, the set of such $t_0=a/q\in [0,1]\cap\mathbb{Q}$ (over different values of $q$) is everywhere dense in $[0,1]$.
\item\label{cor-item2} If $\epsilon_p=-1$ for exactly one $p=p_1\leqslant Z$ and the residues of all $p\in\mathcal{P}_{\bm{\epsilon}_Z}\setminus\{p_1\}$ modulo $q$ generate all quadratic residues modulo $q$:
\begin{equation}
\label{all-QRs}
\big\langle \{p\bmod q:p\leqslant Z,\,\epsilon_p=1\}\big\rangle=(\mathbb{Z}/q\mathbb{Z})^{\times 2},
\end{equation}
then \eqref{local-asymp-cor} holds with
\[ c_{\pm}(t_0)=\pm(1\pm\delta_{p_1,t_0})c^{-}(t_0), \]
where $c^{-}(t_0)\in\mathbb{R}_{\neq 0}$ is as above, $\delta_{p_1,t_0}=(2/\pi)(a/p_1)\log p_1/\sqrt{p_1}$ satisfies $|\delta_{p_1,t_0}|<1/2$, and the path $G^{\sharp}_{\bm{\epsilon}_Z}(t)$ has a cusp at $t=t_0$.
\end{enumerate}
\end{prop}

\begin{proof}
For any $p\in\mathcal{P}_{\bm{\epsilon}_Z}$ such that $\epsilon_p=-1$, the congruence conditions on $q$ imply (fixing an arbitrary primitive root $g$ modulo $q$) that $p\equiv g^{k_p}\bmod q$ for some $2\nmid k_p$, whence $\varphi(q)/\ord_qp$ is odd and \emph{a fortiori} $d_q$ is odd as well. In the notation of Lemma~\ref{big-lemma-sums}, this, in turn, implies that, for $p\in\mathcal{P}_{\bm{\epsilon}_Z}$, $p\in\mathcal{P}_{\bm{\epsilon}_Z}^{+}(q)$ if and only if $\epsilon_p=1$, whence according to Lemma~\ref{big-lemma-sums}, items \eqref{gp-sum} and \eqref{item3},
\[ s^{\ast}(a/q,\bm{\epsilon}_Z)=(1/\varphi(q))\sigma_{d_q}^1(a/q)\in i\mathbb{R}_{\neq 0}. \]

From Proposition~\ref{two-terms-prop} and Lemma~\ref{constants-lemma}, we have that
\begin{align*}
G^{\sharp}_{\bm{\epsilon}_Z}(t)-G^{\sharp}_{\bm{\epsilon}_Z}(t_0)
&=e(t_0)\big(c'(t_0)\!\cdot\!(t-t_0)+c^{-}(t_0)\!\cdot\!|t-t_0|\big)\ell^{|\mathcal{P}_{\bm{\epsilon}_Z}|-1}+\OO_{b,Z}\big(|t-t_0|\ell^{|\mathcal{P}_{\bm{\epsilon}_Z}-1-\delta}\big)\\
&=e(t_0)c_{\pm}(t_0)(t-t_0)+\OO_{b,Z}\big(|t-t_0|\ell^{|\mathcal{P}_{\bm{\epsilon}_Z}-1-\delta}\big)\quad (t\to t_0\pm),
\end{align*}
where $\ell=|\log|t-t_0||$ and
\begin{equation}\label{cpm}\begin{aligned}
c_{\pm}(t_0)&=c'(t_0)\pm c^{-}(t_0),\\
c^{-}(t_0)&=-c_{\bm{\epsilon}_Z}|\mathcal{P}_{\bm{\epsilon}_Z}|\cdot\pi\mIm s^{\ast}(a/q;\bm{\epsilon}_Z),\\
c'(t_0)&=-c_{\bm{\epsilon}_Z}|\mathcal{P}_{\bm{\epsilon}_Z}|\cdot 2\mRe\sum_{p_1\in\mathcal{P}_{\bm{\epsilon}_Z}}(\log p_1)\widetilde{s}_{p_1}(a/q;\bm{\epsilon}_Z).
\end{aligned}
\end{equation} This completes the proof of \eqref{local-asymp-cor}.

Recall the condition that $\mathcal{P}^{-}_{\bm{\epsilon}_Z}(q)=\mathcal{P}_{\bm{\epsilon}_Z}^{-}=\{p\in\mathcal{P}_{\bm{\epsilon}_Z}:\epsilon_p=-1\}\neq\emptyset$. Now, if $|\mathcal{P}^{-}_{\bm{\epsilon}_Z}|\geqslant 2$, then, for every $p_1\in\mathcal{P}_{\bm{\epsilon}_Z}$, there exists a $p\in\mathcal{P}_{\bm{\epsilon}_Z}^{-}\setminus\{p_1\}$, and by shifting variables in \eqref{sq-tilde} by $m_p\mapsto m_p+\varphi(q)/2$ and recalling that $p^{\varphi(q)/2}\equiv -1\bmod q$, we see that $\widetilde{s}_{p_1}(a/q;\bm{\epsilon}_Z)\in i\mathbb{R}$; since this conclusion holds for every $p_1\in\mathcal{P}_{\bm{\epsilon}_Z}$, we conclude that $c'(t_0)=0$ and \eqref{local-asymp-cor} follows with $c_{+}(t_0)=c_{-}(t_0)=c^{-}(t_0)$. Moreover, it follows from quadratic reciprocity and Dirichlet's theorem on primes in arithmetic progressions that there are infinitely many primes such that $q\equiv 3\bmod 4$ and $(q/p)=\epsilon_p$ for all $p\leqslant Z$ with $\epsilon_p\neq 0$, whence the set of the corresponding fractions $t_0=a/q$ is dense in $[0,1]$. This settles the case \eqref{cor-item1}.

If $\mathcal{P}_{\bm{\epsilon}_Z}^{-}=\{p_1\}$, the situation is more complicated: the same change of variables $m_p\mapsto m_p+\varphi(q)/2$ in \eqref{sq-tilde} still shows that $\widetilde{s}_p(a/q;\bm{\epsilon}_Z)\in i\mathbb{R}$ for all $p\in\mathcal{P}_{\bm{\epsilon}_Z}\setminus\{p_1\}$, while
\begin{equation}
\label{2mresp1}
2\mRe\widetilde{s}_{p_1}(a/q;\bm{\epsilon}_Z)=\frac{\varphi(q)/2}{\varphi(q)^{|\mathcal{P}_{\bm{\epsilon}_Z}|}}\sum_{\substack{(m_p)_{p\in\mathcal{P}_{\bm{\epsilon}_Z}}\\0\leqslant m_p<\varphi(q)}}e\Big(\frac aq\prod_{p\in\mathcal{P}_{\bm{\epsilon}_Z}}p^{m_p}\Big)\epsilon_{p_1}^{m_{p_1}}(-1)^{\lfloor m_{p_1}/(\varphi(q)/2)\rfloor}
\end{equation}
and
\begin{equation}
\label{cprime}
c'(t_0)=-c_{\bm{\epsilon}_Z}(t_0)=-c_{\bm{\epsilon}_Z}|\mathcal{P}_{\bm{\epsilon}_Z}|(\log p_1)\cdot 2\mRe\widetilde{s}_{p_1}(a/q;\bm{\epsilon}_Z).
\end{equation}
One final simplification is possible, as follows. The group generated by the residues of all $p\in\mathcal{P}_{\bm{\epsilon}_Z}^{+}$ modulo $q$ is of the form $(\mathbb{Z}/q\mathbb{Z})^{\times 2d_1}$, where $2d_1=\varphi(q)/\lcm(\ord_qp:p\in\mathcal{P}_{\bm{\epsilon}_Z}^{+})$. If we denote by $2\Delta>0$ the smallest positive exponent such that $p_1^{2\Delta}\in(\mathbb{Z}/q\mathbb{Z})^{\times 2d_1}$, say $p_1^{2\Delta}=\prod_{p\in\mathcal{P}_{\bm{\epsilon}_Z}}p^{m_p^1}$, then $\Delta\mid(\varphi(q)/2)$ and a change of variables
\[ m_p\mapsto m_p-m_p^1\,\,(p\in\mathcal{P}_{\bm{\epsilon}_Z}^{+}),\quad m_{p_1}\mapsto m_{p_1}+(\varphi(q)/2+\Delta) \]
shows that the summation in \eqref{2mresp1} may be restricted to $\varphi(q)/2-\Delta\leqslant m_{p_1}<\varphi(q)/2+\Delta$, since the contributions of the terms outside this range cancel out. The same argument as in the proof of Lemma~\ref{big-lemma-sums}\eqref{item3} then shows that
\[ 2\mRe\widetilde{s}_{p_1}(a/q;\bm{\epsilon}_Z)=\frac{\Delta}{\varphi(q)}\sumast_{x\bmod q}\iota_{p_1}(x^d)\Big(\frac xq\Big)e\Big(\frac{ax^d}q\Big), \]
where $\iota_{p_1}(x^d)=\iota_{p_1}(p_1^{\delta}y^{2d_1})=(-1)^{\lfloor\delta/\Delta\rfloor}$ for the well-defined value of $\delta\pmod{2\Delta}$ in a decomposition $x^d\equiv p_1^{\delta}y^{2d_1}\pmod b$.

All of the above applies whenever $\mathcal{P}^{-}_{\bm{\epsilon}_Z}=\{p_1\}$. To reach the conclusion that, for some $t_0=a/q$ with $(a/q)=1$, $c_{+}(t_0)c_{-}(t_0)<0$ holds (whence the path $G^{\sharp}_{\bm{\epsilon}_Z}(t)$ would have a cusp at $t=t_0$), it is thus necessary and sufficient to verify that the constants $c'(t_0)$ and $c^{-}(t_0)$ as explicated in \eqref{cpm} and \eqref{cprime} satisfy
\begin{equation}
\label{equivalent-condition}
\begin{gathered}
|c'(t_0)|<|c^{-}(t_0)|,\quad\text{that is,}\\
(\log p_1)2|\mRe \widetilde{s}_{p_1}(a/q;\bm{\epsilon}_Z)|<\pi|\mIm s^{\ast}(a/q;\bm{\epsilon}_Z)|.
\end{gathered}
\end{equation}
We do not know how to verify or fully characterize the set of points $t_0$ where this fascinating condition is satisfied. However, under the condition \eqref{all-QRs}, $d_1=\Delta=d=1$, and the sums $\mIm s^{\ast}(a/q;\bm{\epsilon}_Z)$ and $\mRe\widetilde{s}_{p_1}(a/q;\bm{\epsilon}_Z)$ are essentially the quadratic Gauss and Ramanujan sums:
\begin{align*}
\mIm s^{\ast}(a/q;\bm{\epsilon}_Z)&=\frac1{i\varphi(q)}\sigma_1^1(a/q)=\Big(\frac aq\Big)\frac{\sqrt{q}}{\varphi(q)},\\
2\mRe\widetilde{s}_{p_1}(a/q;\bm{\epsilon}_Z)&=\frac1{\varphi(q)}\sumast_{x\bmod q}e\Big(\frac{ax}q\Big)=-\frac1{\varphi(q)}.
\end{align*}
Indeed, the former follows from Lemma~\ref{big-lemma-sums}; the latter is clear from the above expressions, or directly from the expression \eqref{sepsilons} for $s_{p_1}^{+}(1/b,\bm{\epsilon}_Z)=2\mRe\widetilde{s}_{p_1}(a/b;\bm{\epsilon}_Z)$, in which only the prinicipal character $\chi=\chi_0$ contributes $(1/\varphi(q))\overline{G(\chi_0)}$. This completes the proof of \eqref{cor-item2}.
\end{proof}

\begin{remark}
\label{remark1}
The condition \eqref{all-QRs}, which pertains to the case when exactly all but one $\epsilon_p=1$, appears similar in spirit to (and perhaps in some ways weaker than) Artin's primitive root conjecture, so it is perhaps reasonable to expect that it is satisfied for infinitely many primes $q\equiv 3\bmod 4$ in the fixed arithmetic progression described by \eqref{ArP}, which would give an everywhere dense set of corresponding points $t_0=a/q$; we stop short of stating this as a formal conjecture. Unconditionally, using Schmidt's estimates on complete character sums~\cite[Theorem II.2C']{Schmidt1976} and Poisson summation, one can argue that, for every sufficiently large odd prime $q$ and every $2\nmid d\mid(q-1)$,
\[ \sum_{x\bmod q}\Big(\frac xq\Big)e\Big(\frac{ax^d}q\Big)\gg\sqrt{qd} \]
for at least one $a\in I$ in every sufficiently large interval $I$; for example, we were able to prove that this holds in the mean square average over all $a\in I$ for $|I|\geqslant 2\sqrt{d}q^{1/4}$. This alone shows that $\pi|\mIm s^{\ast}(a/q;\bm{\epsilon}_Z)|$ in \eqref{equivalent-condition} is often rather large (that is, of expected size), but a sufficiently good complementary upper bound on $2|\mRe\widetilde{s}_{p_1}(a/q;\bm{\epsilon}_Z)|$ is also needed, and in the case $d_1=\Delta=d=1$ of \eqref{all-QRs} this is guaranteed by reduction to the Ramanujan sum.
\end{remark}

\begin{remark}
\label{remark2}
An interesting situation arises when all $\epsilon_p=1$. Consider the specific case when $Z=5$, $\bm{\epsilon}_5=(\epsilon_2,\epsilon_3,\epsilon_5)=(1,1,1)$. As in other cases, pictures strongly suggest that the path $G^{\sharp}_{\bm{\epsilon}_5}$ has an everywhere dense set of cusps, with the most visually prominent ones at many of the points $a/71$ ($1\leqslant a\leqslant 70$); see Figure~\ref{fig: g71}, which appears to be in parallel with the situation of Figure~\ref{fig: gsharp-eps} (in which $\bm{\epsilon}_5=(1,1,-1)$). That these appear at the denominator $q=71$ is not a coincidence in light of
\[ (p/q)=\epsilon_p\quad (p=2,3,5). \]
Nevertheless, it is not difficult to verify that both 2 and 3 are of multiplicative order 35 modulo 71, so they generate (already each one of them generates) the subgroup $(\mathbb{Z}/71\mathbb{Z})^{\times 2}$, whence in this case
\[ s^{\ast}(a/q;\bm{\epsilon}_Z)=\frac1{\varphi(q)}\sum_{x\bmod q}e\Big(\frac{ax^2}q\Big)=\frac1{\varphi(q)}\Big(-1+\Big(\frac aq\Big)\sqrt{q}i\Big). \]
Thus, $\mRe s^{\ast}(a/q;\bm{\epsilon}_Z)\neq 0$, so already in light of Lemma~\ref{local-d-lemma}, the left and right slopes of $G^{\sharp}_{\bm{\epsilon}_Z}$ agree at $t_0=a/q$, and the path has \emph{no cusp} at any of these points, in apparent contradiction with the very convincing Figure~\ref{fig: g71}. What is going on here?

\begin{figure}[ht]
    \centering
    \includegraphics[width=0.9\textwidth]{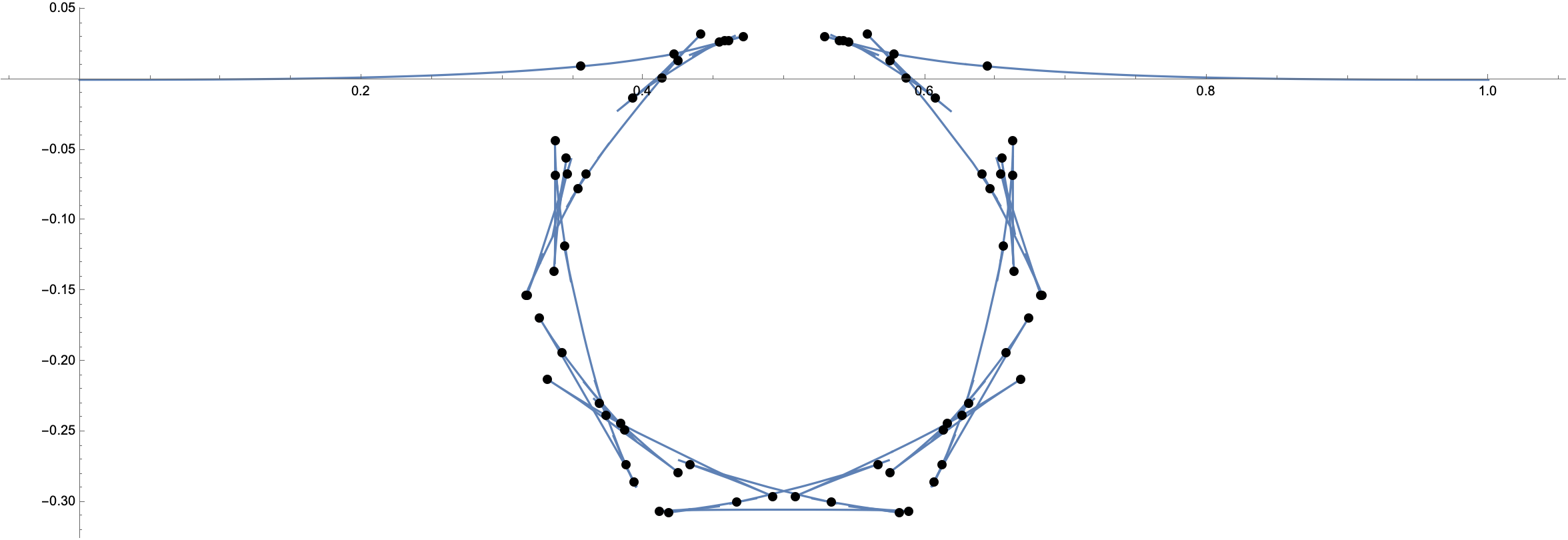}
   \caption{The deterministic path $G^{\sharp}_{\bm{\epsilon}_5}$ for $(\epsilon_2,\epsilon_3,\epsilon_5)=(1,1,1)$, with marked points at $G^{\sharp}_{\bm{\epsilon}_5}(i/71)$, $1\leqslant i\leqslant 71$.}
    \label{fig: g71}
\end{figure}

The answer is that the apparent ``cusps'' are effects of lower-order terms which, in fact, disappear as $t$ gets \emph{really close} to $t_0$. Indeed, according to Proposition~\ref{two-terms-prop} and a quick calculation of constants using Lemma~\ref{constants-lemma} (in which all $\widetilde{s}_p(a/b;\bm{\epsilon}_Z)=0$), the leading two terms in the asymptotic expansion for $G^{\sharp}_{\bm{\epsilon}_Z}(t)-G^{\sharp}_{\bm{\epsilon}_Z}(t_0)$ are given by
\[ \frac{c_{\bm{\epsilon}_5}}{70}\Big(-e(t_0)\cdot 2(t-t_0)\ell^3+e(t_0)\big(-6c^{\prime +}_{\bm{\epsilon}_5}(t-t_0)-\Big(\frac{a}{71}\Big)\cdot 3\pi\sqrt{71} |t-t_0|\big)\ell^2\Big), \]
where $c^{\prime +}_{\bm{\epsilon}_5}=-(\log 2\pi+\gamma-1)+\log(30)/2\approx 0{.}286$ and
\[ -6c^{\prime +}_{\bm{\epsilon}_5}(t-t_0)-\Big(\frac{a}{71}\Big)\cdot 3\pi\sqrt{71} |t-t_0|\approx \Big(\frac{a}{71}\Big)(-79{.}41\pm 1{.}71)|t-t_0| \]
as $t\to t_0\pm$ (or $t\to t_0\mp$, according to the value of $(a/71)$). Thus, while the leading term indeed dominates for extremely small $|t-t_0|$, the secondary term is substantially dominant for $\ell$ of moderate size, say up to $\ell\leqslant 30$, that is up to around $|t-t_0|\leqslant 10^{-15}$, and explains the illusion of cusp behavior. In truth, the path exhibits a transition to a ``smooth'' behavior through $t=t_0$, which can also be observed upon zooming in to the appropriate scale; see Figure~\ref{fig: g71-zoom}.
\end{remark}

\begin{figure}[ht]
    \centering
    \includegraphics[width=0.25\textwidth]{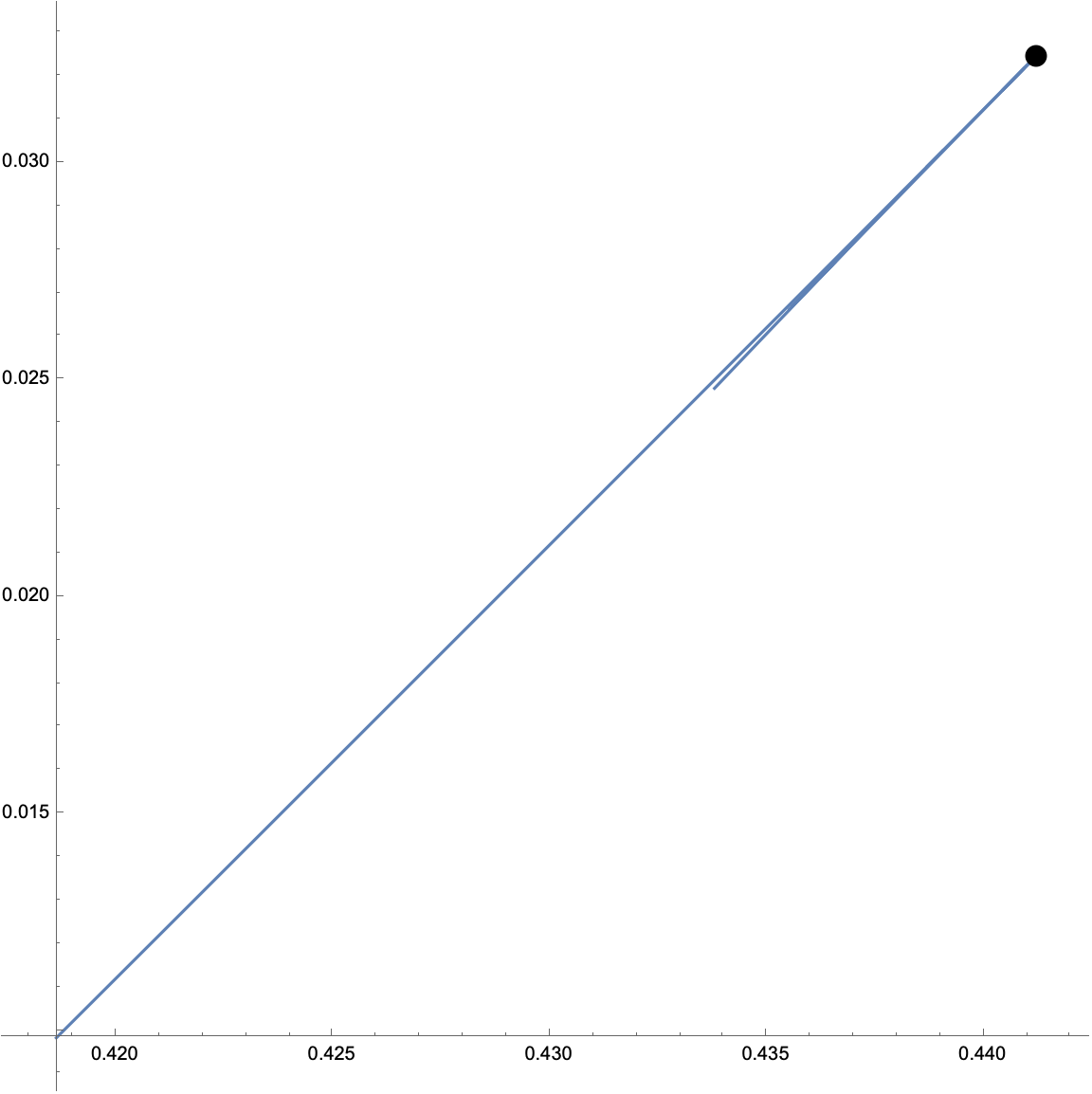}
    \includegraphics[width=0.25\textwidth]{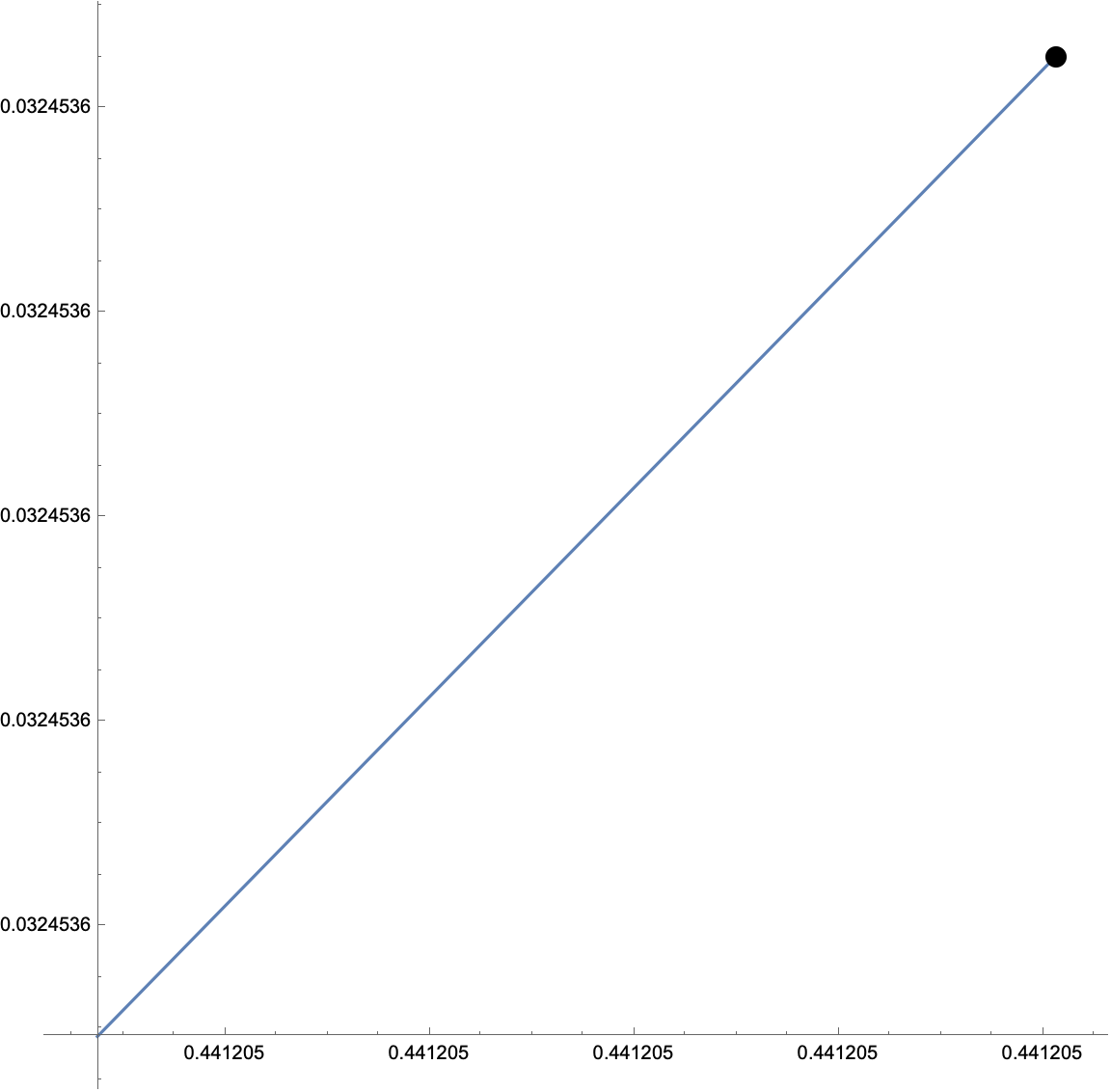}
    \includegraphics[width=0.25\textwidth]{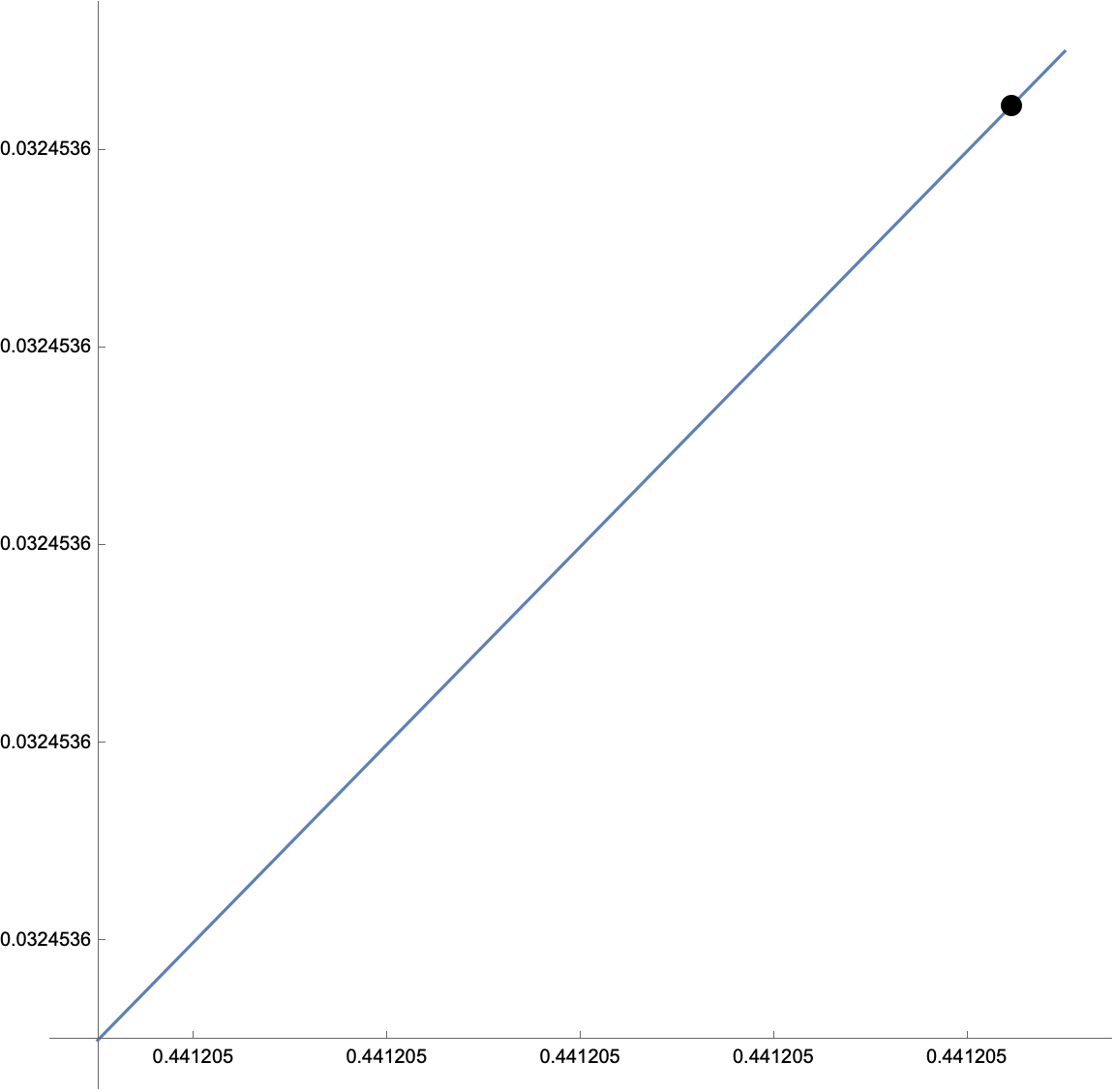}
   \caption{Snapshots of the same deterministic path $G^{\sharp}_{\bm{\epsilon}_Z}$ for $(\epsilon_2,\epsilon_3,\epsilon_5)=(1,1,1)$, computed  with over 22-digit precision, in three progressively zoomed-in neighborhoods of the marked point $t_0=9/71$: $|t-t_0|\leqslant 1/500$, $|t-t_0|\leqslant 1/10^{10}$, and $|t-t_0|\leqslant 1/(5\cdot 10^{17})$.
}
    \label{fig: g71-zoom}
\end{figure}

\section{Classifying quadratic Gauss paths}
\label{final-section}
Having prepared the ground with establishing the properties of the limiting shapes $G^{\sharp}_{\bm{\epsilon}_Z}$ in \S\S\ref{limits-local}--\ref{expsums-nv-mult}, in this section we complete the proof of our main result on the atlas of shapes of Gauss paths, Theorem~\ref{thm2}.

\begin{proof}[Proof of Theorem~\ref{thm2}~\eqref{thm2-gastz}]
The arguments of section~\ref{sec-rv} proceed analogously with the random series $G^{\ast}_{\bm{\epsilon}_Z}(t)$ and its coefficients $(X_{n,\bm{\epsilon}_Z})$ in place of $G^{\ast}(t)$ and $(X_n)$.

Indeed, the sets $A_{y_1}$ and $A_{y_1}^{y_2}$ of \eqref{Ay1-Ay1y2} and along with them the sums $S_{A_{y_1}}$, $S_{A_{y_1}^{y_2}}$, $S_{A_{y_1}}(\Epsilon)$, and $S_{A_{y_1}^{y_2}}(\Epsilon)$ of \eqref{eq: S-A} are insensitive to the values of $(\bm{\epsilon}_Z)$ as soon as $y_1\geqslant Z$; hence the statement and proof of Proposition~\ref{prop: small-sum} remain the same as long as $y_2>y_1\geqslant\max(k^3,Z)$.

Turning to the proof of Proposition~\ref{thm: convergence}, the sums $S_y(t)$ and $S_y(\Epsilon,t)$ of \eqref{Syt} are affected by the change of $(X_n)$ to $(X_{n,\bm{\epsilon}_Z})$, and so are the sums $R_{y_1,y_2}(t)$ and $R_{y_1,y_2}(\Epsilon;t)$ of \eqref{Ry1y2-def}, the dependence in $R_{y_1,y_2}(t)$ in \eqref{Ry1y2-decomp} being that, in the outer sum over $n\neq 0$ with $P^{+}(|n|)\leqslant y_1$, the coefficients $X_n$ are replaced by $X_{n,\bm{\epsilon}_Z}$. Since $|X_{n,\bm{\epsilon}_Z}|\leqslant 1$, the key estimate on $\|R_{y_1,y_2}(\Epsilon;t)\|_{\infty}$ in \eqref{Ry1y2-estimate} remains as stated for $y_1\geqslant Z$, and the rest of the proof of Lemma~\ref{ry1y2-lemma} and Proposition~\ref{thm: convergence} remains exactly the same with the additionally assumption that $y_1\geqslant Z$ and correspondingly replacing $y_0(\delta)$ with $\max(y_0(\delta),Z)$.

This shows that, indeed, the random Fourier series $\widetilde{G^{\ast}_{\bm{\epsilon}_Z}}(t)$ almost surely uniformly converges and defines a continuous function such that the $n$th Fourier coefficient of $\widetilde{G^{\ast}_{\bm{\epsilon}_Z}}(t)-t$ (for $n\neq 0,1$) is precisely $X_{n,\bm{\epsilon}_Z}/2\pi i n$; then, it follows as in \eqref{equality-tilde} that the original random Fourier series $G^{\ast}_{\bm{\epsilon}_Z}(t)$ converges a.s.\ and that $G^{\ast}_{\bm{\epsilon}_Z}(t)=\widetilde{G^{\ast}_{\bm{\epsilon}_Z}}(t)$ a.s.

Turning to the computation of the moments, we may write every $h\in\mathbb{Z}\setminus\{0\}$ as $h=h_Zh^Z$ as in \eqref{b-decomp}; then, in place of \eqref{expectations-eta-def} we have the evaluations
\[ \mathbb{E}(X_{h,\bm{\epsilon}_Z})=\begin{cases} \eta_Z(h),&|h^Z|=\square;\\ 0,&\text{otherwise},\end{cases}\qquad \eta_Z(h)=\epsilon_{h_Z}\eta(h^Z). \]
For every $\bm{t}=(t_1,\dots,t_k)\in[0,1]^k$ and $\bm{m},\bm{n}\in\mathbb{Z}_{\geqslant 0}^k$, the evaluation of the moments $\mathcal{M}^{\ast}_{\bm{\epsilon}_Z}(\bm{t};\bm{m},\bm{n})$ (defined analogously to \eqref{M-ast-moment-def}) of the $\mathbb{C}^k$-valued random variable $G^{\ast}_{\bm{\epsilon}_Z}(\bm{t})=(G^{\ast}_{\bm{\epsilon}_Z}(t_1),\dots,G^{\ast}_{\bm{\epsilon}_Z}(t_k))$ proceeds analogously to the proof of Lemma~\ref{evaluation-of-limit-moment}, with $\eta_Z(h)$ in place of $\eta(h)$. In the evaluation of $\mathbb{E}(|S_{n,\bm{\epsilon}_Z}(t)|^{2p})$ in \eqref{upperbound-sn-p}, all terms with $|(h_1\cdots h_{2p})^Z|=\square$ contribute, and consequently we may estimate
\begin{equation}
\label{E2pZ}
\begin{aligned}
\mathbb{E}(|S_{n,\bm{\epsilon}_Z}(t)|^{2p})
&\leqslant 8^p\sum_{d\mid\prod_{q\leqslant Z}q}\sum_{m=1}^{\infty}\frac{d_{2p}(dm^2)}{dm^2}\\
&\leqslant 8^p\prod_{q\leqslant Z}\Big(1+\frac{2p}{q}\Big)\sum_{m=1}^{\infty}\frac{d_{2p}(m^2)}{m^2}\leqslant e^{\mathrm{O}(p\log\log Z)},
\end{aligned}
\end{equation}
also using \eqref{d2-estimate}. For the same reason, we encounter the dependence of the implied constants on $Z$ in \eqref{Lpnorm-estimate}, \eqref{ESmSn2p-estimate} and below, \eqref{differences} and below, and \eqref{absolute-convergence-eq}. The proof otherwise runs verbatim the same (with $\eta_Z(h)$ and the condition $|h^Z|=\square$ in place of $\eta(h)$ and $|h|=\square$), and we obtain the evaluation
\begin{equation}
\label{M-ast-epsZ}
 \mathcal{M}^{\ast}_{\bm{\epsilon}_Z}(\bm{t};\bm{m},\bm{n})=\sum_{\substack{\vec{\bm{h}}\in\mathcal{H}^{\ast}_{\bm{m},\bm{n}}\\|H(\bm{h})^Z|=\square}}\beta(\vec{\bm{h}};\bm{t})\eta_Z(\bm{h}).
 \end{equation}
In particular, our discussion above shows that $\mathcal{M}^{\ast}_{\bm{\epsilon}_Z}(\bm{t};\bm{m},\bm{n})\leqslant C_Z^{m+n}$ for some $C_Z>0$ depending only on $Z$ (in fact, $C_Z=e^{\mathrm{O}(\log\log Z)}$ is admissible), and so $G^{\ast}_{\bm{\epsilon}_Z}(\bm{t})$ is a mild random variable.

We then proceed to follow section~\ref{computing-moments-section} and demonstrate the analogue of Proposition~\ref{moments-eval-prop}, that the complex moments $\mathcal{M}_{Q,\bm{\epsilon}_Z}(\bm{t};\bm{m},\bm{n})$, defined as in \eqref{MQ-moment-def} with $G_{Q,\bm{\epsilon}_Z}(t_i)$ in place of $G_Q(t_i)$, satisfy
\begin{equation}
\label{moments-match-eq-Z}
\mathcal{M}_{Q,\bm{\epsilon}_Z}(\bm{t};\bm{m},\bm{n})=\mathcal{M}^{\ast}_{\bm{\epsilon}_Z}(\bm{t};\bm{n},\bm{n})+\mathrm{O}_{\epsilon,Z}(Q^{-1/3+\epsilon}).
\end{equation}

Denoting $\lambda'_p(\pm 1)=(1-1/p)/2$ and $\lambda'_p(0)=1/p$ for $p>2$ and $\lambda_2'(\pm 1)=1/2$, we first confirm by sieving as in \eqref{DQ-comput} that
\begin{equation}
\label{DQ-epsZ}
|\mathcal{D}_{Q,\bm{\epsilon}_Z}|=\sum_{\substack{\alpha\in\mathcal{N}^Z\\\alpha^2\leqslant 2Q}}\mu(\alpha)\sum_{\substack{c\in [Q,2Q]\\\alpha^2\mid c,\,c\equiv 1\bmod 4\\(p/c)=\epsilon_p\,(p\leqslant Z)}}1
=\frac{Q}{4\zeta(2)}\lambda'(\bm{\epsilon}_Z)+\mathrm{O}_Z(Q^{1/2}),
\end{equation}
where $\mathcal{N}^Z=\{n\in\mathbb{N}:p\mid n\,\Rightarrow\,p>Z\}$ and
\[ \lambda'(\bm{\epsilon}_Z)=\prod_{p\leqslant Z}\Big(\lambda_p'(\epsilon_p)\Big(1-\frac1{p^2}\Big)^{-1}\Big),\quad m_{Q,\bm{\epsilon}_Z}=\frac{4\zeta(2)}{\lambda'(\bm{\epsilon}_Z)}\frac1Q+\mathrm{O}_Z\Big(\frac1{Q^{3/2}}\Big). \]
Lemmata \ref{lem: g-squiggle-val}--\ref{alpha-beta-sum} are valid for all values of $c$. In the rest of \S\ref{reduction-subsec}, we only need to replace $m_Q(c)$ with $m_Q(\bm{\epsilon}_Z)$ and all sums over $c\in\mathcal{D}_Q$ by $c\in\mathcal{D}_{Q,\bm{\epsilon}_Z}$, in particular when replacing the modified moment $\widetilde{\mathcal{M}_Q}(\bm{t};\bm{m},\bm{n})$ defined in \eqref{modified-moments-equation} with the analogously defined moment $\widetilde{\mathcal{M}_{Q,\bm{\epsilon}_Z}}(\bm{t};\bm{m},\bm{n})$; with these changes, the rest of \S\ref{reduction-subsec} is valid verbatim (with even the implied constants independent of $(\bm{\epsilon}_Z)$). In place of \eqref{moments-main-error-decomp} we obtain
\begin{align*}
\widetilde{\mathcal{M}_{Q,\bm{\epsilon}_Z}^0}(\bm{t};\bm{m},\bm{n})&=\sum_{\substack{\Vec{\bm{h}}\in\mathcal{H}'\\|H(\vec{\bm{h}})^Z|=\square}} \beta(\vec{\bm{h}};\bm{t})\epsilon_{H(\vec{\bm{h}})_Z} \sum_{\substack{c \in \mathcal{D}_{Q,\bm{\epsilon}_Z}\\(c,H(\vec{\bm{h}}))=1}}  m_Q(c),\\
\widetilde{\mathcal{M}_{Q,\bm{\epsilon}_Z}'}(\bm{t}; \bm{m}, \bm{n}) &=   \sum_{\substack{\Vec{\bm{h}}\in\mathcal{H}'\\|H(\vec{\bm{h}})^Z|\neq\square}} \beta(\vec{\bm{h}};\bm{t}) \sum_{c \in\mathcal{D}_{Q,\bm{\epsilon}_Z}}  m_Q(c)\bigg(\frac{H(\Vec{\bm{h}})}{c}\bigg).
\end{align*}

Adjusting the proof of Lemma~\ref{main-term-lemma} as in \eqref{DQ-epsZ}, we find that for $0\neq H=Q^{\OO_Z(1)}$
\begin{align*}
&\sum_{\substack{c\in\mathcal{D}_{Q,\bm{\epsilon}_Z}\\(c,H)=1}}m_{Q,\bm{\epsilon}_Z}(c)=\frac1Q\Big(\frac{4\zeta(2)}{\lambda'(\bm{\epsilon}_Z)}+\OO_Z\Big(\frac1{Q^{1/2}}\Big)\Big)\sum_{\substack{\alpha\in\mathcal{N}^Z\\\alpha^2\leqslant 2Q\\(\alpha,2H^Z)=1}}\mu(\alpha)\sum_{\delta\mid H^Z}\mu(\delta)\sum_{\substack{c\in [Q,2Q]\\\alpha^2\delta\mid c,\,c\equiv 1\bmod 4\\(p/c)=\epsilon_p\,(p\leqslant Z)}}1\\
&=\frac{\zeta(2)}{\lambda'(\bm{\epsilon}_Z)}\prod_p\Big(1-\frac1{p^2}\Big)\prod_{p\mid H^Z}\Big(1+\frac1p\Big)^{-1}\lambda'(\bm{\epsilon}_Z)+\OO_{Z,\epsilon}\Big(\frac{H^{\epsilon}}{Q^{1/2}}\Big)=\eta(H^Z)+\OO_{Z,\epsilon}\Big(\frac1{Q^{1/2-\epsilon}}\Big).
\end{align*}
Estimating tails as in \eqref{tails-estimate} and \eqref{E2pZ}, including
\[ \sum_{\substack{\vec{\bm{h}}\in\mathcal{H}^{\ast}_{\bm{m},\bm{n}}\setminus\mathcal{H}'\\|H(\vec{\bm{h}})^Z|=\square}}|\beta(\vec{\bm{h}};\bm{t})\eta_Z(H(\vec{\bm{h}}))|\ll_{\epsilon,m,n}\sum_{q\mid\prod_{p\leqslant Z}p}\sum_{df^2>Q/2}\mu^2(d)\sum_{df^2\mid k^2q}\frac1{(k^2q)^{1-\epsilon}}\ll_{Z,\epsilon}\frac1{Q^{1/2-\epsilon}}, \]
and keeping in mind the evaluation \eqref{M-ast-epsZ}, we finally conclude
\[ \widetilde{\mathcal{M}^0_{Q,\bm{\epsilon}_Z}}(\bm{t};\bm{m},\bm{n})=\sum_{\substack{\vec{h}\in\mathcal{H}^{\ast}_{\vec{m},\vec{n}}\\|H(\vec{h})^Z|=\square}}\beta(\vec{\bm{h}};\bm{t})\eta_Z(\bm{h})+\OO_{Z,\epsilon}(Q^{-1/2+\epsilon})=\mathcal{M}^{\ast}_{\bm{\epsilon}_Z}(\bm{t};\bm{m},\bm{n})+\OO_{Z,\epsilon}(Q^{-1/2+\epsilon}). \]

We estimate the off-diagonal contributions as in the proof of Lemma~\ref{error-terms-lemma}, with the basic estimate after grouping $\vec{\bm{h}}\in\mathcal{H}'$ according to the values of $|H(\vec{\bm{h}})^Z|\neq\square$ and $H(\vec{\bm{h}})_Z$ with $|H(\vec{\bm{h}})|\leqslant Q^{\ast}$ being
\[ \widetilde{\mathcal{M}'_{Q,\bm{\epsilon}_Z}}(\bm{t};\bm{m},\bm{n})\ll\sum_{\substack{q\in\mathcal{N}_Z^{+}\\\epsilon_p=0\,\Rightarrow\,p\nmid q}}\frac1q\sum_{\substack{k\leqslant Q^{\ast}\\k\in\mathcal{N}^Z}}\frac1{k^2}\sumast_{\substack{1<d\leqslant Q^{\ast}/(qk^2)\\d\in\mathcal{N}^Z}}\frac{\tilde{\tau}_{\mathcal{H}'}(dk^2q)}{d}\bigg|\sum_{\substack{c\in[Q,2Q]\cap\mathcal{D}\\(c,k)=1\\(p/c)=\epsilon_p\,(p\leqslant Z)}}m_Q(c)\left(\frac dc\right)\bigg|. \]
The rest of the argument proceeds completely analogously, additionally restricting the sieving variables to $\alpha,\delta\in\mathcal{N}^Z$, and using the P\'olya--Vinogradov inequality for non-principal characters $(dd_Z/c)$ (for various $d_Z\mid\prod_{p\leqslant Z}p$) of conductor $\asymp_Zd$ to additionally detect congruence conditions $(p/c)=\epsilon_p\neq 0$, and we obtain
\[ \widetilde{\mathcal{M}'_{Q,\bm{\epsilon}_Z}}(\bm{t};\bm{m},\bm{n})\ll_{Z,\epsilon}Q^{-1/3+\epsilon}. \]
Putting everything together completes the proof of \eqref{moments-match-eq-Z} and along with it the convergence of $(G_{Q,\bm{\epsilon}_Z})\to G^{\ast}_{\bm{\epsilon}_Z}$ in the sense of finite distributions as $Q\to\infty$.

Moreover, the sequence of $C^0([0,1],\mathbb{C})$-valued random variables $(G_{Q,\bm{\epsilon}_Z})$ is tight at $Q\to\infty$ by Kolmogorov's Tightness Criterion, because in light of \eqref{DQ-epsZ} we have in the situation of Proposition~\ref{moment-claim} that \emph{a fortiori}
\[ \sum_{c\in\mathcal{D}_{Q,\bm{\epsilon}_Z}}m_{Q,\bm{\epsilon}_Z(c)}|G(t;c)-G(s;c)|^{\alpha}\ll_{\alpha,Z}|t-s|^{1+\delta}. \]
(Here, we profit from the fact that $|\mathcal{D}_{Q,\bm{\epsilon}_Z}|\asymp_Z|\mathcal{D}_Q|$ to execute this bootstrap argument, but, alternatively, it takes only minimal changes to sequentially adapt the arguments of section~\ref{sec-inlaw} to the family $\mathcal{D}_{Q,\bm{\epsilon}_Z}$, with implied constants depending on $Z$.) As in \S\ref{tightness-proof-sec}, using Prokhorov's Criterion we conclude that, indeed, $(G_{Q,\bm{\epsilon}_Z})\to G^{\ast}_{\bm{\epsilon}_Z}$ in law as $Q\to\infty$.
\end{proof}

\begin{proof}[Proof of Theorem~\ref{thm2}~\eqref{thm2-gsharpZ}]
The deterministic Fourier series $G^{\sharp}_{\bm{\epsilon}_Z}(t)$ converges absolutely and uniformly by comparison with the absolutely convergent series $\sum_{n\in\mathcal{N}_{\bm{\epsilon}_Z}}(1/n)\leqslant\prod_{p\in\mathcal{P}_{\bm{\epsilon}_Z}}(1+1/p)$; along with its individual summands, its sum is therefore also a continuous function. The remaining statements about the everywhere dense set of points $t_0\in[0,1]\cap\mathbb{Q}$ at which the path $G^{\sharp}_{\bm{\epsilon}_Z}(t)$ (for $\bm{\epsilon}_Z$ not identically zero) has a cusp follow from Proposition~\ref{main-cor}.
\end{proof}

\begin{proof}[Proof of Theorem~\ref{thm2}~\eqref{exceptional-probabilities}]
We begin by recalling the beginning of the proof of Lemma~\ref{evaluation-of-limit-moment}, where we established that all results of \S\ref{arith-covergence-sec}, including crucially Proposition \ref{prop: small-sum} and Lemma \ref{ry1y2-lemma}, remain valid for
\[ S^{\ast}_{A_{y_1}^{y_2},\bm{1}}=\max_{t\in [0,1]}\bigg|\sum_{n\in A_{y_1}^{y_2}}\frac{e(nt)}{n+1}X_n\bigg|,\quad
S^{\ast}_y(t)=\sum_{\substack{n\neq 0,-1\\P^{+}(|n|)\leqslant y}}\frac{e(nt)-e(-t)}{n+1}X_n, \]
and $R^{\ast}_{y_1,y_2}(t)=S_{y_2}^{\ast}(t)-S^{\ast}_{y_1}(t)$. Further, denote
\[ S^{\ast}_{y,\bm{\epsilon}_Z}:=\sum_{\substack{n\neq 0,-1\\P^{+}(|n|)\leqslant y}}\frac{e(nt)-e(-t)}{n+1}X_{n,\bm{\epsilon}_Z}. \]
As already discussed in the proof of item~\eqref{thm2-gastz}, for $y_1\geqslant \max(y_0(\delta),Z)$, Proposition~\ref{prop: small-sum} is literally unchanged when replacing $X_n$ by $X_{n,\bm{\epsilon}_Z}$ because $X_n=X_{n,\bm{\epsilon}_Z}$ whenever $P^{-}(n)>Z$; the same is true for the statement of Lemma~\ref{ry1y2-lemma}, because the proof only additionally requires that $|X_n|,|X_{n,\bm{\epsilon}_Z}|\leqslant 1$ for $P^{+}(n)\leqslant Z$. Hence Proposition~\ref{thm: convergence} and its proof remain valid for $S^{\ast}_{y,\bm{\epsilon}_Z}$ as well.

As in \eqref{Gtilde-decomp}, we have that
\begin{equation}
\label{G-in-terms-of-S}
G^{\sharp}_{\bm{\epsilon}_Z}(t)=\frac{e(t)}{2\pi i}S^{\ast}_{Z,\bm{\epsilon}_Z}+t,\quad G^{\ast}_{\bm{\epsilon}_Z}(t)=\lim_{y\to\infty}\frac{e(t)}{2\pi i}S^{\ast}_{y,\bm{\epsilon}_Z}+t,
\end{equation}
where we have already verified that the limit converges almost surely. Now, as in \eqref{uniform-over-y2}, for $y_1\geqslant\max(y_0(\delta),Z)$ we have that
\[ \mathbb{P}\Big(\sup_{y_2>y_1}\|S^{\ast}_{y_2,\bm{\epsilon}_Z}-S^{\ast}_{y_1,\bm{\epsilon}_Z}\|_{\infty}>\delta\Big)\ll\exp(-\delta^2y_1^{1/7}). \]
Using this for $y_1=Z\geqslant y_0(\delta)$, we have that outside an event of probability $\ll\exp(-\delta^2y_1^{1/7})$, $\|S^{\ast}_{y_2,\bm{\epsilon}_Z}-S^{\ast}_{Z,\bm{\epsilon}_Z}\|_{\infty}\leqslant\delta$ for all $y_2>Z$; taking limits as $y_2\to\infty$ and invoking \eqref{G-in-terms-of-S} we conclude that
\begin{equation}
\label{explicit-conclusion-1}
\mathbb{P}\big(\|G^{\ast}_{\bm{\epsilon}_Z}-G^{\sharp}_{\bm{\epsilon}_Z}\|_{\infty}>\delta\big)\ll\exp(-\delta^2Z^{1/7}).
\end{equation}

Finally, for any $Z\geqslant y_0(\delta)$, we may consider the bounded continuous function $\varphi:C^0([0,1],\mathbb{C})\to\mathbb{C}$ defined by
\[ \varphi(f)=\min\big(\delta,\|f-G^{\sharp}_{\bm{\epsilon}_Z}\|_{\infty}\big). \]
Since $(G_{Q,\bm{\epsilon}_Z})\to G^{\ast}_{\bm{\epsilon}_Z}$ in law as $Q\to\infty$, for sufficiently large $Q\geqslant Q_0(\delta,\varepsilon,\bm{\epsilon})$ we have that
\[ |\mathbb{E}(\varphi(G_{Q,\bm{\epsilon}_Z}))-\mathbb{E}(\varphi(G^{\ast}_{\bm{\epsilon}_Z}))|<\varepsilon. \]
Since $\|G^{\ast}_{\bm{\epsilon}_Z}-G^{\sharp}_{\bm{\epsilon}_Z}\|_{\infty}\leqslant\delta$ outside an event of probability $\ll\exp(-\delta^2Z^{1/7})$, we have using Chebyshev's inequality that, for all $Q\geqslant Q_0(\delta,\varepsilon,\bm{\epsilon})$,
\[ \mathbb{P}\big(\|G_{Q,\bm{\epsilon}_Z}-G^{\sharp}_{\bm{\epsilon}_Z}\|_{\infty}>\delta\big)\leqslant(1/\delta)\mathbb{E}(\varphi(G_{Q,\bm{\epsilon}_Z}))\ll (1/\delta)\exp(-\delta^2Z^{1/7})+\varepsilon. \]
Therefore
\begin{equation}
\label{explicit-conclusion-2}
\limsup_{Q\to\infty}\mathbb{P}\big(\|G_{Q,\bm{\epsilon}_Z}-G^{\sharp}_{\bm{\epsilon}_Z}\|_{\infty}>\delta/2\big)\ll\exp(-\delta^2Z^{1/8})
\end{equation}
for all sufficiently large $Z\geqslant Z_1(\delta)$. This completes the proof.
\end{proof}

\bibliographystyle{amsalpha}
\bibliography{gauss-references}

\end{document}